%% file: operadebigreffe.tex
\documentclass[11pt,twoside,a4paper]{article}
\usepackage{makeidx}
\usepackage{graphicx}
\usepackage{amscd}
\usepackage{amsmath}
\usepackage{amssymb}
\usepackage{latexsym}
\usepackage[latin1]{inputenc}
\usepackage[T1]{fontenc}
\usepackage[english]{babel}
\usepackage[T1]{fontenc}
\input{xy}
\xyoption{all}

\input{arbres.tex}

\input{arbresbinaires.tex}

\newcommand{\tdelta}{\tilde{\Delta}}
\newcommand{\mmodels}{\mid \hspace{-.2mm} \models}

\title{The bigraft algebras}

\date{}

\author{Anthony Mansuy \\ \\
{\small{\it Laboratoire de Mathématiques, Université de Reims}}\\
\small{{\it Moulin de la Housse - BP 1039 - 51687 REIMS Cedex 2, France}}\\
\small{e-mail : anthony.mansuy@univ-reims.fr}}

\newtheorem{defi}{\indent Definition}
\newtheorem{lemma}[defi]{\indent Lemma}
\newtheorem{cor}[defi]{\indent Corollary}
\newtheorem{theo}[defi]{\indent Theorem}
\newtheorem{prop}[defi]{\indent Proposition}

\newenvironment{proof}{{\bf Proof.}}{\hfill $\Box$}
\usepackage{fancyhdr}
\begin{document}
\maketitle

\input{resume.tex}

\tableofcontents

\input{introduction.tex}
\input{algebredegreffegetd.tex}

\input{algebrebigreffelibre.tex}

\input{algebrebigreffe.tex}

\input{dualite.tex}

\input{theoderigidite.tex}

\end{document}

%% file: arbres.tex
\newcommand{\tun}{\begin{picture}(5,0)(-2,-1)
\put(0,0){\circle*{2}}
\end{picture}}

\newcommand{\tdeux}{\begin{picture}(7,7)(0,-1)
\put(3,0){\circle*{2}}
\put(3,0){\line(0,1){5}}
\put(3,5){\circle*{2}}
\end{picture}}

\newcommand{\ttroisun}{\begin{picture}(15,8)(-5,-1)
\put(3,0){\circle*{2}}
\put(-0.65,0){$\vee$}
\put(6,7){\circle*{2}}
\put(0,7){\circle*{2}}
\end{picture}}
\newcommand{\ttroisdeux}{\begin{picture}(5,12)(-2,-1)
\put(0,0){\circle*{2}}
\put(0,0){\line(0,1){5}}
\put(0,5){\circle*{2}}
\put(0,5){\line(0,1){5}}
\put(0,10){\circle*{2}}
\end{picture}}

\newcommand{\tquatreun}{\begin{picture}(15,12)(-5,-1)
\put(3,0){\circle*{2}}
\put(-0.65,0){$\vee$}
\put(6,7){\circle*{2}}
\put(0,7){\circle*{2}}
\put(3,7){\circle*{2}}
\put(3,0){\line(0,1){7}}
\end{picture}}
\newcommand{\tquatredeux}{\begin{picture}(15,18)(-5,-1)
\put(3,0){\circle*{2}}
\put(-0.65,0){$\vee$}
\put(6,7){\circle*{2}}
\put(0,7){\circle*{2}}
\put(0,14){\circle*{2}}
\put(0,7){\line(0,1){7}}
\end{picture}}
\newcommand{\tquatretrois}{\begin{picture}(15,18)(-5,-1)
\put(3,0){\circle*{2}}
\put(-0.65,0){$\vee$}
\put(6,7){\circle*{2}}
\put(0,7){\circle*{2}}
\put(6,14){\circle*{2}}
\put(6,7){\line(0,1){7}}
\end{picture}}
\newcommand{\tquatrequatre}{\begin{picture}(15,18)(-5,-1)
\put(3,5){\circle*{2}}
\put(-0.65,5){$\vee$}
\put(6,12){\circle*{2}}
\put(0,12){\circle*{2}}
\put(3,0){\circle*{2}}
\put(3,0){\line(0,1){5}}
\end{picture}}
\newcommand{\tquatrecinq}{\begin{picture}(9,19)(-2,-1)
\put(0,0){\circle*{2}}
\put(0,0){\line(0,1){5}}
\put(0,5){\circle*{2}}
\put(0,5){\line(0,1){5}}
\put(0,10){\circle*{2}}
\put(0,10){\line(0,1){5}}
\put(0,15){\circle*{2}}
\end{picture}}

\newcommand{\tcinqun}{\begin{picture}(20,8)(-5,-1)
\put(3,0){\circle*{2}}
\put(-0.5,0){$\vee$}
\put(6,7){\circle*{2}}
\put(0,7){\circle*{2}}
\put(3,0){\line(2,1){10}}
\put(3,0){\line(-2,1){10}}
\put(-7,5){\circle*{2}}
\put(13,5){\circle*{2}}
\end{picture}}
\newcommand{\tcinqdeux}{\begin{picture}(15,14)(-5,-1)
\put(3,0){\circle*{2}}
\put(-0.65,0){$\vee$}
\put(6,7){\circle*{2}}
\put(0,7){\circle*{2}}
\put(3,7){\circle*{2}}
\put(3,0){\line(0,1){7}}
\put(0,7){\line(0,1){7}}
\put(0,14){\circle*{2}}
\end{picture}}
\newcommand{\tcinqtrois}{\begin{picture}(15,15)(-5,-1)
\put(3,0){\circle*{2}}
\put(-0.65,0){$\vee$}
\put(6,7){\circle*{2}}
\put(0,7){\circle*{2}}
\put(3,7){\circle*{2}}
\put(3,0){\line(0,1){7}}
\put(3,7){\line(0,1){7}}
\put(3,14){\circle*{2}}
\end{picture}}
\newcommand{\tcinqquatre}{\begin{picture}(15,14)(-5,-1)
\put(3,0){\circle*{2}}
\put(-0.65,0){$\vee$}
\put(6,7){\circle*{2}}
\put(0,7){\circle*{2}}
\put(3,7){\circle*{2}}
\put(3,0){\line(0,1){7}}
\put(6,7){\line(0,1){7}}
\put(6,14){\circle*{2}}
\end{picture}}
\newcommand{\tcinqcinq}{\begin{picture}(15,19)(-5,-1)
\put(3,0){\circle*{2}}
\put(-0.65,0){$\vee$}
\put(6,7){\circle*{2}}
\put(0,7){\circle*{2}}
\put(6,14){\circle*{2}}
\put(6,7){\line(0,1){7}}
\put(0,14){\circle*{2}}
\put(0,7){\line(0,1){7}}
\end{picture}}
\newcommand{\tcinqsix}{\begin{picture}(15,20)(-7,-1)
\put(3,0){\circle*{2}}
\put(-0.65,0){$\vee$}
\put(6,7){\circle*{2}}
\put(0,7){\circle*{2}}
\put(-3.65,7){$\vee$}
\put(3,14){\circle*{2}}
\put(-3,14){\circle*{2}}
\end{picture}}
\newcommand{\tcinqsept}{\begin{picture}(15,8)(-5,-1)
\put(3,0){\circle*{2}}
\put(-0.65,0){$\vee$}
\put(6,7){\circle*{2}}
\put(0,7){\circle*{2}}
\put(2.35,7){$\vee$}
\put(3,14){\circle*{2}}
\put(9,14){\circle*{2}}
\end{picture}}
\newcommand{\tcinqhuit}{\begin{picture}(15,26)(-5,-1)
\put(3,0){\circle*{2}}
\put(-0.65,0){$\vee$}
\put(6,7){\circle*{2}}
\put(0,7){\circle*{2}}
\put(0,14){\circle*{2}}
\put(0,7){\line(0,1){7}}
\put(0,21){\circle*{2}}
\put(0,14){\line(0,1){7}}
\end{picture}}
\newcommand{\tcinqneuf}{\begin{picture}(15,26)(-5,-1)
\put(3,0){\circle*{2}}
\put(-0.65,0){$\vee$}
\put(6,7){\circle*{2}}
\put(0,7){\circle*{2}}
\put(6,14){\circle*{2}}
\put(6,7){\line(0,1){7}}
\put(6,21){\circle*{2}}
\put(6,14){\line(0,1){7}}
\end{picture}}
\newcommand{\tcinqdix}{\begin{picture}(15,19)(-5,-1)
\put(3,5){\circle*{2}}
\put(-0.5,5){$\vee$}
\put(6,12){\circle*{2}}
\put(0,12){\circle*{2}}
\put(3,0){\circle*{2}}
\put(3,0){\line(0,1){12}}
\put(3,12){\circle*{2}}
\end{picture}}
\newcommand{\tcinqonze}{\begin{picture}(15,26)(-5,-1)
\put(3,5){\circle*{2}}
\put(-0.65,5){$\vee$}
\put(6,12){\circle*{2}}
\put(0,12){\circle*{2}}
\put(3,0){\circle*{2}}
\put(3,0){\line(0,1){5}}
\put(0,12){\line(0,1){7}}
\put(0,19){\circle*{2}}
\end{picture}}
\newcommand{\tcinqdouze}{\begin{picture}(15,26)(-5,-1)
\put(3,5){\circle*{2}}
\put(-0.65,5){$\vee$}
\put(6,12){\circle*{2}}
\put(0,12){\circle*{2}}
\put(3,0){\circle*{2}}
\put(3,0){\line(0,1){5}}
\put(6,12){\line(0,1){7}}
\put(6,19){\circle*{2}}
\end{picture}}
\newcommand{\tcinqtreize}{\begin{picture}(5,26)(-2,-1)
\put(0,0){\circle*{2}}
\put(0,0){\line(0,1){7}}
\put(0,7){\circle*{2}}
\put(0,7){\line(0,1){7}}
\put(0,14){\circle*{2}}
\put(-3.65,14){$\vee$}
\put(-3,21){\circle*{2}}
\put(3,21){\circle*{2}}
\end{picture}}
\newcommand{\tcinqquatorze}{\begin{picture}(9,26)(-5,-1)
\put(0,0){\circle*{2}}
\put(0,0){\line(0,1){5}}
\put(0,5){\circle*{2}}
\put(0,5){\line(0,1){5}}
\put(0,10){\circle*{2}}
\put(0,10){\line(0,1){5}}
\put(0,15){\circle*{2}}
\put(0,15){\line(0,1){5}}
\put(0,20){\circle*{2}}
\end{picture}}

\newcommand{\tdun}[1]{\begin{picture}(10,5)(-2,-1)
\put(0,0){\circle*{2}}
\put(3,-2){\tiny #1}
\end{picture}}

\newcommand{\addeux}[1]{\begin{picture}(12,5)(0,-1)
\put(3,0){\circle*{2}}
\put(3,0){\line(0,1){5}}
\put(3,5){\circle*{2}}
\put(4,1){\tiny #1}
\end{picture}}

\newcommand{\adtroisun}[2]{\begin{picture}(20,12)(-5,-1)
\put(3,0){\circle*{2}}
\put(-0.65,0){$\vee$}
\put(6,7){\circle*{2}}
\put(0,7){\circle*{2}}
\put(6,1){\tiny #1}
\put(-3,1){\tiny #2}
\end{picture}}

\newcommand{\adtroisdeux}[2]{\begin{picture}(12,14)(-2,-1)
\put(0,0){\circle*{2}}
\put(0,0){\line(0,1){5}}
\put(0,5){\circle*{2}}
\put(0,5){\line(0,1){5}}
\put(0,10){\circle*{2}}
\put(2,1){\tiny #1}
\put(2,6){\tiny #2}
\end{picture}}

\newcommand{\adquatreun}[3]{\begin{picture}(25,12)(-5,-1)
\put(5,0){\circle*{2}}
\put(15,10){\circle*{2}}
\put(-5,10){\circle*{2}}
\put(5,10){\circle*{2}}
\put(5,0){\line(0,1){10}}
\put(5,0){\line(-1,1){10}}
\put(5,0){\line(1,1){10}}
\put(-5,1){\tiny #1}
\put(1,4){\tiny #2}
\put(11,1){\tiny #3}
\end{picture}}

\newcommand{\adquatredeux}[3]{\begin{picture}(20,20)(-5,-1)
\put(3,0){\circle*{2}}
\put(-.65,0){$\vee$}
\put(6,7){\circle*{2}}
\put(0,7){\circle*{2}}
\put(0,14){\circle*{2}}
\put(0,7){\line(0,1){7}}
\put(6,1){\tiny #1}
\put(-3,1){\tiny #2}
\put(-5,9){\tiny #3}
\end{picture}}

\newcommand{\adquatretrois}[3]{\begin{picture}(20,20)(-5,-1)
\put(3,0){\circle*{2}}
\put(-.65,0){$\vee$}
\put(6,7){\circle*{2}}
\put(0,7){\circle*{2}}
\put(6,14){\circle*{2}}
\put(6,7){\line(0,1){7}}
\put(6,1){\tiny #1}
\put(-3,1){\tiny #2}
\put(7,9){\tiny #3}
\end{picture}}

\newcommand{\adquatrequatre}[3]{\begin{picture}(20,14)(-5,-1)
\put(3,5){\circle*{2}}
\put(-.65,5){$\vee$}
\put(6,12){\circle*{2}}
\put(0,12){\circle*{2}}
\put(3,0){\circle*{2}}
\put(3,0){\line(0,1){5}}
\put(4,1){\tiny #1}
\put(-3,6){\tiny #2}
\put(5,6){\tiny #3}
\end{picture}}

\newcommand{\adquatrecinq}[3]{\begin{picture}(12,19)(-2,-1)
\put(0,0){\circle*{2}}
\put(0,0){\line(0,1){5}}
\put(0,5){\circle*{2}}
\put(0,5){\line(0,1){5}}
\put(0,10){\circle*{2}}
\put(0,10){\line(0,1){5}}
\put(0,15){\circle*{2}}
\put(2,1){\tiny #1}
\put(2,6){\tiny #2}
\put(2,11){\tiny #3}
\end{picture}}

\newcommand{\saddeux}[3]{\begin{picture}(12,5)(0,-1)
\put(3,0){\circle*{2}}
\put(3,0){\line(0,1){10}}
\put(3,10){\circle*{2}}
\put(4,3){\tiny #1}
\put(4,-5){\tiny #2}
\put(4,12){\tiny #3}
\end{picture}}

\newcommand{\sadtroisun}[5]{\begin{picture}(20,12)(-5,-1)
\put(3,0){\circle*{2}}
\put(3,0){\line(-1,2){5}}
\put(3,0){\line(1,2){5}}
\put(-2,10){\circle*{2}}
\put(8,10){\circle*{2}}
\put(7,3){\tiny #1}
\put(-4,3){\tiny #2}
\put(5,-5){\tiny #3}
\put(9,12){\tiny #4}
\put(-6,12){\tiny #5}
\end{picture}}

\newcommand{\sadquatreun}[7]{\begin{picture}(25,12)(-5,-1)
\put(5,0){\circle*{2}}
\put(15,10){\circle*{2}}
\put(-5,10){\circle*{2}}
\put(5,10){\circle*{2}}
\put(5,0){\line(0,1){10}}
\put(5,0){\line(-1,1){10}}
\put(5,0){\line(1,1){10}}
\put(-5,1){\tiny #1}
\put(1,4){\tiny #2}
\put(11,1){\tiny #3}
\put(5,-5){\tiny #4}
\put(14,12){\tiny #5}
\put(2,12){\tiny #6}
\put(-9,12){\tiny #7}
\end{picture}}

%% file: arbresbinaires.tex
\newcommand{\bddtroisun}[5]{
\begin{picture}(30,40)(-20,0)
\put(0,0){\line(0,0){10}}
\put(0,10){\line(1,1){10}}
\put(0,10){\line(-1,1){20}}
\put(-10,20){\line(1,1){10}}
\put(-9,7){\tiny #1}
\put(-19,17){\tiny #2}
\put(-21,32){\tiny #3}
\put(-1,32){\tiny #4}
\put(9,22){\tiny #5}
\end{picture}}
\newcommand{\bddtroisdeux}[5]{
\begin{picture}(30,40)(-10,0)
\put(0,0){\line(0,0){10}}
\put(0,10){\line(1,1){20}}
\put(0,10){\line(-1,1){10}}
\put(10,20){\line(-1,1){10}}
\put(3,7){\tiny #1}
\put(13,17){\tiny #2}
\put(-11,22){\tiny #3}
\put(-1,32){\tiny #4}
\put(19,32){\tiny #5}
\end{picture}}

\newcommand{\bdgg}[3]{
\begin{picture}(40,50)(-20,0)
\put(0,0){\line(0,0){10}}
\put(0,10){\line(1,1){10}}
\put(0,10){\line(-1,1){30}}
\put(-10,20){\line(1,1){10}}
\put(-20,30){\line(1,1){10}}
\put(-9,7){\tiny #1}
\put(-19,17){\tiny #2}
\put(-29,27){\tiny #3}
\end{picture}}

\newcommand{\bdgd}[3]{
\begin{picture}(40,50)(-20,0)
\put(0,0){\line(0,0){10}}
\put(0,10){\line(1,1){10}}
\put(0,10){\line(-1,1){20}}
\put(-10,20){\line(1,1){20}}
\put(0,30){\line(-1,1){10}}
\put(-9,7){\tiny #1}
\put(-19,17){\tiny #2}
\put(3,27){\tiny #3}
\end{picture}}

\newcommand{\bddg}[3]{
\begin{picture}(40,50)(-20,0)
\put(0,0){\line(0,0){10}}
\put(0,10){\line(1,1){20}}
\put(0,10){\line(-1,1){10}}
\put(10,20){\line(-1,1){20}}
\put(0,30){\line(1,1){10}}
\put(3,7){\tiny #1}
\put(13,17){\tiny #2}
\put(-9,27){\tiny #3}
\end{picture}}

\newcommand{\bddd}[3]{
\begin{picture}(40,50)(-20,0)
\put(0,0){\line(0,0){10}}
\put(0,10){\line(1,1){30}}
\put(0,10){\line(-1,1){10}}
\put(10,20){\line(-1,1){10}}
\put(20,30){\line(-1,1){10}}
\put(3,7){\tiny #1}
\put(13,17){\tiny #2}
\put(23,27){\tiny #3}
\end{picture}}

\newcommand{\bdm}[3]{
\begin{picture}(40,50)(-20,0)
\put(0,0){\line(0,0){10}}
\put(0,10){\line(1,1){23}}
\put(0,10){\line(-1,1){23}}
\put(-13,23){\line(1,1){10}}
\put(13,23){\line(-1,1){10}}
\put(3,7){\tiny #1}
\put(-19,17){\tiny #2}
\put(13,17){\tiny #3}
\end{picture}}

%% file: resume.tex
\textbf{Abstract.} In this paper, we introduce the notion of bigraft algebra, generalizing the notions of left and right graft algebras. We give a combinatorial description of the free bigraft algebra generated by one generator and we endow this algebra with a Hopf algebra structure, and a pairing. Next, we study the Koszul dual of the bigraft operad and we give a combinatorial description of the free dual bigraft algebra generated by one generator. With the help of a rewriting method, we prove that the bigraft operad is Koszul. Finally, we define the notion of infinitesimal bigraft bialgebra and we prove a rigidity theorem for connected infinitesimal bigraft bialgebras. \\

\textbf{Résumé.} Dans ce papier, nous introduisons la notion d'algèbre bigreffe, généralisant les notions d'algèbres de greffes à gauche et à droite. Nous donnons une description combinatoire de l'algèbre bigreffe libre engendrée par un générateur et nous munissons cette algèbre d'une structure d'algèbre de Hopf et d'un couplage. Nous étudions ensuite le dual de Koszul de l'operade bigreffe et nous donnons une description combinatoire de l'algèbre bigreffe dual engendrée par un générateur. A l'aide d'une méthode de réécriture, nous prouvons que l'opérade bigreffe est Koszul. Enfin, nous définissons la notion de bialgèbre bigreffe infinitésimale et nous prouvons un théorème de rigidité pour les bialgèbres bigreffe infinitésimales connexes. \\



%% file: introduction.tex
\section*{Introduction}

The Connes-Kreimer Hopf algebra of rooted trees is introduced and studied in \cite{Connes,Moerdijk}. This commutative, non cocommutative Hopf algebra is used to study a problem of Renormalisation
in Quantum Fields Theory, as explained in \cite{Connes1,Connes2}. A non commutative version, the Hopf algebra of planar rooted trees, is introduced in \cite{Foissyarbre,Holtkamp}. With a total order on the vertices, we obtain the Hopf algebra of ordered trees.

In \cite{Menous}, F. Menous study some sets of probabilities, called induced averages by J. Ecalle, associated to a random walk on $ \mathbb{R} $. For this, he constructs, with a formalism similar to mould calculus, a set of ordered trees using two grafting operators. In \cite{Mansuy}, the author study more generally these two grafting operators and constructs a Hopf subalgebra $ \mathcal{B} $ of the Hopf algebra of ordered trees. He equips $ \mathcal{B} $ with two operations $ \succ $ and $ \prec $ (respectively the left graft and the right graft) and, after having defined the concept of bigraft algebra, he proves that $ \mathcal{B} $ is a bigraft algebra.\\

The aim of this text is an algebraic study of the bigraft algebras introduced in \cite{Mansuy}. A bigraft algebra ($ \mathcal{BG} $-algebra for short) is an associative algebra $ (A,\ast) $ equipped with a left graft $ \succ $ and a right graft $ \prec $ such that $ (A,\ast, \succ) $ is a left graft algebra, $ (A,\ast,\prec) $ is a right graft algebra and with the entanglement relation:
\begin{eqnarray} \label{entangl}
(x \succ y) \prec z = x \succ (y \prec z) .
\end{eqnarray}
We construct a decorated version of the Hopf algebra of planar rooted trees denoted $ \mathcal{BT} $ (we decorate the edges with two possible decorations). We equip $ \mathcal{BT} $ with a Hopf algebra structure and a pairing. We prove that the augmentation ideal $ \mathcal{M} $ of $ \mathcal{BT} $ is the free $ \mathcal{BG} $-algebra on one generator and we deduce the dimension of the bigraft operad $ \mathcal{BG} $ and a combinatorial description of the composition.

Afterward, we give a presentation of the Koszul dual $ \mathcal{BG}^{!} $ of the bigraft operad $ \mathcal{BG} $, and a combinatorial description of the free $ \mathcal{BG}^{!} $-algebra on one generator. We also give a presentation of the homology of a $ \mathcal{BG} $-algebra. With the help of a rewriting method (see \cite{Dotsenko,Hoffbeck,LodayV}), we prove that $ \mathcal{BG} $ is Koszul and we give the PBW basis of the bigraft operad and its Koszul dual.

We finally study the compatibilities of products $ \ast $, $ \succ $, $ \prec $ and the coproduct $ \tdelta_{\mathcal{A}ss} $ of deconcatenation. This leads to the definition of infinitesimal bigraft bialgebras. We prove that the subspace corresponding to primitive elements of $ \mathcal{M} $ is a $ \mathcal{L} $-algebra, that is to say a $ \mathbb{K} $-vector space with two operations $ \succ $ and $ \prec $ satisfying the entanglement relation (\ref{entangl}) (see \cite{Leroux2,Leroux} for more details on $ \mathcal{L} $-algebras). We deduce a combinatorial description of the free $ \mathcal{L} $-algebra generated by one generator. Finally, we prove that $ (\mathcal{A}ss,\mathcal{BG}, \mathcal{L}) $ is a good triple of operads (see \cite{Loday2} for the notion of triple of operads).
\\

This text is organised as follows: in the first section, we recall several facts on the Hopf algebra of planar trees and on right graft algebras. The Hopf algebra of planar decorated trees $ \mathcal{BT} $ and its pairing are introduced in section 2. In section 3, we define the notion of $ \mathcal{BG} $-algebra and we give a combinatorial description of the free $ \mathcal{BG} $-algebra on one generator as the augmentation ideal $ \mathcal{M} $ of the Hopf algebra $ \mathcal{BT} $. Section 4 is devoted to the study of the Koszul dual of the bigraft operad. In particular, we prove that the bigraft operad is Koszul with the rewriting method. The last section deals with the notion of the infinitesimal bigraft bialgebras and we prove that $ (\mathcal{A}ss,\mathcal{BG}, \mathcal{L}) $ is a good triple of operads.\\

{\bf Acknowledgment.} {I am grateful to my advisor Loïc Foissy for stimulating discussions and his support during my research. I would also like to thank Eric Hoffbeck for helpful explanations on the rewriting method and Mohamed Khodja for his careful rereading.}
\\

{\bf Notations.} {
\begin{enumerate}
\item We shall denote by $ \mathbb{K} $ a commutative field, of any characteristic. Every vector space, algebra, coalgebra, etc, will be taken over $ \mathbb{K} $. Given a set $ X $, we denote by $ \mathbb{K} [X] $ the vector space spanned by $ X $.
\item Let $ V = \displaystyle\bigoplus_{n = 0}^{\infty} V_{n} $ be a graded vector space. We denote by $ V^{\circledast} = \displaystyle\bigoplus_{n = 0}^{\infty} V_{n}^{\ast} $ the graded dual of $ V $. If $ H $ is a graded Hopf algebra, $ H^{\circledast} $ is also a graded Hopf algebra.
\item Let $ V $ and $ W $ be two $ \mathbb{K} $-vector spaces. We note $ \tau : V \otimes W \rightarrow W \otimes V $ the unique $ \mathbb{K} $-linear map called \textit{the flip} such that $ \tau (v \otimes w) = w \otimes v $ for all $ v \otimes w \in V \otimes W $.
\item Let $ (A, \Delta , \varepsilon) $ be a counitary coalgebra. Let $ 1 \in A $, non zero, such that $ \Delta (1) = 1 \otimes 1 $. We then define the noncounitary coproduct:
\begin{eqnarray*}
\tdelta : \left\lbrace 
\begin{array}{rcl}
Ker(\varepsilon) & \rightarrow & Ker(\varepsilon) \otimes Ker(\varepsilon) ,\\
a & \mapsto & \Delta (a ) - a \otimes 1 - 1 \otimes a .
\end{array} \right. 
\end{eqnarray*}
\end{enumerate}
}

%% file: algebredegreffegetd.tex
\section{Left graft algebra and right graft algebra}

\subsection{Presentation}

\begin{defi}
\begin{enumerate}
\item A left graft algebra (or $ \mathcal{LG} $-algebra for short) is a $ \mathbb{K} $-vector space $ A $ together with two $ \mathbb{K} $-linear maps $ \ast , \succ : A \otimes A \rightarrow A $ respectively called product and left graft, satisfying the following relations : for all $ x,y,z \in A $,
\begin{eqnarray*}
(x \ast y) \ast z & = & x \ast (y \ast z),\\
(x \ast y) \succ z & = & x \succ (y \succ z),\\
(x \succ y) \ast z & = & x \succ (y \ast z).
\end{eqnarray*}
\item A right graft algebra (or $ \mathcal{RG} $-algebra for short) is a $ \mathbb{K} $-vector space $ A $ together with two $ \mathbb{K} $-linear maps $ \ast , \prec : A \otimes A \rightarrow A $ respectively called product and right graft, satisfying the following relations : for all $ x,y,z \in A $,
\begin{eqnarray*}
(x \ast y) \ast z & = & x \ast (y \ast z),\\
(x \prec y) \prec z & = & x \prec (y \ast z),\\
(x \ast y) \prec z & = & x \ast (y \prec z).
\end{eqnarray*}
\end{enumerate}
\end{defi}

From this definition, it is clear that the operads $ \mathcal{LG} $ and $ \mathcal{RG} $ are binary, quadratic, regular and set-theoretic (see \cite{LodayV} for a definition). We do not suppose that $ \mathcal{LG} $-algebras and $ \mathcal{RG} $-algebras have units for the product $ \ast $. If $ A $ and $ B $ are two $ \mathcal{LG} $-algebras, we say that a $ \mathbb{K} $-linear map $ f : A \rightarrow B $ is a $ \mathcal{LG} $-morphism if $ f(x \ast y) = f(x) \ast f(y) $ and $ f(x \succ y) = f(x) \succ f(y) $ for all $ x,y \in A $. We define in the same way the notion of $ \mathcal{RG} $-morphism. We denote by $ \mathcal{LG} $-alg the category of $ \mathcal{LG} $-algebras and $ \mathcal{RG} $-alg the category of $ \mathcal{RG} $-algebras\\

{\bf Remark.} {The category $ \mathcal{LG} $-alg is equivalent to the category $ \mathcal{RG} $-alg: let $ (A,\ast , \succ ) $ be a $ \mathcal{LG} $-algebra, then $ (A,\ast^{\dagger},\succ^{\dagger}) $ is a $ \mathcal{RG} $-algebra, where $ x \ast^{\dagger} y = y \ast x $ and $ x \succ^{\dagger} y = y \succ x $ for all $ x,y \in A $. Note that $ \ast^{\dagger \dagger} = \ast $ and $ \succ^{\dagger \dagger} = \succ $. So we will only study the operad $ \mathcal{RG} $.}

\subsection{Hopf algebra of planar trees}

\subsubsection{The Connes-Kreimer Hopf algebra of rooted trees}

We briefly recall the construction of the Connes-Kreimer Hopf algebra of rooted trees \cite{Connes}. A {\it rooted tree} is a finite graph, connected, without loops, with a distinguished vertex called the {\it root} \cite{Stanley}.  A {\it rooted forest} is a finite graph $ F $ such that any connected component of $ F $ is a rooted tree. The \textit{length} of a forest $ F $ is the number of connected components of $ F $. The set of vertices of the rooted forest $ F $ is denoted by $V(F)$. The {\it degree} of a forest $ F $ is the number of its vertices.\\

{\bf Examples.} {\begin{enumerate}
\item Rooted trees of degree $ \leq 5 $:
$$ 1,\tun,\tdeux,\ttroisun,\ttroisdeux,\tquatreun,\tquatredeux,\tquatrequatre,\tquatrecinq,\tcinqun,\tcinqdeux,\tcinqcinq,\tcinqsix,\tcinqhuit,\tcinqdix,\tcinqonze,\tcinqtreize,\tcinqquatorze . $$
\item Rooted forests of degree $ \leq 4 $:
$$ 1,\tun,\tun\tun,\tdeux,\tun\tun\tun,\tdeux\tun,\ttroisun,\ttroisdeux,\tun\tun\tun\tun,\tdeux\tun\tun,\tdeux\tdeux,\ttroisun\tun,\ttroisdeux\tun,\tquatreun,\tquatredeux,\tquatrequatre,\tquatrecinq . $$
\end{enumerate}}

Let $ F $ be a rooted forest. The edges of $ F $ are oriented downwards (from the leaves to the roots). If $v,w \in V(F)$, we shall note $v \rightarrow w$ if there is an edge in $ F $ from $v$ to $w$ and $v \twoheadrightarrow w$ if there is an oriented path from $v$ to $w$ in $ F $. By convention, $v \twoheadrightarrow v$ for any $v \in V(F)$. If $ T $ is a rooted tree and if $ v \in V(T) $, we denote by $ h(v) $ the \textit{height} of $ v $, that is to say the number of edges on the oriented path from $ v $ to the root of $ T $. The \textit{height} of a rooted forest $ F $ is $ h(F) = \mbox{max} \left( \{ h(v), \: v \in V(F) \} \right) $. We shall say that a tree $ T $ is a corolla if $ h(T) \leq 1 $.

Let $\boldsymbol{v}$ be a subset of $V(F)$. We shall say that $\boldsymbol{v}$ is an admissible cut of $F$, and we shall write $\boldsymbol{v} \models V(F)$, if $\boldsymbol{v}$ is totally disconnected, that is to say that $v \twoheadrightarrow w \hspace{-.7cm} / \hspace{.7cm}$ for any couple $(v,w)$ of two different elements of $\boldsymbol{v}$. If $\boldsymbol{v} \models V(F)$, we denote by $Lea_{\boldsymbol{v}} (F) $ the rooted subforest of $ F $ obtained by keeping only the vertices above $\boldsymbol{v}$, that is to say $\{ w \in V(F), \: \exists v \in \boldsymbol{v}, \:w \twoheadrightarrow v \}$. Note that $\boldsymbol{v} \subseteq Lea_{\boldsymbol{v}}(F) $. We denote by $Roo_{\boldsymbol{v}}(F)$ the rooted subforest obtained by keeping the other vertices.

In particular, if $\boldsymbol{v}=\emptyset$, then $Lea_{\boldsymbol{v}} (F) =1$ and $Roo_{\boldsymbol{v}} (F) = F $: this is the {\it empty cut} of $ F $. If $\boldsymbol{v}$ contains all the roots of $ F $, then it contains only the roots of $ F $, $Lea_{\boldsymbol{v}} (F) = F$ and $Roo_{\boldsymbol{v}}(F) = 1$: this is the {\it total cut} of $ F $. We shall write $\boldsymbol{v} \mmodels V( F )$ if $\boldsymbol{v}$ is a nontotal, nonempty admissible cut of $ F $.\\

Connes and Kreimer proved in \cite{Connes} that the vector space $ \mathcal{H} $ generated by the set of rooted forests is a Hopf algebra. Its product is given by the concatenation of rooted forests, and the coproduct is defined for any rooted forest $ F $ by:
$$\Delta(F)=\sum_{\boldsymbol{v} \models V(F)} Lea_{\boldsymbol{v}} (F) \otimes Roo_{\boldsymbol{v}} (F)
=F \otimes 1+1\otimes F+\sum_{\boldsymbol{v} \mmodels V(F)} Lea_{\boldsymbol{v}}(F) \otimes Roo_{\boldsymbol{v}}(F) .$$
For example:
$$\Delta\left(\tquatredeux\right)=\tquatredeux \otimes 1+1\otimes \tquatredeux+
\tun \otimes \ttroisun+\tdeux \otimes \tdeux+\tun \otimes \ttroisdeux+\tun\tun \otimes \tdeux+\tdeux \tun \otimes \tun.$$

\subsubsection{Hopf algebras of planar trees and universal property} \label{HP}

We now recall the construction of the noncommutative generalization of the Connes-Kreimer Hopf algebra \cite{Foissyarbre,Holtkamp}.\\

A {\it planar forest} is a rooted forest $ F $ such that the set of the roots of $ F $ is totally ordered and, for any vertex $v \in V(F)$, the set $\{w \in V(F)\:\mid \:w \rightarrow v\}$ is totally ordered. Planar forests are represented such that the total orders on the set of roots and the sets $\{w \in V(F)\:\mid \:w \rightarrow v\}$ for any $v \in V(F)$ is given from left to right. We denote by $ \mathbb{T}_{P} $ the set of the planar trees.\\

{\bf Examples.} {\begin{enumerate}
\item Planar rooted trees of degree $\leq 5$:
$$\tun,\tdeux,\ttroisun,\ttroisdeux,\tquatreun, \tquatredeux,\tquatretrois,\tquatrequatre,\tquatrecinq,\tcinqun,\tcinqdeux,\tcinqtrois,\tcinqquatre,\tcinqcinq,
\tcinqsix,\tcinqsept,\tcinqhuit,\tcinqneuf,\tcinqdix,\tcinqonze,\tcinqdouze,\tcinqtreize,\tcinqquatorze.$$
\item Planar rooted forests of degree $\leq 4$:
$$1,\tun,\tun\tun,\tdeux,\tun\tun\tun,\tdeux\tun,\tun \tdeux,\ttroisun,\ttroisdeux,\tun\tun\tun\tun,\tdeux\tun\tun,\tun \tdeux \tun, \tun \tun \tdeux,
\ttroisun\tun,\tun \ttroisun,\ttroisdeux\tun,\tun \ttroisdeux,\tdeux\tdeux,\tquatreun,\tquatredeux,\tquatretrois,\tquatrequatre,\tquatrecinq.$$
\end{enumerate}}

If $\boldsymbol{v} \models V(F)$, then $Lea_{\boldsymbol{v}}(F)$ and $Roo_{\boldsymbol{v}}(F)$ are naturally planar forests. It is proved in \cite{Foissyarbre} that the space $ \mathcal{H}_{P} $ generated by planar forests is a Hopf algebra. Its product is given by the concatenation of planar forests and its coproduct is defined for any rooted forest $ F $ by:
$$\Delta(F)=\sum_{\boldsymbol{v} \models V(F)} Lea_{\boldsymbol{v}}(F) \otimes Roo_{\boldsymbol{v}}(F)
=F \otimes 1+1\otimes F+\sum_{\boldsymbol{v} \mmodels V(F)} Lea_{\boldsymbol{v}}(F) \otimes Roo_{\boldsymbol{v}}(F).$$
For example:
\begin{eqnarray*}
\Delta\left(\tquatredeux\right)&=&\tquatredeux \otimes 1+1\otimes \tquatredeux+
\tun \otimes \ttroisun+\tdeux \otimes \tdeux+\tun \otimes \ttroisdeux+\tun\tun \otimes \tdeux+\tdeux \tun \otimes \tun ,\\
\Delta\left(\tquatretrois\right)&=&\tquatretrois \otimes 1+1\otimes \tquatretrois+
\tun \otimes \ttroisun+\tdeux \otimes \tdeux+\tun \otimes \ttroisdeux+\tun\tun \otimes \tdeux+\tun\tdeux \otimes \tun.
\end{eqnarray*}

We define the operator $B_{P}:\mathcal{H}_{P} \longrightarrow \mathcal{H}_{P} $, which associates, to a forest $ F \in \mathcal{H}_{P} $, the tree obtained by grafting the roots of the trees of $ F $ on a common root. For example, $B_{P}(\tdeux \tun)=\tquatredeux$, and $B_{P}(\tun\tdeux)=\tquatretrois$. It is shown in \cite{Moerdijk} that $(\mathcal{H}_{P},B_{P})$ is an initial object in the category of couples $(A,L)$, where $A$ is an algebra and $L:A \longrightarrow A$ any linear operator. Explicitely, one has

\begin{theo}
Let $A$ be any algebra and let $L:A\longrightarrow A$ be a linear map. Then there exists a unique algebra morphism $\phi:\mathcal{H}_{P} \longrightarrow A$, such that $\phi \circ B_{P}=L\circ \phi$.
\end{theo}

{\bf Remarks.} {\begin{enumerate}
\item Note that $\phi$ is inductively defined in the following way: for all trees $T_1,\ldots, T_n \in \mathcal{H}_{P}$,
$$  \left\{ \begin{array}{rcl}
\phi(1)&=&1,\\
\phi(T_1\hdots T_n)&=&\phi(T_1)\hdots \phi(T_n),\\
\phi(B_{P}(T_1\hdots T_n))&=&L(\phi(T_1)\hdots \phi(T_n)).
\end{array}
\right.$$
\item $ \Delta : \mathcal{H}_{P} \longrightarrow \mathcal{H}_{P} \otimes \mathcal{H}_{P} $ is the unique algebra morphism such that
$$ \Delta \circ B_{P} = B_{P} \otimes 1 + (id \otimes B_{P}) \circ \Delta .$$
\end{enumerate}}

\subsubsection{The free right graft algebra}

We here recall some results of \cite{Foissyinfinitesimal}.\\

Let $ F,G \in \mathcal{H}_{P} $ two nonempty planar forests. We set $ F = F_{1} \hdots F_{n} $ and $ F_{n} = B_{P}(H) $. We define:
$$ F \prec G = F_{1} \hdots F_{n-1} B_{P}(H G) .$$
In other terms, $G$ is grafted on the root of the last tree of $F$, on the right. In particular, $\tun \prec G =B_{P}(G)$. \\

{\bf Examples.} {$$\begin{array}{|rclcl|rclcl|rclcl|rclcl|}
\hline
\tdeux &\prec& \tun\tun\tun&=&\tcinqun& \tun \tun \tun &\prec& \tdeux&=& \tun \tun \ttroisdeux&
\tun\tun\tun&\prec&\tun\tun&=&\tun\tun\ttroisun&\tun\tun&\prec&\tun\tun\tun&=&\tun\tquatreun\\
\tdeux &\prec& \tdeux \tun &=&\tcinqtrois&\tdeux \tun &\prec& \tdeux&=& \tdeux \ttroisdeux &
\tdeux \tun&\prec&\tun\tun&=&\tdeux\ttroisun&\tun\tun&\prec&\tun\tdeux&=&\tun\tquatretrois\\
\tdeux &\prec& \tun \tdeux &=&\tcinqquatre&\tun \tdeux &\prec& \tdeux  &=&\tun\tquatretrois&
\tun\tdeux&\prec&\tun\tun&=&\tun\tquatreun& \tun\tun&\prec&\tdeux\tun&=&\tun\tquatredeux\\
\tdeux &\prec& \ttroisun&=&\tcinqsept&\ttroisun &\prec& \tdeux&=&\tcinqquatre&
\ttroisun&\prec&\tun\tun&=&\tcinqun&\tun\tun&\prec&\ttroisun&=&\tun\tquatrequatre\\
\tdeux &\prec& \ttroisdeux&=&\tcinqneuf&\ttroisdeux &\prec& \tdeux&=&\tcinqcinq&
\ttroisdeux&\prec& \tun\tun&=&\tcinqdeux& \tun\tun&\prec&\ttroisdeux&=&\tun\tquatrecinq \\
\hline
\end{array}$$}

We denote by $ \mathcal{M}_{P} $ the augmentation ideal of $ \mathcal{H}_{P} $, that is to say the vector space generated by the nonempty forests of $ \mathcal{H}_{P} $. We extend $ \prec : \mathcal{M}_{P}  \otimes \mathcal{M}_{P} \rightarrow \mathcal{M}_{P} $ by linearity.\\

{\bf Remark.} {$ \prec $ is not associative :
\begin{eqnarray*}
\tun \prec (\tun \prec \tun) = \ttroisdeux \neq \ttroisun = (\tun \prec \tun) \prec \tun .
\end{eqnarray*}}

The following result is proved in \cite{Foissyinfinitesimal}.

\begin{theo} \label{mestlg}
$ (\mathcal{M}_{P} , m , \prec) $ is the free $ \mathcal{RG} $-algebra generated by $ \tun $ ($ m $ is the concatenation product).
\end{theo}

\subsection{Relations with the coproduct}

First, we recall the definition of a dipterous algebra (see \cite{Loday}):

\begin{defi}
A (right) dipterous algebra is a $ \mathbb{K} $-vector space $ A $ equipped with two binary operations denoted $ \ast $ and $ \prec $ satisfying the following relations : for all $ x , y, z \in A $,
\begin{eqnarray}
(x \ast y) \ast z & = & x \ast (y \ast z) , \label{axiomdiptere1} \\
(x \prec y) \prec z & = & x \prec (y \ast z) . \label{axiomdiptere2}
\end{eqnarray}
\end{defi}

Dipterous algebras do not have unit for the product $ \ast $. If $ A $ and $ B $ are two dipterous algebras, we say that a $ \mathbb{K} $-linear map $ f : A \rightarrow B $ is a dipterous morphism if $ f(x \ast y) = f(x) \ast f(y) $ and $ f(x \prec y) = f(x) \prec f(y) $ for all $ x,y \in A $. We denote by $ Dipt $-alg the category of dipterous algebras.\\

{\bf Remark.} {A $ \mathcal{RG} $-algebra is also a dipterous algebra. We get the following canonical functor: $ \mathcal{RG} \mbox{-alg} \rightarrow Dipt \mbox{-alg} $.}
\\

As $ \mathcal{RG} $-algebras are not unitary objects, we need to extend the usual tensor product in order to obtain a copy of $ A $ and $ B $ in the tensor product of $ A $ and $ B $.

\begin{defi}
Let $ A, B $ be two vector spaces. Then :
$$ A \overline{\otimes} B = (A \otimes \mathbb{K}) \oplus (A \otimes B) \oplus (\mathbb{K} \otimes B) .$$
\end{defi}

Let $ A $ be a $ \mathcal{RG} $-algebra. We extend $ \prec : A \otimes A \rightarrow A $ to a map $ \prec : A \overline{\otimes} A \rightarrow A $ in the following way : for all $ a \in A $, $ a \prec 1 = a $ and  $ 1 \prec a = 0 $. Moreover, we extend the product of $ A $ to a map from $ (A \oplus \mathbb{K}) \otimes (A \oplus \mathbb{K}) $ to $ A \oplus \mathbb{K} $ by putting $ 1.a = a.1 = a $ for all $ a \in A $ and $ 1.1 = 1 $. Note that $ 1 \prec 1 $ is not defined.

\begin{prop} \label{tenseurdipterous} Let $ A $ and $ B $ be two $ \mathcal{RG} $-algebras. Then $ A \overline{\otimes} B $ is given a structure of dipterous algebra in the following way : for $ a,a' \in A \cup \mathbb{K} $ and $ b , b' \in B \cup \mathbb{K} $,
\begin{eqnarray*}
(a \otimes b) \ast (a' \otimes b') & = & (a \ast a') \otimes (b \ast b') ,\\
(a \otimes b) \prec (a' \otimes b') & = & (a \ast a') \otimes (b \prec b') , \mbox{ if $ b $ or $ b' \in B $}, \\
(a \otimes 1) \prec (a' \otimes 1)  & = & (a \prec a') \otimes 1 .
\end{eqnarray*}
\end{prop}

\begin{proof}
The associativity of $ \ast : (A \overline{\otimes} B) \otimes (A \overline{\otimes} B) \rightarrow A \overline{\otimes} B $ (that is to say the relation (\ref{axiomdiptere1})) is obvious. We prove (\ref{axiomdiptere2}) : for all $ a , a' , a'' \in A $ and $ b , b' ,b'' \in B $,
\begin{eqnarray*}
(a \otimes b) \prec ((a' \otimes b') \ast (a'' \otimes b'')) & = & (a \otimes b) \prec ((a' \ast a'') \otimes (b' \ast b'')) \\
& = & (a \ast (a' \ast a'')) \otimes (b \prec (b' \ast b'')) \\
& = & ((a \ast a') \ast a'' ) \otimes ((b \prec b') \prec b'') \\
& = & ((a \ast a') \otimes (b \prec b')) \prec (a'' \otimes b'') \\
& = & ((a \otimes b) \prec (a' \otimes b')) \prec (a'' \otimes b'') .
\end{eqnarray*}
This calculation is still true if $ b , b' $ or $ b'' $ is equal to $ 1 $ or if $ b' = b'' = 1 $ and $ b \in B $. If $ b = b'' = 1 $ and $ b' \in B $,
\begin{eqnarray*}
(a \otimes 1) \prec ((a' \otimes b') \ast (a'' \otimes 1)) & = & (a \otimes 1) \prec ((a' \ast a'') \otimes b') \\
& = & (a \ast (a' \ast a'')) \otimes (1 \prec b') \\
& = & 0 ,\\
((a \otimes 1) \prec (a' \otimes b')) \prec (a'' \otimes 1) & = & ((a \ast a') \otimes (1 \prec b')) \prec (a'' \otimes 1) \\
& = & 0 .
\end{eqnarray*}
If $ b = b' = 1 $ and $ b'' \in B $, then $ a $ and $ a' $ are not equal to $ 1 $ and
\begin{eqnarray*}
(a \otimes 1) \prec ((a' \otimes 1) \ast (a'' \otimes b'')) & = & (a \otimes 1) \prec ((a' \ast a'') \otimes b'')\\
& = & (a \ast (a' \ast a'')) \otimes (1 \prec b'') \\
& = & 0 ,\\
((a \otimes 1) \prec (a' \otimes 1)) \prec (a'' \otimes b'') & = & ((a \prec a') \otimes 1) \prec (a'' \otimes b'') \\
& = & ((a \ast a') \succ a'') \otimes (1 \succ b'') \\
& = & 0 .
\end{eqnarray*}
Finally if $ b = b' = b'' = 1 $, then $ a , a' $ and $ a'' $ are not equal to $ 1 $ and
\begin{eqnarray*}
(a \otimes 1) \prec ((a' \otimes 1) \ast (a'' \otimes 1)) & = & (a \otimes 1) \prec ((a' \ast a'') \otimes 1) \\
& = & (a \prec (a' \ast a'')) \otimes 1 \\
& = & ((a \prec a') \prec a'' ) \otimes 1 \\
& = & ((a \prec a') \otimes 1) \prec (a'' \otimes 1) \\
& = & ((a \otimes 1) \prec (a' \otimes 1)) \prec (a'' \otimes 1) .
\end{eqnarray*}
In all cases, the relation (\ref{axiomdiptere2}) is satisfied and $ A \overline{\otimes} B $ is a dipterous algebra.
\end{proof}
\\

{\bf Remarks.} {\begin{enumerate}
\item $ A \otimes \mathbb{K} $ is a dipterous subalgebra of $ A \overline{\otimes} B $ which is isomorphic to $ A $, and $ \mathbb{K} \otimes B $ is a dipterous subalgebra of $ A \overline{\otimes} B $ which is isomorphic to $ B $.
\item Suppose that $ A,B \ne \{0\} $. $ (A \overline{\otimes} B , \ast , \prec ) $ is a $ \mathcal{RG} $-algebra if and only if $ \prec : A \otimes A \rightarrow A $ and $ \prec : B \otimes B \rightarrow B $ are zero. Indeed, if $ A \overline{\otimes} B $ is a $ \mathcal{RG} $-algebra then for all $ a, a' \in A $ and $ b,b' \in B $,
\begin{eqnarray*}
(a \prec a') \otimes b & = & ((a \prec a') \otimes 1) \ast (1 \otimes b) \\
& = & ((a \otimes 1) \prec (a' \otimes 1)) \ast (1 \otimes b) \\
& = & (a \otimes 1) \prec (a' \otimes b) \\
& = & (a \ast a') \otimes (1 \prec b) \\
& = & 0 ,
\end{eqnarray*}
therefore, by taking $ b \neq 0 $, $ a \prec a' = 0 $ for all $ a , a' \in A $. Moreover,
\begin{eqnarray*}
a \otimes (b \prec b') & = & (a \otimes b) \prec (1 \otimes b') \\
& = & ((1 \otimes b) \ast (a \otimes 1)) \prec (1 \otimes b') \\
& = & (1 \otimes b) \ast (a \otimes (1 \prec b')) \\
& = & 0 ,
\end{eqnarray*}
therefore, by taking $ a \neq 0 $, $ b \prec b' = 0 $ for all $ b,b' \in B $.\\

Reciprocally, if $ \prec : A \otimes A \rightarrow A $ and $ \prec : B \otimes B \rightarrow B $ are zero, it is clear that $ (A \overline{\otimes} B , \ast , \prec ) $ is a $ \mathcal{RG} $-algebra.
\end{enumerate}}

\begin{prop} \label{relcopg}
For all tree $ T \in \mathcal{M}_{P} $ and for every forest $ F \in \mathcal{M}_{P} $, $ \Delta(T \prec F) = \Delta(T) \prec \Delta(F) $.
\end{prop}

{\bf Remark.} {As $ \mathcal{M}_{P} \overline{\otimes} \mathcal{M}_{P} $ is not a $ \mathcal{RG} $-algebra, if $ T \in \mathcal{M}_{P} $ is a forest, $ \Delta(T \prec F) \neq \Delta(T) \prec \Delta(F) $ in general. For example,
\begin{eqnarray*}
\Delta((\tun\tun) \prec \tun) & = & \Delta( \tun \tdeux ) \\
& = & \tun \tdeux \otimes 1 + 1 \otimes \tun \tdeux + \tdeux \otimes \tun + \tun \otimes \tdeux + \tun \tun \otimes \tun + \tun \otimes \tun \tun \\
\Delta (\tun \tun) \prec \Delta (\tun) & = & (\tun \tun \otimes 1 + 1 \otimes \tun \tun + 2 \tun \otimes \tun ) \prec (\tun \otimes 1 + 1 \otimes \tun)\\
& = & \tun \tdeux \otimes 1 + 1 \otimes \tun \tdeux + \tun \otimes \tun \tun + 2 \tun \tun \otimes \tun + 2 \tun \otimes \tdeux .
\end{eqnarray*}}

\begin{proof}
Let $ T $ and $ F $ are two planar trees $ \in \mathcal{M}_{P} $. We denote $ \Delta(T) = T \otimes 1 + 1 \otimes T + \sum_{T} T^{(1)} \otimes T^{(2)} $ and $ \Delta(F) = F \otimes 1 + 1 \otimes F + \sum_{F} F^{(1)} \otimes F^{(2)} $. Then
\begin{eqnarray*}
\Delta(T) \prec \Delta(F) & = & \left( T \otimes 1 + 1 \otimes T + \sum_{T} T^{(1)} \otimes T^{(2)} \right) \prec \left( F \otimes 1 + 1 \otimes F + \sum_{F} F^{(1)} \otimes F^{(2)} \right) \\
& = & (T \prec F) \otimes 1 + T \otimes F + \sum_{T} T^{(1)} F \otimes T^{(2)} + 1 \otimes (T \prec F) \\
& & + \sum_{T} T^{(1)} \otimes (T^{(2)} \prec F) + \sum_{F} F^{(1)} \otimes (T \prec F^{(2)}) \\
& & + \sum_{T,F} T^{(1)} F^{(1)} \otimes (T^{(2)} \prec F^{(2)}) \\
& = & \Delta(T \prec F) .
\end{eqnarray*}
If $ F = F_{1} \hdots F_{n} \in \mathcal{M}_{P} $ is a forest and $ T $ is again a tree $ \in \mathcal{M}_{P} $,
\begin{eqnarray*}
\Delta(T) \prec \Delta(F) & = & \Delta(T) \prec \left( \Delta(F_{1}) \hdots \Delta(F_{n}) \right) \\
& = & ( \hdots ((\Delta(T) \prec \Delta(F_{1})) \prec \Delta(F_{2})) \hdots \prec \Delta(F_{n-1})) \prec \Delta(F_{n}) \\
& = & \Delta(( \hdots ((T \prec F_{1}) \prec F_{2}) \hdots \prec F_{n-1}) \prec F_{n}) \\
& = & \Delta (T \prec F) ,
\end{eqnarray*}
as $ A \overline{\otimes} B $ is a dipterous algebra for the second equality.
\end{proof}

\subsection{A rigidity theorem for the right graft algebras}

We now give a rigidity theorem for the right graft algebras (from \cite{Foissyinfinitesimal}).\\

Recall that if $ A $ and $ B $ are two $ \mathcal{RG} $-algebras then $ A \overline{\otimes} B $ is a dipterous algebra (proposition \ref{tenseurdipterous}). We define another coproduct $ \Delta_{\mathcal{A}ss} $ on $ \mathcal{H}_{P} $ (the deconcatenation) in the following way: for every forest $ F \in \mathcal{H}_{P} $,
$$ \Delta_{\mathcal{A}ss} (F) = \sum_{F_{1},F_{2} \in \mathcal{H}_{P} , F_{1} F_{2} = F} F_{1} \otimes F_{2} .$$

We now have defined two products, namely $ m $ and $ \succ $ and one coproduct on $ \mathcal{M}_{P} $, namely $ \tdelta_{\mathcal{A}ss} $, obtained from $ \Delta_{\mathcal{A}ss} $ by substracting its primitive part. The following properties sum up the different compatibilities.

\begin{prop} \label{infibialgLG}
For all $ x,y \in \mathcal{M}_{P} $;
\begin{eqnarray*}
\left\lbrace \begin{array}{rcl}
\tdelta_{\mathcal{A}ss} (xy) & = & (x \otimes 1) \tdelta_{\mathcal{A}ss} (y) + \tdelta_{\mathcal{A}ss} (x) (1 \otimes y) + x \otimes y ,\\
\tdelta_{\mathcal{A}ss} (x \prec y) & = & \tdelta_{\mathcal{A}ss}(x) \prec (1 \otimes y) .
\end{array} \right. 
\end{eqnarray*}
\end{prop}

In particular, with the first equality, $ (\mathcal{M}_{P} , \Delta_{\mathcal{A}ss} ) $ is an infinitesimal bialgebra (see \cite{Loday}).\\

\begin{proof}
We can restrict ourselves to $ F , G \in \mathcal{H}_{P} $ two nonempty forests. We put $ F = F_{1} \hdots F_{n} $, $ G = G_{1} \hdots G_{m} $ where the $ F_{i} $'s and the $ G_{i} $'s are trees and $ F_{n} = B_{P}(H) $. Hence :
\begin{eqnarray*}
\tdelta_{\mathcal{A}ss} (F G) & = & \sum_{H_{1},H_{2} \in \mathcal{M}_{P} , H_{1} H_{2} = F G} H_{1} \otimes H_{2} \\
& = & \sum_{H_{1},H_{2} \in \mathcal{M}_{P} , H_{1} H_{2} = G} F H_{1} \otimes H_{2} + \sum_{H_{1},H_{2} \in \mathcal{M}_{P} , H_{1} H_{2} = F} H_{1} \otimes H_{2} G + F \otimes G \\
& = & (F \otimes 1) \tdelta_{\mathcal{A}ss} (G) + \tdelta_{\mathcal{A}ss}(F) (1 \otimes G) + F \otimes G .
\end{eqnarray*}
It is the same case as in \cite{Loday}. Moreover,
\begin{eqnarray*}
\tdelta_{\mathcal{A}ss} (F \prec G) & = & \tdelta_{\mathcal{A}ss} ( F_{1} \hdots F_{n-1} B_{P}(H G) ) \\
& = & \sum_{i=1}^{n-2} F_{1} \hdots F_{i} \otimes F_{i+1} \hdots F_{n-1} B_{P}(H G) + F_{1} \hdots F_{n-1} \otimes B_{P}(HG) \\
& = & \sum_{i=1}^{n-2} F_{1} \hdots F_{i} \otimes (F_{i+1} \hdots F_{n-1} F_{n}) \prec G + F_{1} \hdots F_{n-1} \otimes F_{n} \prec G \\
& = & \tdelta_{\mathcal{A}ss} (F) \prec (1 \otimes G) .
\end{eqnarray*}
\end{proof}

This justifies the following definition :

\begin{defi}
A $ \mathcal{RG} $-infinitesimal bialgebra is a family $ (A,m,\prec,\tdelta_{\mathcal{A}ss} ) $ where $ A $ is a $ \mathbb{K} $-vector space, $ m , \prec : A \otimes A \rightarrow A $ and $ \tdelta_{\mathcal{A}ss} : A \rightarrow A \otimes A $ are $ \mathbb{K} $-linear maps, with the following compatibilities :
\begin{enumerate}
\item $ (A,m,\prec ) $ is a $ \mathcal{RG} $-algebra.
\item For all $ x,y \in A $ :
\begin{eqnarray} \label{formuleright}
\left\lbrace \begin{array}{rcl}
\tdelta_{\mathcal{A}ss} (xy) & = & (x \otimes 1) \tdelta_{\mathcal{A}ss} (y) + \tdelta_{\mathcal{A}ss} (x) (1 \otimes y) + x \otimes y ,\\
\tdelta_{\mathcal{A}ss} (x \prec y) & = & \tdelta_{\mathcal{A}ss}(x) \prec (1 \otimes y) .
\end{array} \right. 
\end{eqnarray}
\end{enumerate}
\end{defi}

With theorem \ref{mestlg} and proposition \ref{infibialgLG}, we have immediately :

\begin{prop}
$ (\mathcal{M}_{P},m,\prec,\tdelta_{\mathcal{A}ss}) $ is a $ \mathcal{RG} $-infinitesimal bialgebra.
\end{prop}

{\bf Remark.} {If $ A $ is a $ \mathcal{RG} $-infinitesimal bialgebra, we denote $ Prim(A) = Ker(\tdelta_{\mathcal{A}ss}) $. In the $ \mathcal{RG} $-infinitesimal bialgebra $ \mathcal{M}_{P} $, $ Prim(\mathcal{M}_{P}) = \mathbb{K}[\mathbb{T}_{P} \setminus \{1\} ] $. We denote by $ \mathcal{P}_{P} $ the primitive part of $ \mathcal{M}_{P} $.}
\\

Let us recall the definition of a magmatic algebra :

\begin{defi}
A magmatic algebra (or $ \mathcal{M}ag $-algebra for short) is a $ \mathbb{K} $-vector space $ A $ equipped with a binary operation $ \bullet $, without any relation.
\end{defi}

We do not suppose that $ \mathcal{M}ag $-algebras have units. If $ A $ and $ B $ are two $ \mathcal{M}ag $-algebras, we say that a $ \mathbb{K} $-linear map $ f : A \rightarrow B $ is a $ \mathcal{M}ag $-morphism if $ f(x \bullet y) = f(x) \bullet f(y) $ for all $ x,y \in A $. We denote by $ \mathcal{M}ag $-alg the category of $ \mathcal{M}ag $-algebras.

\begin{prop} \label{primag}
For any $ \mathcal{RG} $-infinitesimal bialgebra, its primitive part is a $ \mathcal{M}ag $-algebra.
\end{prop}

\begin{proof}
Let $ A $ be a $ \mathcal{RG} $-infinitesimal bialgebra. If $ x \in Prim(A) $ and $ y \in A $ then $ \tdelta_{\mathcal{A}ss} (x \prec y) = \tdelta_{\mathcal{A}ss}(x) \prec (1 \otimes y) = 0 $ by (\ref{formuleright}). So $ (Prim(A) , \prec) $ is a $ \mathcal{M}ag $-algebra.
\end{proof}

\begin{prop}
$ (\mathcal{P}_{P},\prec) $ is the free $ \mathcal{M}ag $-algebra generated by $ \tun $.
\end{prop}

\begin{proof}
With proposition \ref{primag}, we have immediately that $ (\mathcal{P}_{P},\prec) $ is a $ \mathcal{M}ag $-algebra. Let us prove that $ (\mathcal{P}_{P},\prec) $ is the free $ \mathcal{M}ag $-algebra generated by $ \tun $. \\

Let $A$ be a $ \mathcal{M}ag $-algebra and let $a \in A$. Let us prove that there exists a unique morphism of $ \mathcal{M}ag $-algebras $ \phi :\mathcal{P}_{P} \rightarrow A$, such that $ \phi(\tun)=a $. We define $ \phi(T) $ for any nonempty tree $ T \in \mathcal{P}_{P} $ inductively on the degree of $ T $ by:
$$\left\{\begin{array}{rcl}
\phi(\tun)&=&a,\\
\phi(B_{P}(T_{1} \hdots T_{k}))&=& ( \hdots ((a \prec \phi(T_{1})) \prec \phi(T_{2})) \hdots \prec \phi(T_{k-1})) \prec \phi(T_{k}) \mbox{ if } k \geq 1.
\end{array}\right.$$
This map is linearly extended into a map $\phi: \mathcal{P}_{P} \rightarrow A$. Let us show that it is a morphism of $ \mathcal{M}ag $-algebras. Let $F,G$ be two nonempty trees. Let us prove that $\phi(F \prec G)=\phi(F) \prec \phi(G)$. We put $ F = B_{P}(H_{1} \hdots H_{k}) $. Then:
\begin{eqnarray*}
\phi(F \prec G) & = & \phi(B_{P}(H_{1} \hdots H_{k} G)) \\
& = & ( \hdots ((a \prec \phi(H_{1})) \prec \phi(H_{2})) \hdots \prec \phi(H_{k})) \prec \phi(G)\\
& = & \phi(B_{P}(H_{1} \hdots H_{k})) \prec \phi(G)\\
& = & \phi(F) \prec \phi(G)
\end{eqnarray*}
So $\phi$ is a morphism of $\mathcal{M}ag$-algebras. \\

Let $\phi':\mathcal{P}_{P} \rightarrow A$ be another morphism of $\mathcal{M}ag$-algebras such that $\phi'(\tun) = a $. Then $ T_{1} \hdots T_{k} \in \mathcal{P}_{P} $,
\begin{eqnarray*}
\phi'(B_{P}(T_{1} \hdots T_{k})) & = & \phi' \left( ( \hdots ((\tun \prec T_{1}) \prec T_{2}) \hdots \prec T_{k-1}) \prec T_{k} \right) \\
& = & ( \hdots ((a \prec \phi'(T_{1})) \prec \phi'(T_{2})) \hdots \prec \phi'(T_{k-1})) \prec \phi'(T_{k}) .
\end{eqnarray*}
So $\phi=\phi'$.
\end{proof}
\\

The following result is proved in \cite{Foissyinfinitesimal} :

\begin{theo}
The triple of operads $ (\mathcal{A}ss,\mathcal{RG},\mathcal{M}ag) $ is a good triple of operads.
\end{theo}

%% file: algebrebigreffelibre.tex
\section{Hopf algebra of planar decorated trees}

\subsection{Construction}

To start, we shall need decorated versions of the Hopf algebra $ \mathcal{H}_{P} $ described in the section \ref{HP}. If $ T \in \mathcal{H}_{P} $ is a tree, we note $ E(T) $ the set of edges of $ T $. 

\begin{defi}
Let $ \mathcal{D} $ be a nonempty set.
\begin{enumerate}
\item A planar rooted tree decorated by $ \mathcal{D} $ is a couple $ (T,d) $, where $ T $ is a planar tree and $ d : E(T) \longrightarrow \mathcal{D} $ is any map.
\item A planar rooted forest decorated by $ \mathcal{D} $ is a noncommutative monomial in planar rooted trees decorated by $ \mathcal{D} $.
\end{enumerate}
\end{defi}

Let $ T $ be a planar rooted tree decorated by $ \mathcal{D} $. If $ v , w \in V(T) $ and $ d \in \mathcal{D} $, we shall denote $ v \overset{d}{\rightarrow} w $ if there is an edge in $ T $ from $ v $ to $ w $ decorated by $ d $.\\

{\bf Note.} {This definition is different from \cite{Foissyarbre}, where vertices of the trees are decorated.}\\

{\bf Examples.} {\begin{enumerate}
\item Planar rooted trees decorated by $ \mathcal{D} $ with degree smaller than $ 4 $:
$$ 1, \hspace{0.5cm} \tun , \hspace{0.5cm} \addeux{$a$}, a \in \mathcal{D}, \hspace{0.5cm} \adtroisun{$b$}{$a$}, \adtroisdeux{$a$}{$b$} , (a,b) \in \mathcal{D}^{2} , \hspace{0.5cm} \adquatreun{$a$}{$b$}{$c$} , \adquatredeux{$b$}{$a$}{$c$} , \adquatretrois{$b$}{$a$}{$c$} , \adquatrequatre{$a$}{$b$}{$c$}, \adquatrecinq{$a$}{$b$}{$c$}, (a,b,c) \in \mathcal{D}^{3} .$$
\item Planar rooted forests decorated by $ \mathcal{D} $ with degree smaller than $ 3 $:
$$ 1, \hspace{0.5cm} \tun , \hspace{0.5cm} \tun \tun , \addeux{$a$}, a \in \mathcal{D} , \hspace{0.5cm} \tun \tun \tun , \tun \addeux{$a$} , \addeux{$a$} \tun , \adtroisun{$b$}{$a$}, \adtroisdeux{$a$}{$b$} , (a,b) \in \mathcal{D}^{2} .$$
\end{enumerate}}

If $\boldsymbol{v} \models V(F)$, then $Lea_{\boldsymbol{v}}(F)$ and $Roo_{\boldsymbol{v}}(F)$ are naturally planar decorated forests. The space $ \mathcal{H}_{P}^{\mathcal{D}} $ generated by planar forests decorated by $ \mathcal{D} $ is a Hopf algebra. Its product is given by the concatenation of planar decorated forests and its coproduct is defined for any rooted decorated forest $ F $ by:
$$\Delta(F)=\sum_{\boldsymbol{v} \models V(F)} Lea_{\boldsymbol{v}}(F) \otimes Roo_{\boldsymbol{v}}(F)
=F \otimes 1+1\otimes F+\sum_{\boldsymbol{v} \mmodels V(F)} Lea_{\boldsymbol{v}}(F) \otimes Roo_{\boldsymbol{v}}(F).$$
For example: if $ (a,b,c) \in \mathcal{D}^{3} $
\begin{eqnarray*}
\Delta\left(\adquatredeux{$b$}{$a$}{$c$}\right)&=&\adquatredeux{$b$}{$a$}{$c$} \otimes 1+1\otimes \adquatredeux{$b$}{$a$}{$c$}+
\tun \otimes \adtroisun{$b$}{$a$}+\addeux{$c$} \otimes \addeux{$b$}+\tun \otimes \adtroisdeux{$a$}{$c$}+\tun\tun \otimes \addeux{$a$}+\addeux{$c$} \tun \otimes \tun,\\
\Delta\left(\adquatretrois{$b$}{$a$}{$c$}\right)&=&\adquatretrois{$b$}{$a$}{$c$} \otimes 1+1\otimes \adquatretrois{$b$}{$a$}{$c$}+
\tun \otimes \adtroisun{$b$}{$a$}+\addeux{$c$} \otimes \addeux{$a$}+\tun \otimes \adtroisdeux{$b$}{$c$}+\tun\tun \otimes \addeux{$b$}+\tun\addeux{$c$} \otimes \tun.
\end{eqnarray*}

In the following, \underline{we take $ \mathcal{D} = \{l,r \} $}. We define the operator $ B : \mathcal{H}_{P}^{\mathcal{D}} \otimes \mathcal{H}_{P}^{\mathcal{D}} \longrightarrow \mathcal{H}_{P}^{\mathcal{D}} $, which associates, to a tensor $ F \otimes G $ of two forests $ F , G \in \mathcal{H}_{P}^{\mathcal{D}} $, the tree obtained by grafting the roots of the trees of $ F $ and $ G $ on a common root and by decorating the edges between the common root and the roots of $ F $ by $ l $ and the edges between the common root and the roots of $ G $ by $ r $.\\

{\bf Examples.} {$$ \begin{array}{|rcl|rcl|rcl|rcl|} 
\hline
B(1 \otimes 1) & = & \tun & B(\tun \otimes 1) & = & \addeux{$l$} & B(1 \otimes \tun) & = & \addeux{$r$} & B(\tun \tun \otimes 1) & = & \adtroisun{$l$}{$l$} \\
B(\tun \otimes \tun) & = & \adtroisun{$r$}{$l$} & B(1 \otimes \tun \tun) & = & \adtroisun{$r$}{$r$} & B(1 \otimes \addeux{$l$}) & = & \adtroisdeux{$r$}{$l$} & B(\tun \otimes \tun \tun) & = & \adquatreun{$l$}{$r$}{$r$} \\
B(\addeux{$l$} \tun \otimes 1) & = & \adquatredeux{$l$}{$l$}{$l$} & B(\tun \otimes \addeux{$l$}) & = & \adquatretrois{$r$}{$l$}{$l$} & B(\adtroisun{$r$}{$l$} \otimes 1) & = & \adquatrequatre{$l$}{$l$}{$r$} & B(\adtroisdeux{$r$}{$l$} \otimes 1) & = & \adquatrecinq{$l$}{$r$}{$l$} \\
\hline
\end{array} $$}

For all $ n \geq 1 $, we give an inductive definition of a set $ \mathbb{T}(n) $ that corresponds of a set of trees in $ \mathcal{H}_{P}^{\mathcal{D}} $ of degree $ n $.
\begin{itemize}
\item If $ n = 1 $, $ \mathbb{T}(1) = \{ \tun \} $ is the set reduced to one element, the unique tree of degree $ 1 $.
\item If $ n \geq 2 $, we suppose the set $ \bigcup_{1 \leq k \leq n-1} \mathbb{T}(k) $ already constructed. Pick elements $ F_{1} , \hdots , F_{p} $ and $ G_{1}, \hdots , G_{q} \in \bigcup_{1 \leq k \leq n-1} \mathbb{T}(k) $ such that $ \left| F_{1} \right| + \hdots + \left| F_{p} \right| + \left| G_{1} \right| + \hdots + \left| G_{q} \right| = n-1 $ and consider the new tree $ B(F_{1} \hdots F_{p} \otimes G_{1} \hdots G_{q}) $.
\end{itemize}

Let $ \mathbb{T} = \bigcup_{n \geq 0} \mathbb{T}(n) $ with $ \mathbb{T}(0) = \{1 \} $, $ \mathbb{F}(n) $ be the forests of degree $ n $ constructed with trees of $ \mathbb{T} $ and $ \mathbb{F} = \bigcup_{n \geq 0} \mathbb{F}(n) $.\\

{\bf Examples.} {Planar decorated trees of $ \mathbb{T} $ with degree $ \leq 4 $:
\begin{center}
$1, \tun , \addeux{$l$} , \addeux{$r$} , \adtroisun{$l$}{$l$}, \adtroisun{$r$}{$l$} , \adtroisun{$r$}{$r$} , \adtroisdeux{$l$}{$l$} , \adtroisdeux{$l$}{$r$} , \adtroisdeux{$r$}{$l$} , \adtroisdeux{$r$}{$r$} , \adquatreun{$l$}{$l$}{$l$} , \adquatreun{$l$}{$l$}{$r$} , \adquatreun{$l$}{$r$}{$r$} , \adquatreun{$r$}{$r$}{$r$} , \adquatredeux{$l$}{$l$}{$l$} , \adquatredeux{$l$}{$l$}{$r$} , \adquatredeux{$r$}{$l$}{$l$} , \adquatredeux{$r$}{$l$}{$r$} , \adquatredeux{$r$}{$r$}{$l$} ,$
\end{center}

\begin{center}
$\adquatredeux{$r$}{$r$}{$r$} , \adquatretrois{$l$}{$l$}{$l$} , \adquatretrois{$l$}{$l$}{$r$} , \adquatretrois{$r$}{$l$}{$l$} , \adquatretrois{$r$}{$l$}{$r$} , \adquatretrois{$r$}{$r$}{$l$} , \adquatretrois{$r$}{$r$}{$r$} , \adquatrequatre{$l$}{$l$}{$l$} , \adquatrequatre{$r$}{$l$}{$l$} , \adquatrequatre{$l$}{$l$}{$r$} , \adquatrequatre{$r$}{$l$}{$r$} , \adquatrequatre{$l$}{$r$}{$r$} , \adquatrequatre{$r$}{$r$}{$r$} , \adquatrecinq{$l$}{$l$}{$l$} , \adquatrecinq{$l$}{$l$}{$r$} , \adquatrecinq{$l$}{$r$}{$l$} , \adquatrecinq{$l$}{$r$}{$r$} , \adquatrecinq{$r$}{$l$}{$l$} , \adquatrecinq{$r$}{$l$}{$r$} , \adquatrecinq{$r$}{$r$}{$l$} , \adquatrecinq{$r$}{$r$}{$r$}.$
\end{center}}

\begin{lemma}
Let $ T \in \mathcal{H}_{P}^{\mathcal{D}} $ be a nonempty tree. $ T \in \mathbb{T} $ if and only if for all $ v \in V(T) $ and $ w_{1}, w_{2} \in V(T) $ such as $ w_{1} \overset{d_{1}}{\rightarrow} v $, $ w_{2} \overset{d_{2}}{\rightarrow} v $ and $ w_{1} $ is left to $ w_{2} $, $ (d_{1},d_{2}) \in \{(l,l),(l,r),(r,r) \} $, that is to say that we cannot have $ (d_{1},d_{2}) = (r,l) $.
\end{lemma}

\begin{proof}
By induction on the degree $ n $ of $ T $. If $ n = 1 $, $ T = \tun $ and it is trivial. Let $ T \in \mathcal{H}_{P}^{\mathcal{D}} $ be a nonempty tree of degree $ n \geq 2 $ and assume that the property is true at rank $ \leq n-1 $.\\

By construction, if $ T \in \mathbb{T} $, there are $ F_{1}, \hdots , F_{p} $ and $ G_{1} , \hdots , G_{q} \in \mathbb{T} $ such that $ T = B(F_{1} \hdots F_{p} \otimes G_{1} \hdots G_{q}) $ and $ \left| F_{i} \right| , \left| G_{i} \right| \leq n-1 $. Consider $ v \in V(T) $ and $ w_{1}, w_{2} \in V(T) $ such as $ w_{1} \overset{d_{1}}{\rightarrow} v $, $ w_{2} \overset{d_{2}}{\rightarrow} v $ and $ w_{1} $ is left to $ w_{2} $. If $ v $ is not the root of $ T $, $ v \in V(F_{i}) $ or $ v \in V(G_{i}) $ and by induction hypothesis $ (d_{1},d_{2}) \in \{(l,l),(l,r),(r,r) \} $. If $ v $ is the root of $ T $, we have three cases:

1. $ w_{1} $ is the root of $ F_{i} $ and $ w_{2} $ is the root of $ F_{j} $ with $ 1 \leq i < j \leq p $. Then $ (d_{1},d_{2}) = (l,l) $.

2. $ w_{1} $ is the root of $ F_{i} $ and $ w_{2} $ is the root of $ G_{j} $ with $ 1 \leq i \leq p $ and $ 1 \leq j \leq q $. Then $ (d_{1},d_{2}) = (l,r) $.

3. $ w_{1} $ is the root of $ G_{i} $ and $ w_{2} $ is the root of $ G_{j} $ with $ 1 \leq i < j \leq q $. Then $ (d_{1},d_{2}) = (r,r) $.\\
In all cases, $ (d_{1},d_{2}) \in \{(l,l),(l,r),(r,r) \} $.\\

Reciprocally, let $ T \in \mathcal{H}_{P}^{\mathcal{D}} $ be a nonempty tree such that for all $ v \in V(T) $ and $ w_{1}, w_{2} \in V(T) $ such as $ w_{1} \overset{d_{1}}{\rightarrow} v $, $ w_{2} \overset{d_{2}}{\rightarrow} v $ and $ w_{1} $ is left to $ w_{2} $, $ (d_{1},d_{2}) \in \{(l,l),(l,r),(r,r) \} $. If $ v $ is the root of $ T $, this condition implies that $ T = B(F_{1} \hdots F_{p} \otimes G_{1} \hdots G_{q}) $ with $ F_{i}, G_{i} \in \mathcal{H}_{P}^{\mathcal{D}} $ and $ \left| F_{i} \right| , \left| G_{i} \right| \leq n-1 $. Then $ F_{i} $ and $ G_{i} $ satisfy the condition of the decorations of the edges. Therefore $ F_{i}, G_{i} \in \mathbb{T} $ by induction hypothesis and $ T \in \mathbb{T} $.
\end{proof}
\\

Let $ \mathcal{BT} = \mathbb{K}[\mathbb{T}] $ be the subalgebra of $ \mathcal{H}_{P}^{\mathcal{D}} $ generated by $ \mathbb{T} $. We denote by $ f_{n}^{\mathcal{BT}} $ the number of forests of degree $ n $ in $ \mathcal{BT} $, $ t_{n}^{\mathcal{BT}} $ the number of trees of degree $ n $ in $ \mathcal{BT} $ and $ F_{\mathcal{BT}} (x) = \sum_{n \geq 0} f_{n}^{\mathcal{BT}} x^{n} $, $ T_{\mathcal{BT}} (x) = \sum_{n \geq 1} t_{n}^{\mathcal{BT}} x^{n} $ the formal series associated. It is possible to calculate $ F_{\mathcal{BT}} (x) $ and $ T_{\mathcal{BT}} (x) $ :

\begin{prop} \label{serieformelle}
The formal series $ F_{\mathcal{BT}} $ and $ T_{\mathcal{BT}} $ are given by:
\begin{eqnarray*}
F_{\mathcal{BT}} (x) & = & \dfrac{3}{-1+4 \cos^{2} \left( \dfrac{1}{3} \arcsin \left( \sqrt{\dfrac{27x}{4}} \right) \right) } ,\\
T_{\mathcal{BT}} (x) & = & \dfrac{4}{3} \sin^{2} \left( \dfrac{1}{3} \arcsin \left( \sqrt{\dfrac{27x}{4}} \right) \right) .
\end{eqnarray*}
\end{prop}

\begin{proof}
$ \mathcal{BT} $ is freely generated by the trees, therefore
\begin{eqnarray} \label{relation1}
F_{\mathcal{BT}} (x) = \dfrac{1}{1-T_{\mathcal{BT}} (x)} .
\end{eqnarray}
We have the following relations:
\begin{eqnarray*}
t^{\mathcal{BT}}_{1} & = & 1,\\
t^{\mathcal{BT}}_{n} & = & \sum_{k=1}^{n} \sum_{a_{1} + \hdots + a_{k} = n-1} (k+1) t^{\mathcal{BT}}_{a_{1}} \hdots t^{\mathcal{BT}}_{a_{k}} {\rm ~ if ~} n \geq 2.
\end{eqnarray*}
Then
$$ T_{\mathcal{BT}}(x) - x = \sum_{k=1}^{\infty} (k+1)xT_{\mathcal{BT}}(x)^{k} = xF \circ T_{\mathcal{BT}}(x) ,$$
where $ F(h) = \displaystyle\sum_{k=1}^{\infty} (k+1)h^{k} = \dfrac{2h-h^{2}}{(1-h)^{2}} $. We deduce the following equality:
\begin{eqnarray} \label{relation2}
T_{\mathcal{BT}}(x)^{3} - 2 T_{\mathcal{BT}}(x)^{2} + T_{\mathcal{BT}}(x) = x .
\end{eqnarray}
So $ T_{\mathcal{BT}}(x) $ is the inverse for the composition of $ x^{3}-2x^{2}+x $, that is to say
$$ T_{\mathcal{BT}} (x) =  \dfrac{4}{3} \sin^{2} \left( \dfrac{1}{3} \arcsin \left( \sqrt{\dfrac{27x}{4}} \right) \right) .$$
Remark that, with (\ref{relation2}), we obtain an inductive definition of the coefficients $ t_{n} $:
\begin{eqnarray*}
\left\lbrace \begin{array}{l}
t^{\mathcal{BT}}_{1} = 1,\\
t^{\mathcal{BT}}_{n} = 2 \displaystyle\sum_{i,j \geq 1 {\rm ~ and ~} i+j=n} t^{\mathcal{BT}}_{i} t^{\mathcal{BT}}_{j} - \sum_{i,j,k \geq 1 {\rm ~ and ~} i+j+k=n} t^{\mathcal{BT}}_{i} t^{\mathcal{BT}}_{j} t^{\mathcal{BT}}_{k} {\rm ~ if ~} n \geq 2.
\end{array} \right. 
\end{eqnarray*}
We deduce from (\ref{relation1}) the equality
$$ F_{\mathcal{BT}} (x) = \dfrac{3}{-1+4 \cos^{2} \left( \dfrac{1}{3} \arcsin \left( \sqrt{\dfrac{27x}{4}} \right) \right) } .$$
\end{proof}
\\

This gives:

$$\begin{array}{c|c|c|c|c|c|c|c|c|c|c}
n&1&2&3&4&5&6&7&8&9&10\\
\hline t_{n}^{\mathcal{BT}} &1&2&7&30&143&728&3876&21318&120175&690690\\
\hline f_{n}^{\mathcal{BT}} &1&3&12&55&273&1428&7752&43263&246675&1430715
\end{array}$$

These are the sequences A006013 and A001764 in \cite{Sloane}.

\begin{prop} \label{involution} Let $ \dagger : \mathcal{BT} \rightarrow \mathcal{BT} $ be the $ \mathbb{K} $-linear map built by induction as follows : $ 1^{\dagger} = 1 $ and for all $ F , G \in \mathbb{F} $, $ (F G)^{\dagger} = G^{\dagger} F^{\dagger} $ and $ (B(F \otimes G) )^{\dagger} = B(G^{\dagger} \otimes F^{\dagger}) $. Then $ \dagger $ is an involution over $ \mathcal{BT} $.
\end{prop}

\begin{proof}
Let us prove that $ (F^{\dagger})^{\dagger} = F $ for all $ F \in \mathbb{F} $ by induction on the degree $ n $ of $ F $. If $ n = 0 $, $ F = 1 $ and this is obvious. Suppose that $ n \geq 1 $. We have two cases :
\begin{enumerate}
\item If $ F = B(G \otimes H) $ is a tree, with $ G , H \in \mathbb{F} $ such that $ \left| G \right| , \left| H \right| < n $. Then $ (F^{\dagger})^{\dagger} = (B(H^{\dagger} \otimes G^{\dagger}))^{\dagger} = B((G^{\dagger})^{\dagger} \otimes (H^{\dagger})^{\dagger}) = B(G \otimes H) = F $, using the induction hypothesis for $ G $ and $ H $.
\item If $ F = G H $ is a forest, with $ G , H \in \mathbb{F} \setminus \{1 \} $ such that $ \left| G \right| , \left| H \right| < n $. Then $ (F^{\dagger})^{\dagger} = (H^{\dagger} G^{\dagger})^{\dagger} = (G^{\dagger})^{\dagger} (H^{\dagger})^{\dagger} = G H = F $, using again the induction hypothesis for $ G $ and $ H $.
\end{enumerate}
In all cases, $ (F^{\dagger})^{\dagger} = F $.
\end{proof}
\\

{\bf Remark.} {We can rewrite the relations between $ \dagger $ and the product $ m $ and $ B $ as follows:
\begin{eqnarray*}
m \circ (\dagger \otimes \dagger) & = & \dagger \circ m \circ \tau ,\\
B \circ (\dagger \otimes \dagger) & = & \dagger \circ B \circ \tau .
\end{eqnarray*}
where $ \tau $ is the flip defined in the introduction.}
\\

{\bf Examples.} {$$ \begin{array}{|rcl|rcl|rcl|} 
\hline
\addeux{$l$} \tun \tun & \overset{\dagger}{\longrightarrow} & \tun \tun \addeux{$r$} & \adtroisun{$l$}{$l$} \adtroisdeux{$r$}{$l$} & \overset{\dagger}{\longrightarrow} & \adtroisdeux{$l$}{$r$} \adtroisun{$r$}{$r$} & \adquatreun{$l$}{$r$}{$r$} & \overset{\dagger}{\longrightarrow} & \adquatreun{$l$}{$l$}{$r$} \\
\adquatredeux{$r$}{$l$}{$l$} \tun \adtroisun{$r$}{$l$} & \overset{\dagger}{\longrightarrow} & \adtroisun{$r$}{$l$} \tun \adquatretrois{$r$}{$l$}{$r$} & \adquatrequatre{$l$}{$l$}{$r$} & \overset{\dagger}{\longrightarrow} & \adquatrequatre{$r$}{$l$}{$r$} & \addeux{$r$} \addeux{$l$} \adtroisdeux{$r$}{$l$} & \overset{\dagger}{\longrightarrow} & \adtroisdeux{$l$}{$r$} \addeux{$r$} \addeux{$l$} \\
\adtroisun{$r$}{$l$} \addeux{$r$} \addeux{$l$} \adtroisun{$r$}{$l$} & \overset{\dagger}{\longrightarrow} & \adtroisun{$r$}{$l$} \addeux{$r$} \addeux{$l$} \adtroisun{$r$}{$l$} & \tun \addeux{$r$} \adtroisdeux{$l$}{$l$} & \overset{\dagger}{\longrightarrow} & \adtroisdeux{$r$}{$r$} \addeux{$l$} \tun & \adquatrequatre{$l$}{$r$}{$r$} \adquatrecinq{$r$}{$l$}{$r$} & \overset{\dagger}{\longrightarrow} & \adquatrecinq{$l$}{$r$}{$l$} \adquatrequatre{$r$}{$l$}{$l$} \\
\hline
\end{array} $$}

\subsection{Hopf algebra structure on $ \mathcal{BT} $} \label{coproduithopf}

As in the case $ (\mathcal{H}_{P},B_{P}) $ (see section \ref{HP}), $ (\mathcal{BT},B) $ is an initial object in the category of couples $ (A,L) $ where $ A $ is an algebra and $ L : A \otimes A \longrightarrow A $ any $ \mathbb{K} $-linear map:

\begin{theo} \label{univ1}
Let $ A $ be any algebra and let $ L : A \otimes A \rightarrow A $ be a $ \mathbb{K} $-linear map. Then there exists a unique algebra morphism $ \phi : \mathcal{BT} \rightarrow A $ such that $ \phi \circ B = L \circ (\phi \otimes \phi) $.
\end{theo}

\begin{proof}
\textit{Existence.} We define an element $ a_{F} \in A $ for any forest $ F \in \mathcal{BT} $ by induction on the degree of $ F $ as follows:
\begin{enumerate}
\item $ a_{1} = 1_{A} $.
\item If $ F $ is a tree, there exists two forests $ G,H \in \mathcal{BT} $ with $ \left| G \right| + \left| H \right| = \left| F \right| - 1 $ such that $ F = B(G \otimes H) $; then $ a_{F} = L(a_{G} \otimes a_{H}) $.
\item If $ F $ is not a tree, $ F = F_{1} \hdots F_{n} $ with $ F_{i} \in \mathbb{T} $ and $ \left| F_{i} \right| < \left| F \right| $; then $ a_{F} = a_{F_{1}} \hdots a_{F_{n}} $.
\end{enumerate}
Let $ \phi : \mathcal{BT} \rightarrow A $ be the unique linear morphism such that $ \phi(F) = a_{F} $ for any forest $ F $. Given two forests $ F $ and $ G $, let us prove that $ \phi(FG) = \phi(F) \phi(G) $. If $ F = 1 $ or $ G = 1 $, this is trivial. If not, $ F = F_{1} \hdots F_{n} $ and $ G = G_{1} \hdots G_{m} \neq 1 $ and therefore
$$ \phi(FG) = a_{F_{1}} \hdots a_{F_{n}} a_{G_{1}} \hdots a_{G_{m}} = \phi(F) \phi(G) .$$
Therefore $ \phi $ is an algebra morphism. On the other hand, for all forests $ F,G $,
$$ \phi \circ B(F \otimes G) = a_{B(F \otimes G)} = L(a_{F} \otimes a_{G}) = L \circ (\phi \otimes \phi) (F \otimes G) ,$$
therefore $ \phi \circ B = L \circ (\phi \otimes \phi) $.\\

\textit{Uniqueness.} Let $ \psi : \mathcal{BT} \rightarrow A $ be an algebra morphism such that $ \psi \circ B = L \circ (\psi \otimes \psi) $. Let us prove that $ \psi(F) = a_{F} $ for any forest $ F $ by induction on the degree of $ F $. If $ F = 1 $, $ \psi(F) = 1_{A} = a_{1} $. If $ \left| F \right| \geq 1 $, we have two cases:
\begin{enumerate}
\item If $ F $ is a tree then $ F = B(G \otimes H) $ with $ G,H \in \mathbb{F} $ and by induction hypothesis,
$$ \psi(F) = \psi \circ B (G \otimes H) = L \circ (\psi \otimes \psi) (G \otimes H) = L(a_{G} \otimes a_{H}) = a_{F} .$$
\item If $ F $ is not a tree, $ F = F_{1} \hdots F_{n} $ with $ F_{i} \in \mathbb{T} $ and by induction hypothesis,
$$ \psi(F) = \psi(F_{1}) \hdots \psi(F_{n}) = a_{F_{1}} \hdots a_{F_{n}} = a_{F} .$$
\end{enumerate}
So $ \phi(F) = \psi(F) $ for any forest $ F $ and $ \phi = \psi $.
\end{proof}
\\

We now define
\begin{eqnarray*}
\varepsilon : \left\lbrace \begin{array}{rcl}
\mathcal{BT} & \rightarrow & \mathbb{K} \\
F \in \mathbb{F} & \mapsto & \delta_{F,1}.
\end{array} \right. 
\end{eqnarray*}
$ \varepsilon $ is an algebra morphism satisfying $ \varepsilon \circ B = 0 $.

\begin{theo} \label{defideltahopf}
Let $ \Delta : \mathcal{BT} \rightarrow \mathcal{BT} \otimes \mathcal{BT} $ be the unique algebra morphism such that $ \Delta \circ B = L \circ (\Delta \otimes \Delta) $ with
\begin{eqnarray*}
L : \left\lbrace \begin{array}{rcl}
(\mathcal{BT} \otimes \mathcal{BT}) \otimes (\mathcal{BT} \otimes \mathcal{BT}) & \rightarrow & \mathcal{BT} \otimes \mathcal{BT} \\
(x \otimes y) \otimes (z \otimes t) & \mapsto & B(x \otimes z) \otimes \varepsilon(y t) 1 + xz \otimes B(y \otimes t).
\end{array} \right. 
\end{eqnarray*}
Equipped with this coproduct, $ \mathcal{BT} $ is a graded connected Hopf algebra.
\end{theo}

\begin{proof}
We use the Sweedler's notation for $ \mathcal{BT} $: for all $ x \in \mathcal{BT} $, $ \Delta(x) = \displaystyle\sum_{x} x^{(1)} \otimes x^{(2)} $. Then, if $ G,H \in \mathbb{F} $,
\begin{eqnarray*}
\Delta (B(G \otimes H)) & = & \sum_{G,H} B(G^{(1)} \otimes H^{(1)}) \otimes \varepsilon (G^{(2)} H^{(2)}) 1 + \sum_{G,H} G^{(1)} H^{(1)} \otimes B(G^{(2)} \otimes H^{(2)}) \\
& = & B(G \otimes H) \otimes 1 + \sum_{G,H} G^{(1)} H^{(1)} \otimes B(G^{(2)} \otimes H^{(2)}) .
\end{eqnarray*}
Thanks to theorem \ref{univ1}, $ \Delta $ is well defined. To prove that $ \mathcal{BT} $ is a graded connected Hopf algebra, we must prove that $ \Delta $ is coassociative, counitary, and homogeneous of degree $ 0 $.\\

We first show that $ \varepsilon $ is a counit for $ \Delta $. Let us prove that $ (\varepsilon \otimes Id) \circ \Delta (F) = F $ for any forest $ F \in \mathbb{F} $ by induction on the degree of $ F $. If $ F = 1 $,
$$ (\varepsilon \otimes Id) \circ \Delta(1) = \varepsilon(1) 1 = 1 .$$
If $ F $ is a tree, $ F = B(G \otimes H) $ with $ G,H \in \mathbb{F} $. Then,
\begin{eqnarray*}
(\varepsilon \otimes Id) \circ \Delta (F) & = & (\varepsilon \otimes Id) \left( B(G \otimes H) \otimes 1 + \sum_{G,H} G^{(1)} H^{(1)} \otimes B(G^{(2)} \otimes H^{(2)}) \right) \\
& = & \varepsilon \circ B(G \otimes H) \otimes 1 + \sum_{G,H} \varepsilon(G^{(1)} H^{(1)}) B(G^{(2)} \otimes H^{(2)}) \\
& = & 0 + \sum_{G,H} B( \varepsilon(G^{(1)}) G^{(2)} \otimes \varepsilon(H^{(1)}) H^{(2)}) \\
& = & B(G \otimes H) \\
& = & F.
\end{eqnarray*}
If not, $ F = F_{1} \hdots F_{n} $ with $ F_{i} \in \mathbb{T} $. As $ (\varepsilon \otimes Id) \circ \Delta $ is an algebra morphism and using the induction hypothesis,
\begin{eqnarray*}
(\varepsilon \otimes Id) \circ \Delta (F) & = & (\varepsilon \otimes Id) \circ \Delta (F_{1}) \hdots (\varepsilon \otimes Id) \circ \Delta (F_{n}) \\
& = & F_{1} \hdots F_{n} \\
& = & F.
\end{eqnarray*}
\indent Let us show that $ (Id \otimes \varepsilon) \circ \Delta (F) = F $ for any forest $ F \in \mathbb{F} $ by induction. If $ F = 1 $,
$$ (Id \otimes \varepsilon) \circ \Delta (1) = 1 \varepsilon (1) = 1 .$$
If $ F $ is a tree, $ F = B(G \otimes H) $ with $ G,H \in \mathbb{F} $ and then
\begin{eqnarray*}
(Id \otimes \varepsilon) \circ \Delta (F) & = & (Id \otimes \varepsilon) \left( B(G \otimes H) \otimes 1 + \sum_{G,H} G^{(1)} H^{(1)} \otimes B(G^{(2)} \otimes H^{(2)}) \right) \\
& = & B(G \otimes H) + \sum_{G,H} G^{(1)} H^{(1)} \varepsilon \circ B(G^{(2)} \otimes H^{(2)}) \\
& = & B(G \otimes H)\\
& = & F.
\end{eqnarray*}
If not, $ F = F_{1} \hdots F_{n} $ with $ F_{i} \in \mathbb{T} $. As $ (Id \otimes \varepsilon) \circ \Delta $ is an algebra morphism and using the induction hypothesis,
\begin{eqnarray*}
(Id \otimes \varepsilon) \circ \Delta (F) & = & (Id \otimes \varepsilon) \circ \Delta (F_{1}) \hdots (Id \otimes \varepsilon) \circ \Delta (F_{n}) \\
& = & F_{1} \hdots F_{n} \\
& = & F.
\end{eqnarray*}
Therefore $ \varepsilon $ is a counit for $ \Delta $.\\

Let us prove that $ \Delta $ is coassociative. More precisely, we show that $ (\Delta \otimes Id) \circ \Delta(F) = (Id \otimes \Delta) \circ \Delta(F) $ for any forest $ F \in \mathbb{F} $ by induction on the degree of $ F $. If $ F = 1 $ this is obvious. If $ F $ is a tree, $ F = B(G \otimes H) $ with $ G,H \in \mathbb{F} $. Then
\begin{eqnarray*}
(\Delta \otimes Id) \circ \Delta(F) & = & (\Delta \otimes Id) \circ \Delta \circ B(G \otimes H)\\
& = & (\Delta \otimes Id) \left( B(G \otimes H) \otimes 1 + \sum_{G,H} G^{(1)} H^{(1)} \otimes B(G^{(2)} \otimes H^{(2)}) \right) \\
& = & B(G \otimes H) \otimes 1 \otimes 1 + \sum_{G,H} G^{(1)} H^{(1)} \otimes B(G^{(2)} \otimes H^{(2)}) \otimes 1\\
& & + \sum_{G,H,G^{(1)},H^{(1)}} (G^{(1)})^{(1)} (H^{(1)})^{(1)} \otimes (G^{(1)})^{(2)} (H^{(1)})^{(2)} \otimes B(G^{(2)} \otimes H^{(2)}) ,
\end{eqnarray*}
\begin{eqnarray*}
(Id \otimes \Delta) \circ \Delta(F) & = & (Id \otimes \Delta) \circ \Delta \circ B(G \otimes H)\\
& = & (Id \otimes \Delta) \left( B(G \otimes H) \otimes 1 + \sum_{G,H} G^{(1)} H^{(1)} \otimes B(G^{(2)} \otimes H^{(2)}) \right) \\
& = & B(G \otimes H) \otimes 1 \otimes 1 + \sum_{G,H} G^{(1)} H^{(1)} \otimes B(G^{(2)} \otimes H^{(2)}) \otimes 1\\
& & + \sum_{G,H,G^{(2)},H^{(2)}} G^{(1)} H^{(1)} \otimes (G^{(2)})^{(1)} (H^{(2)})^{(1)} \otimes B((G^{(2)})^{(2)} \otimes (H^{(2)})^{(2)}) .
\end{eqnarray*}
We conclude with the coassociativity of $ \Delta $ applied to $ G $ and $ H $. If $ F $ is not a tree, $ F = F_{1} \hdots F_{n} $ with $ F_{i} \in \mathbb{T} $. As $ (\Delta \otimes Id) \circ \Delta $ and $ (Id \otimes \Delta) \circ \Delta $ are algebra morphisms and using the induction hypothesis,
\begin{eqnarray*}
& & (\Delta \otimes Id) \circ \Delta(F) \\
& = & (\Delta \otimes Id) \circ \Delta(F_{1}) \hdots (\Delta \otimes Id) \circ \Delta(F_{n}) \\
& = & \sum_{F_{1} , \hdots , F_{n}} \sum_{F_{1}^{(1)}, \hdots , F_{n}^{(1)}} (F_{1}^{(1)})^{(1)} \hdots (F_{n}^{(1)})^{(1)} \otimes (F_{1}^{(1)})^{(2)} \hdots (F_{n}^{(1)})^{(2)} \otimes F_{1}^{(2)} \hdots F_{n}^{(2)} \\
& = & \sum_{F_{1} , \hdots , F_{n}} \sum_{F_{1}^{(2)}, \hdots , F_{n}^{(2)}} F_{1}^{(1)} \hdots F_{n}^{(1)} \otimes (F_{1}^{(2)})^{(1)} \hdots (F_{n}^{(2)})^{(1)} \otimes (F_{1}^{(2)})^{(2)} \hdots (F_{n}^{(2)})^{(2)} \\
& = & (Id \otimes \Delta) \circ \Delta (F_{1}) \hdots (Id \otimes \Delta) \circ \Delta (F_{n}) \\
& = & (Id \otimes \Delta) \circ \Delta(F) .
\end{eqnarray*}
Therefore $ \Delta $ is coassociative.\\

Let us show that $ \Delta $ is homogeneous of degree 0. Easy induction, using the fact that $ L $ is homogeneous of degree 1. Note that it can also be proved using proposition \ref{coupe}. As it is graded and connected, it had an antipode. This ends the proof.
\end{proof}
\\

{\bf Remark.} {It was shown that for all $ x,y \in \mathcal{BT} $,
\begin{eqnarray*}
\Delta \circ B(x \otimes y) = B(x \otimes y) \otimes 1 + (m \otimes B) \circ \Delta(x \otimes y) ,
\end{eqnarray*}
with $ \Delta (x \otimes y) = (Id \otimes \tau \otimes Id) \circ (\Delta \otimes \Delta) (x \otimes y) $, where $ \tau $ is the flip.
}
\\

\begin{prop} \label{daggerdelta}
For all forests $ F \in \mathcal{BT} $, $ \Delta(F^{\dagger}) = ( \dagger \otimes \dagger) \circ \Delta (F) $.
\end{prop}

\begin{proof}
By induction on the degree $ n $ of $ F $. If $ n = 0 $, $ F = 1 $ and this is obvious. Suppose that $ n \geq 1 $. We have two cases :
\begin{enumerate}
\item If $ F = B(G \otimes H) $ is a tree, with $ G , H \in \mathbb{F} $ such that $ \left| G \right| , \left| H \right| < n $. Then
\begin{eqnarray*}
\Delta(F^{\dagger}) & = & \Delta( B(H^{\dagger} \otimes G^{\dagger}) ) \\
& = & B(H^{\dagger} \otimes G^{\dagger}) \otimes 1 + (m \otimes B) \circ \Delta(H^{\dagger} \otimes G^{\dagger}) \\
& = & B(G \otimes H)^{\dagger} \otimes 1^{\dagger} \\
& & + (m \otimes B) \circ (Id \otimes \tau \otimes Id) \circ (((\dagger \otimes \dagger) \circ \Delta) \otimes ((\dagger \otimes \dagger) \circ \Delta)) (H \otimes G) \\
& = & (\dagger \otimes \dagger) (B(G \otimes H) \otimes 1) \\
& & + ((m \circ (\dagger \otimes \dagger)) \otimes (B \circ (\dagger \otimes \dagger))) \circ (Id \otimes \tau \otimes Id) \circ (\Delta \otimes \Delta) (H \otimes G) \\
& = & (\dagger \otimes \dagger) (B(G \otimes H) \otimes 1) \\
& & + (\dagger \otimes \dagger) \circ (m \otimes B) \circ (\tau \otimes \tau) \circ (Id \otimes \tau \otimes Id) \circ (\Delta \otimes \Delta) (H \otimes G) \\
& = & (\dagger \otimes \dagger) \left( B(G \otimes H) \otimes 1 + (m \otimes B) \circ (Id \otimes \tau \otimes Id) \circ (\Delta \otimes \Delta) (G \otimes H) \right) \\
& = & (\dagger \otimes \dagger) \circ \Delta (F) ,
\end{eqnarray*}
using the induction hypothesis in the third equality.
\item If $ F = G H $ is a forest, with $ G , H \in \mathbb{F} \setminus \{1 \} $ such that $ \left| G \right| , \left| H \right| < n $. Then
\begin{eqnarray*}
\Delta( F^{\dagger}) & = & \Delta (H^{\dagger} G^{\dagger} ) \\
& = & (m \otimes m) \circ (Id \otimes \tau \otimes Id) \circ (\Delta(H^{\dagger}) \otimes \Delta(G^{\dagger} )) \\
& = & (m \otimes m) \circ (Id \otimes \tau \otimes Id) \circ (((\dagger \otimes \dagger) \circ \Delta(H)) \otimes ((\dagger \otimes \dagger) \circ \Delta(G))) \\
& = & ((m \circ (\dagger \otimes \dagger)) \otimes (m \circ (\dagger \otimes \dagger))) \circ (Id \otimes \tau \otimes Id) \circ (\Delta(H) \otimes \Delta(G)) \\
& = & (\dagger \otimes \dagger) \circ  (m \otimes m) \circ (\tau \otimes \tau) \circ (Id \otimes \tau \otimes Id) \circ (\Delta(H) \otimes \Delta(G)) \\
& = & (\dagger \otimes \dagger) \circ  (m \otimes m) \circ (Id \otimes \tau \otimes Id) \circ (\Delta(G) \otimes \Delta(H)) \\
& = & (\dagger \otimes \dagger) \circ \Delta(F) ,
\end{eqnarray*}
using the induction hypothesis in the third equality.
\end{enumerate}
In all cases, $ \Delta(F^{\dagger}) = ( \dagger \otimes \dagger) \circ \Delta (F) $.
\end{proof}
\\

We now give a combinatorial description of this coproduct:

\begin{prop} \label{coupe}
Let $ F \in \mathbb{F} $. Then $ \Delta(F) = F \otimes 1+1\otimes F+\displaystyle\sum_{\boldsymbol{v} \mmodels V(F)} Lea_{\boldsymbol{v}}(F) \otimes Roo_{\boldsymbol{v}}(F) $.
\end{prop}

\begin{proof}
Let $ \Delta' $ be the unique algebra morphism from $ \mathcal{BT} $ to $ \mathcal{BT} \otimes \mathcal{BT} $ defined by the formula of proposition \ref{coupe}. It is easy to show that, if $ F, G \in \mathbb{F} $ :
\begin{eqnarray*}
\left\lbrace \begin{array}{rcl}
\Delta'(1) & = & 1 \otimes 1, \\
\Delta'(B(F \otimes G)) & = & B(F \otimes G) \otimes 1 + (m \otimes B) \circ \Delta' (F \otimes G) .
\end{array} \right. 
\end{eqnarray*}
By unicity in theorem \ref{defideltahopf}, $ \Delta' = \Delta $.
\end{proof}

\begin{theo} \label{propuniv}
Let $ A $ be a Hopf algebra and $ L : A \otimes A \rightarrow A $ a linear morphism such that for all $ a,b \in A $,
$$ \Delta \circ L(a \otimes b) = L(a \otimes b) \otimes 1 + (m \otimes L) \circ \Delta(a \otimes b) .$$
Therefore the unique algebra morphism $ \phi : \mathcal{BT} \rightarrow A $ such that $ \phi \circ B = L \circ (\phi \otimes \phi) $ (theorem \ref{univ1}) is a Hopf algebra morphism.
\end{theo}

\begin{proof}
Let us prove that $ \Delta \circ \phi(F) = (\phi \otimes \phi) \circ \Delta(F) $ for any forest $ F \in \mathcal{BT} $ by induction on the degree of $ F $. If $ F = 1 $, this is obvious. If $ F $ is a tree, $ F = B(G \otimes H) $ with $ G,H \in \mathbb{F} $. Then
\begin{eqnarray*}
\Delta \circ \phi(F) & = & \Delta \circ \phi \circ B(G \otimes H) \\
& = & \Delta \circ L \circ (\phi \otimes \phi) (G \otimes H) \\
& = & L( \phi(G) \otimes \phi(H)) \otimes 1 + \sum_{\phi(G) , \phi(H)} (\phi(G))^{(1)} (\phi(H))^{(1)} \otimes L( (\phi(G))^{(2)} \otimes (\phi(H))^{(2)} ) \\
& = & \phi \circ B (G \otimes H) \otimes \phi(1) + \sum_{G,H} \phi(G^{(1)} H^{(1)}) \otimes \phi \circ B(G^{(2)} \otimes H^{(2)}) \\
& = & (\phi \otimes \phi) \left( B (G \otimes H) \otimes 1 + \sum_{G,H} G^{(1)} H^{(1)} \otimes B(G^{(2)} \otimes H^{(2)}) \right)  \\
& = & (\phi \otimes \phi) \circ \Delta(F),
\end{eqnarray*}
using the induction hypothesis for the fourth equality. If $ F $ is not a tree, $ F = F_{1} \hdots F_{n} $ and then
\begin{eqnarray*}
\Delta \circ \phi(F) & = & \Delta \circ \phi(F_{1}) \hdots \Delta \circ \phi(F_{n}) \\
& = & (\phi \otimes \phi) \circ \Delta(F_{1}) \hdots (\phi \otimes \phi) \circ \Delta(F_{n}) \\
& = & (\phi \otimes \phi) \circ \Delta(F),
\end{eqnarray*}
using the induction hypothesis for the second equality and $ \Delta \circ \phi $ and $ (\phi \otimes \phi) \circ \Delta $ are algebra morphisms for the first and third equality.\\

Let us prove that $ \varepsilon \circ \phi = \varepsilon $. As $ \varepsilon \circ \phi $ and $ \varepsilon $ are algebra morphisms, it suffices to show that $ \varepsilon \circ \phi(F) = \varepsilon(F) $ for any tree $ F \in \mathbb{T} $. If $ F = 1 $, this is obvious. If not, $ F = B(G \otimes H) $ with $ G,H \in \mathbb{F} $. $ \varepsilon (F) = 0 $ and
\begin{eqnarray*}
\varepsilon \circ \phi(F) & = & \varepsilon \circ \phi \circ B(G \otimes H) \\
& = & \varepsilon \circ L \circ (\phi \otimes \phi) (G \otimes H).
\end{eqnarray*}
Let us prove that $ \varepsilon \circ L = 0 $: for all $ a,b \in A $,
\begin{eqnarray*}
\Delta \circ L(a \otimes b) & = & L(a \otimes b) \otimes 1 + \sum_{a,b} a^{(1)} b^{(1)} \otimes L( a^{(2)} \otimes b^{(2)})\\
(\varepsilon \otimes Id) \circ \Delta \circ L(a \otimes b) & = & \varepsilon \circ L(a \otimes b) 1 + \sum_{a,b} \varepsilon(a^{(1)} b^{(1)}) L( a^{(2)} \otimes b^{(2)})\\
& = & \varepsilon \circ L(a \otimes b) + L(a \otimes b) \\
& = & L(a \otimes b)
\end{eqnarray*}
Therefore $ \varepsilon \circ L = 0 $ and $ \varepsilon \circ \phi = \varepsilon $.
\end{proof}

\subsection{A pairing on $ \mathcal{BT} $}

Let $ \gamma : \mathbb{K} [ \mathbb{T} ] \rightarrow \mathbb{K} [\mathbb{F}] \otimes \mathbb{K} [\mathbb{F}] $ be the $ \mathbb{K} $-linear map homogeneous of degree $ -1 $ such that for all $ F ,G \in \mathbb{F} $, $ \gamma (B(F \otimes G)) = (-1)^{\left| G \right|} F \otimes G $. From $ \gamma $, we define the $ \mathbb{K} $-linear map $ \Gamma $ by:
\begin{eqnarray*}
\Gamma : \left\lbrace \begin{array}{rcl}
\mathcal{BT} & \rightarrow & \mathcal{BT} \otimes \mathcal{BT} ,\\
T_{1} \hdots T_{n} & \rightarrow & \Delta (T_{1}) \hdots \Delta (T_{n-1}) \gamma (T_{n}) .
\end{array} \right. 
\end{eqnarray*}

\noindent The following result is immediate :

\begin{lemma} \label{fctgamma}
\begin{enumerate}
\item $ \Gamma $ is homogeneous of degree $ -1 $.
\item For all $ x,y \in \mathcal{BT} $, $ \Gamma (xy) = \Delta (x) \Gamma (y) + \varepsilon (y) \Gamma (x) $.
\end{enumerate}
\end{lemma}

\noindent Let $ \Gamma^{\ast} : \mathcal{BT}^{\circledast} \otimes \mathcal{BT}^{\circledast} \rightarrow \mathcal{BT}^{\circledast} $ be the transpose of $ \Gamma $.

\begin{lemma} \label{interm}
\begin{enumerate}
\item $ \Gamma^{\ast} $ is homogeneous of degree $ 1 $.
\item For all $ f,g \in \mathcal{BT}^{\circledast} $,
$$ \Delta ( \Gamma^{\ast}(f \otimes g) ) = \Gamma^{\ast}(f \otimes g) \otimes 1 + (m \otimes \Gamma^{\ast}) \circ \Delta (f \otimes g) .$$
\end{enumerate}
\end{lemma}

\begin{proof}
\begin{enumerate}
\item Because $ \Gamma $ is homogeneous of degree $ -1 $.
\item Let $ f,g \in \mathcal{BT}^{\circledast} $. For all $ x,y \in \mathcal{BT} $:
\begin{eqnarray*}
\Delta (\Gamma^{\ast} (f \otimes g) ) (x \otimes y) & = & (f \otimes g) ( \Gamma (xy)) \\
& = & (f \otimes g) ( \Delta (x) \Gamma (y) + \varepsilon (y) \Gamma (x) ) \\
& = & \varepsilon (y) (f \otimes g) (\Gamma(x)) + \Delta (f \otimes g) (\Delta (x) \otimes \Gamma (y)) \\
& = & \left( \Gamma^{\ast}(f \otimes g) \otimes 1 + (m \otimes \Gamma^{\ast}) \circ \Delta (f \otimes g) \right) (x \otimes y) .
\end{eqnarray*}
\end{enumerate}
\end{proof}

\begin{prop}
The unique algebra morphism $ \Phi : \mathcal{BT} \rightarrow \mathcal{BT}^{\circledast} $ such that $ \Phi \circ B = \Gamma^{\ast} \circ (\Phi \otimes \Phi) $ is a graded Hopf algebra morphism.
\end{prop}

\begin{proof}
By the universal property (theorem \ref{propuniv}) and the second point of Lemma \ref{interm}, $ \Phi $ is a Hopf algebra morphism. Let us prove that it is homogeneous of degree $ 0 $. Let $ F $ be a forest in $ \mathcal{BT} $ with $ n $ vertices. $ \Phi (F) $ is homogeneous of degree $ n $. If $ n = 0 $, then $ F = 1 $ and $ \Phi (F) = 1 $ is homogeneous of degree $ 0 $. Assume the result is true for all forests with $ k < n $ vertices. If $ F $ is a tree, $ F = B(G \otimes H) $. Then $ \Phi (F) = \Gamma^{\ast} \circ (\Phi(G) \otimes \Phi(H)) $. By the induction hypothesis, $ \Phi(G) \otimes \Phi(H) $ is homogeneous of degree $ n-1 $. As $ \Gamma^{\ast} $ is homogeneous of degree $ 1 $, $ \Phi (F) $ is homogeneous of degree $ n-1+1 = n $. If $ F = T_{1} \hdots T_{k} $, $ k \geq 2 $. Then $ \Phi (F) = \Phi (T_{1}) \hdots \Phi (T_{k}) $ is homogeneous of degree $ \left| T_{1} \right| + \hdots + \left| T_{k} \right| = n $ by the induction hypothesis.
\end{proof}
\\

{\bf Notations.} {For all $ x \in \mathcal{BT} $, we note $ \Delta (x) = \displaystyle\sum_{x} x^{(1)} \otimes x^{(2)} $ and $ \Gamma (x) = \displaystyle\sum_{x} x_{(1)} \otimes x_{(2)} $.}

\begin{theo}
For all $ x,y \in \mathcal{BT} $, we set $ \left\langle  x,y \right\rangle  = \Phi (x)(y) $. Then:
\begin{enumerate}
\item $ \left\langle  1,x \right\rangle = \varepsilon (x) $ for all $ x \in \mathcal{BT} $.
\item $ \left\langle  xy,z \right\rangle = \displaystyle\sum_{z} \left\langle  x,z^{(1)} \right\rangle \left\langle  y,z^{(2)} \right\rangle $ for all $ x,y,z \in \mathcal{BT} $.
\item $ \left\langle  B(x \otimes y) , z \right\rangle = \displaystyle\sum_{z} \left\langle  x,z_{(1)} \right\rangle \left\langle  y,z_{(2)} \right\rangle $ for all $ x,y,z \in \mathcal{BT} $.
\end{enumerate}
Moreover:
\begin{enumerate} \setcounter{enumi}{3}
\item $ \left\langle  -,- \right\rangle $ is symmetric.
\item If $ x $ and $ y $ are homogeneous of different degrees, $ \left\langle  x,y \right\rangle = 0 $.
\item $ \left\langle  S(x),y \right\rangle = \left\langle  x,S(y) \right\rangle $ for all $ x,y \in \mathcal{BT} $.
\end{enumerate}
\end{theo}

{\bf Remark.} {Assertions 1-3 entirely determine $ \left\langle  F,G \right\rangle $ for $ F , G \in \mathbb{F} $ by induction on the degree. Therefore $ \left\langle  -,- \right\rangle $ is the unique pairing satisfying the assertions 1-3.}
\\

\begin{proof}
First, as $ \Phi $ is a Hopf algebra morphism, we have for all $ x,y,z \in \mathcal{BT} $,
\begin{eqnarray*}
\left\langle  1,x \right\rangle = \Phi(1)(x) = \varepsilon(x) = \varepsilon \circ \Phi (x) = \Phi(x)(1) = \left\langle  x,1 \right\rangle ,
\end{eqnarray*}
\begin{eqnarray*}
\sum_{z} \left\langle  x,z^{(1)} \right\rangle \left\langle  y,z^{(2)} \right\rangle & = & \sum_{z} \Phi(x)(z^{(1)}) \Phi(y)(z^{(2)}) \\
& = & (\Phi(x) \Phi(y)) (z) \\
& = & \Phi(xy)(z) \\
& = & \left\langle  xy,z \right\rangle ,\\
\sum_{x} \left\langle  x^{(1)},y \right\rangle \left\langle  x^{(2)},z \right\rangle & = & \sum_{x} \Phi(x^{(1)})(y) \Phi(x^{(2)})(z) \\
& = & \sum_{x} \Phi(x)^{(1)}(y) \Phi(x)^{(2)}(z) \\
& = & \Phi(x)(yz)\\
& = & \left\langle  x,yz \right\rangle
\end{eqnarray*}
and
\begin{eqnarray*}
\left\langle  S(x),y \right\rangle = \Phi(S(x))(y) = S^{\ast}(\Phi(x))(y) = \Phi(x)(S(y)) = \left\langle  x,S(y) \right\rangle .
\end{eqnarray*}
We obtain points 1, 2 and 6.

As $ \Phi $ is homogeneous of degree $ 0 $, this implies point 5.

Let $ x,y,z \in \mathcal{BT} $. Then:
\begin{eqnarray*}
\left\langle  B(x \otimes y) , z \right\rangle & = & \Phi \circ B(x \otimes y) (z) \\
& = & \Gamma^{\ast} \circ (\Phi(x) \otimes \Phi(y)) (z) \\
& = & (\Phi(x) \otimes \Phi(y)) (\Gamma(z))\\
& = & \displaystyle\sum_{z} \left\langle  x,z_{(1)} \right\rangle \left\langle  y,z_{(2)} \right\rangle .
\end{eqnarray*}
We obtain point 3.

It remains to prove point 4, that is to say that $ \left\langle  -,- \right\rangle $ is symmetric. First we show that for all $ x,y,z \in \mathcal{BT} $,
$$ \sum_{x} \left\langle  x_{(1)},y \right\rangle \left\langle  x_{(2)},z \right\rangle = \left\langle  x,B(y \otimes z) \right\rangle .$$
We can suppose by bilinearity that $ x,y $ and $ z $ are three forests. By induction on the degree $ n $ of $ x $. If $ n = 0 $, then $ x = 1 $ and
\begin{eqnarray*} \begin{array}{c}
\displaystyle\sum_{x} \left\langle  x_{(1)},y \right\rangle \left\langle  x_{(2)},z \right\rangle = \left\langle  0,y \right\rangle \left\langle 0,z \right\rangle = 0, \\
\left\langle  x,B(y \otimes z) \right\rangle = \varepsilon (B(y \otimes z)) = 0.
\end{array} \end{eqnarray*}
If $ n = 1 $, then $ x = \tun $ and $ \Gamma( \tun ) = \gamma(\tun) = 1 \otimes 1 $. Then
\begin{eqnarray*}
\displaystyle\sum_{x} \left\langle  x_{(1)},y \right\rangle \left\langle  x_{(2)},z \right\rangle = \left\langle  1,y \right\rangle \left\langle  1,z \right\rangle = \delta_{y,1} \delta_{z,1}
\end{eqnarray*}
\begin{eqnarray*}
\left\langle  x,B(y \otimes z) \right\rangle & = & \left\langle  B(1 \otimes 1),B(y \otimes z) \right\rangle \\
& = & \displaystyle\sum_{B(y \otimes z)} \left\langle  1,B(y \otimes z)_{(1)} \right\rangle \left\langle  1,B(y \otimes z)_{(2)} \right\rangle \\
& = & \delta_{B(y \otimes z), \tun} = \delta_{y,1} \delta_{z,1} .
\end{eqnarray*}
Suppose the result true for every forest $ x $ of degree $ < n $. Two cases are possible:
\begin{enumerate}
\item $ x $ is a tree, $ x = B(x' \otimes x'') $. First, $ \Gamma (x) = (-1)^{\left| x'' \right|} x' \otimes x'' $ so
\begin{eqnarray*}
\displaystyle\sum_{x} \left\langle  x_{(1)},y \right\rangle \left\langle  x_{(2)},z \right\rangle & = & (-1)^{\left| x'' \right|} \left\langle  x',y \right\rangle \left\langle  x'',z \right\rangle \\
& = & \left\lbrace 
\begin{array}{l}
0 \mbox{ if $ deg(x'') \neq deg(z)  $,} \\
(-1)^{d} \left\langle  x',y \right\rangle \left\langle  x'',z \right\rangle \mbox{ if $ d = deg(x'') =deg(z)  $.}
\end{array}
\right. 
\end{eqnarray*}
Moreover,
\begin{eqnarray*}
\left\langle  x,B(y \otimes z) \right\rangle & = & \left\langle  B(x' \otimes x''),B(y \otimes z) \right\rangle \\
& = & \displaystyle\sum_{B(y \otimes z)} \left\langle  x',B(y \otimes z)_{(1)} \right\rangle \left\langle  x'',B(y \otimes z)_{(2)} \right\rangle \\
& = & (-1)^{\left| z \right|} \left\langle  x',y \right\rangle \left\langle  x'',z \right\rangle \\
& = & \left\lbrace 
\begin{array}{l}
0 \mbox{ if $ deg(x'') \neq deg(z)  $,} \\
(-1)^{d} \left\langle  x',y \right\rangle \left\langle  x'',z \right\rangle \mbox{ if $ d = deg(x'') =deg(z)  $.}
\end{array}
\right. 
\end{eqnarray*}
\item $ x $ is a forest with at least two trees. $ x $ can be written $ x = x' x'' $, with the induction hypothesis avalaible for $ x' $ and $ x'' $. Then
\begin{eqnarray*}
\left\langle  x,B(y \otimes z) \right\rangle & = & \left\langle  x'x'',B(y \otimes z) \right\rangle \\
& = & \displaystyle\sum_{B(y \otimes z)} \left\langle  x',B(y \otimes z)^{(1)} \right\rangle \left\langle  x'',B(y \otimes z)^{(2)} \right\rangle \\
& = & \left\langle  x',B(y \otimes z) \right\rangle \left\langle  x'',1 \right\rangle + \displaystyle\sum_{y,z} \left\langle  x',y^{(1)} z^{(1)} \right\rangle \left\langle  x'',B(y^{(2)} \otimes z^{(2)}) \right\rangle \\
& = & \left\langle  x',B(y \otimes z) \right\rangle \varepsilon(x'') \\
& & + \displaystyle\sum_{y,z} \left( \sum_{x'} \left\langle  x'^{(1)},y^{(1)} \right\rangle \left\langle  x'^{(2)},z^{(1)} \right\rangle \right) \left( \sum_{x''} \left\langle  x''_{(1)},y^{(2)} \right\rangle \left\langle  x''_{(2)},z^{(2)} \right\rangle \right) \\
& = & \left\langle  x',B(y \otimes z) \right\rangle \varepsilon(x'') \\
& & + \displaystyle\sum_{x',x''} \left( \sum_{y} \left\langle  x'^{(1)},y^{(1)} \right\rangle \left\langle  x''_{(1)},y^{(2)} \right\rangle \right) \left( \sum_{z} \left\langle  x'^{(2)},z^{(1)} \right\rangle \left\langle  x''_{(2)},z^{(2)} \right\rangle \right) \\
& = & \displaystyle\sum_{x'} \left\langle  x'_{(1)},y \right\rangle \left\langle  x'_{(2)},z \right\rangle \varepsilon(x'') + \displaystyle\sum_{x',x''} \left\langle  x'^{(1)} x''_{(1)} ,y \right\rangle \left\langle  x'^{(2)} x''_{(2)},z \right\rangle \\
& = & \displaystyle\sum_{x} \left\langle x_{(1)} , y \right\rangle \left\langle x_{(2)} , z \right\rangle ,
\end{eqnarray*}
by point 2 of Lemma \ref{fctgamma} for the seventh equality.
\end{enumerate}
So for all $ x,y,z \in \mathcal{BT} $,
$$ \sum_{x} \left\langle  x_{(1)},y \right\rangle \left\langle  x_{(2)},z \right\rangle = \left\langle  x,B(y \otimes z) \right\rangle .$$

Let us prove that $ \left\langle  -,- \right\rangle $ is symmetric. By induction on the degree $ n $ of $ x $. If $ n = 0 $, then $ x = 1 $ and
$$ \left\langle  x,y \right\rangle = \left\langle  1,y \right\rangle = \varepsilon(y) = \left\langle  y,1 \right\rangle = \left\langle  y,x \right\rangle .$$
If $ deg(x) \geq 1 $, two cases are possible:
\begin{enumerate}
\item $ x $ is a tree, $ x = B(x' \otimes x'') $. Then:
\begin{eqnarray*}
\left\langle  x,y \right\rangle & = & \left\langle  B(x' \otimes x'') ,y \right\rangle \\
& = & \displaystyle\sum_{y} \left\langle  x',y_{(1)} \right\rangle \left\langle  x'',y_{(2)} \right\rangle \\
& = & \displaystyle\sum_{y} \left\langle  y_{(1)}, x' \right\rangle \left\langle  y_{(2)} , x'' \right\rangle \\
& = & \left\langle  y,B(x' \otimes x'') \right\rangle \\
& = & \left\langle  y,x \right\rangle .
\end{eqnarray*}
\item $ x $ is a forest with at least two trees. $ x $ can be written $ x = x' x'' $ with $ deg(x') , deg(x'') < deg(x) $. Then:
\begin{eqnarray*}
\left\langle  x,y \right\rangle & = & \left\langle  x'x'',y \right\rangle \\
& = & \displaystyle\sum_{y} \left\langle  x',y^{(1)} \right\rangle \left\langle  x'',y^{(2)} \right\rangle \\
& = & \displaystyle\sum_{y} \left\langle  y^{(1)} , x' \right\rangle \left\langle  y^{(2)} , x'' \right\rangle \\
& = & \left\langle  y,x'x'' \right\rangle \\
& = & \left\langle  y,x \right\rangle .
\end{eqnarray*}
\end{enumerate}
\end{proof}
\\

{\bf Examples.} {Values of the pairing $ \left\langle  -,- \right\rangle $ for forests of degree $ \leq 3 $:
$$
\begin{array}{c|c}
& \tun \\
\hline
\tun & 1
\end{array}
\hspace{1cm}
\begin{array}{c|ccc}
& \tun \tun & \addeux{$ l $} & \addeux{$ r $} \\
\hline
\tun \tun & 2 & 1 & 1 \\
\addeux{$ l $} & 1 & 1 & 0 \\
\addeux{$ r $} & 1 & 0 & -1
\end{array}
$$
$$
\begin{array}{c|cccccccccccc}
& \tun \tun \tun & \tun \addeux{$l $} & \addeux{$ l $} \tun & \tun \addeux{$ r $} & \addeux{$ r $} \tun & \adtroisun{$l$}{$l$} & \adtroisun{$r$}{$l$} & \adtroisun{$r$}{$r$} & \adtroisdeux{$l$}{$l$} & \adtroisdeux{$l$}{$r$} & \adtroisdeux{$r$}{$l$} & \adtroisdeux{$r$}{$r$} \\
\hline
\tun \tun \tun & 6 & 3  & 3&3&3&1&2&1&1&1&1&1 \\
\tun \addeux{$l $} & 3&2&2&1&1&1&1&0&1&1&0&0 \\
\addeux{$ l $} \tun & 3&2&2&1&1&0&1&1&1&0&1&0 \\
\tun \addeux{$ r $} & 3&1&1&0&0&0&-1&-1&0&0&-1&-1 \\
\addeux{$ r $} \tun & 3&1&1&0&0&1&1&1&0&-1&0&-1 \\
\adtroisun{$l$}{$l$} & 1&1&0&0&1&1&0&0&1&1&0&0 \\
\adtroisun{$r$}{$l$} & 2&1&1&-1&1&0&-1&0&0&0&0&0 \\
\adtroisun{$r$}{$r$} & 1&0&1&-1&1&0&0&2&0&0&1&1 \\
\adtroisdeux{$l$}{$l$} & 1&1&1&0&0&1&0&0&1&0&0&0 \\
\adtroisdeux{$l$}{$r$} & 1&1&0&0&-1&1&0&0&0&-1&0&0 \\
\adtroisdeux{$r$}{$l$} & 1&0&1&-1&0&0&0&1&0&0&1&0 \\
\adtroisdeux{$r$}{$r$} & 1&0&0&-1&-1&0&0&1&0&0&0&-1
\end{array}
$$}

{\bf Question.} {It is not difficult to see that $ \left\langle  -,- \right\rangle $ is nondegenerate in degree $ \leq 3 $. We conjoncture that it is nondegenerate for all degrees.}

%% file: algebrebigreffe.tex
\section{Bigraft algebra}

\subsection{Presentation}

\begin{defi} \label{defiBGalgebra}
A bigraft algebra (or $ \mathcal{BG} $-algebra for short) is a $ \mathbb{K} $-vector space $ A $ together with three $ \mathbb{K} $-linear maps $ \ast , \succ, \prec : A \otimes A \rightarrow A $ satisfying the following relations : for all $ x,y,z \in A $,
\begin{eqnarray} \label{definitionalgbig}
\begin{array}{rcl}
(x \ast y) \ast z & = & x \ast (y \ast z),\\
(x \ast y) \succ z & = & x \succ (y \succ z),\\
(x \succ y) \ast z & = & x \succ (y \ast z),\\
(x \prec y) \prec z & = & x \prec (y \ast z),\\
(x \ast y) \prec z & = & x \ast (y \prec z),\\
(x \succ y) \prec z & = & x \succ (y \prec z).
\end{array}
\end{eqnarray}
\end{defi}

In other words, a $ \mathcal{BG} $-algebra $ (A,\ast , \succ, \prec) $ is a $ \mathcal{LG} $-algebra $ (A,\ast ,\succ) $ and a $ \mathcal{RG} $-algebra $ (A,\ast , \prec) $ verifying the so-called entanglement relation $ (x \succ y) \prec z = x \succ (y \prec z) $ for all $ x,y,z \in A $, that is to say $ (A,\succ , \prec ) $ is a $ \mathcal{L} $-algebra (see definition \ref{Lalgebra}). We do not suppose that $ \mathcal{BG} $-algebras have a unit for the product $ \ast $. If $ A $ and $ B $ are two $ \mathcal{BG} $-algebras, a $ \mathcal{BG} $-morphism from $ A $ to $ B $ is a $ \mathbb{K} $-linear map $ f : A \rightarrow B $ such that $ f(x \ast y ) = f(x) \ast f(y) $, $ f(x \succ y) = f(x) \succ f(y) $ and $ f(x \prec y) = f(x) \prec f(y) $ for all $ x,y \in A $. We denote by $ \mathcal{BG} $-alg the category of $ \mathcal{BG} $-algebras.\\

{\bf Remark.} {Let $ (A,\ast , \succ , \prec ) $ be a $ \mathcal{BG} $-algebra. Then $ (A,\ast^{\dagger} , \prec^{\dagger} , \succ^{\dagger}) $ is a $ \mathcal{BG} $-algebra, where $ x \ast^{\dagger} y = y \ast x $, $ x \succ^{\dagger} y = y \prec x $ and $ x \prec^{\dagger} y = y \succ x $ for all $ x,y \in A $.} \\

From definition \ref{defiBGalgebra}, it is clear that the operad $ \mathcal{BG} $ is binary, quadratic, regular and set-theoretic. Specifically, one has

\begin{defi} \label{defioperade}
$ \mathcal{BG} $ is the operad $ \mathcal{P}(E,R) $, where $ E $ is concentrated in degree 2, with $ E(2) = \mathbb{K} m \oplus \mathbb{K} \succ \oplus \mathbb{K} \prec $, and $ R $ is concentrated in degree 3, with $ R(3) \subseteq \mathcal{P}_{E}(3) $ being the subspace generated by
$$ \left\{ \begin{array}{l}
\succ \circ (m, I)-\succ \circ (I,\succ),\\
m \circ (\succ,I)-\succ \circ (I,m),\\
\prec \circ (\prec,I)-\prec \circ (I,m),\\
\prec \circ (m,I)-m \circ (I,\prec),\\
\prec \circ (\succ,I)-\succ \circ (I,\prec),\\
m \circ (m, I)-m \circ (I,m).
\end{array}\right.$$
\end{defi}

{\bf Remark.} {Graphically, the relations defining $ \mathcal{BG} $ can be written in the following way:
$$\bddtroisun{$\succ $}{$m$}{$1$}{$2$}{$3$}-\bddtroisdeux{$\succ $}{$\succ $}{$1$}{$2$}{$3$},\hspace{1cm}
 \bddtroisun{$m$}{$\succ$}{$1$}{$2$}{$3$}-\bddtroisdeux{$\succ$}{$m$}{$1$}{$2$}{$3$},\hspace{1cm}
\bddtroisun{$\prec $}{$\prec$}{$1$}{$2$}{$3$}-\bddtroisdeux{$\prec $}{$m $}{$1$}{$2$}{$3$},$$

$$ \bddtroisun{$\prec$}{$m$}{$1$}{$2$}{$3$}-\bddtroisdeux{$m$}{$\prec$}{$1$}{$2$}{$3$},\hspace{1cm}
\bddtroisun{$\prec $}{$\succ$}{$1$}{$2$}{$3$} - \bddtroisdeux{$\succ $}{$\prec$}{$1$}{$2$}{$3$},\hspace{1cm}
 \bddtroisun{$m$}{$m $}{$1$}{$2$}{$3$}-\bddtroisdeux{$m $}{$m$}{$1$}{$2$}{$3$}.$$
}

{\bf Notation.} {We denote $ \widetilde{\mathcal{BG}} $ the nonsymmetric operad associated with the regular operad $ \mathcal{BG} $.}

\subsection{Free $ \mathcal{BG} $-algebra and dimension of $ \mathcal{BG} $} \label{defisuccprec}

Let $ F, G \in \mathbb{F} \setminus \{1\} $. We put $ F = F_{1} \hdots F_{n} $, $ G = G_{1} \hdots G_{m} $, $ F_{1} = B(F_{1}^{1} \otimes F_{1}^{2}) $ and $ G_{m} = B(G_{m}^{1} \otimes G_{m}^{2}) $. We define :
\begin{eqnarray*}
G \succ F & = & B( G F_{1}^{1} \otimes F_{1}^{2}) F_{2} \hdots F_{n} ,\\
G \prec F & = & G_{1} \hdots G_{m-1} B(G_{m}^{1} \otimes G_{m}^{2} F).
\end{eqnarray*}
In other terms, in the first case, $ G $ is grafted on the root of the first tree of $ F $ on the left with an edge decorated by $ l $. In the second case, $ F $ is grafted on the root of the last tree of $ G $ on the right with an edge decorated by $ r $. In particular, $ F \succ \tun = B(F \otimes 1) $ and $ \tun \prec F = B(1 \otimes F) $.\\

{\bf Examples.} {
$$\begin{array}{|rclcl|rclcl|rclcl|}
\hline \tun \tun &\succ& \tun \adtroisun{$r$}{$l$} &=& \adtroisun{$l$}{$l$} \adtroisun{$r$}{$l$} & \addeux{$r$} &\succ& \tun &=& \adtroisdeux{$l$}{$r$} &
\addeux{$r$} &\succ& \addeux{$r$} &=& \adquatredeux{$r$}{$l$}{$r$} \\
\tun &\succ& \adtroisdeux{$l$}{$r$} \addeux{$r$} &=& \adquatretrois{$l$}{$l$}{$r$} \addeux{$r$} & \tun \tun &\succ& \addeux{$r$} &=& \adquatreun{$l$}{$l$}{$r$} &
\adtroisun{$r$}{$l$} &\succ& \tun \tun &=& \adquatrequatre{$l$}{$l$}{$r$} \tun \\
\tun \tun &\prec& \tun \tun &=& \tun \adtroisun{$r$}{$r$} & \addeux{$r$} \addeux{$l$} &\prec& \addeux{$r$} &=& \addeux{$r$} \adquatretrois{$r$}{$l$}{$r$} &
\tun \tun &\prec& \adtroisdeux{$l$}{$r$} &=& \tun \adquatrecinq{$r$}{$l$}{$r$} \\
\adtroisun{$r$}{$l$} &\prec& \tun &=& \adquatreun{$l$}{$r$}{$r$} & \addeux{$r$} \adtroisdeux{$r$}{$l$} &\prec& \tun &=& \addeux{$r$} \adquatredeux{$r$}{$r$}{$l$} &
\tun &\prec& \adtroisun{$r$}{$l$} &=& \adquatrequatre{$r$}{$l$}{$r$} \\
\hline \end{array}$$
}

We denote by $ \mathcal{M} $ the augmentation ideal of $ \mathcal{BT} $, that is to say the vector space generated by the nonempty forests of $ \mathcal{BT} $. We extend $ \succ , \prec : \mathcal{M} \times \mathcal{M} \rightarrow \mathcal{M} $ by linearity.\\

{\bf Remark.} {$ \succ $ and $ \prec $ are not associative:
\begin{eqnarray*}
\tun \succ (\tun \succ \tun) = \adtroisun{$l$}{$l$} & \neq & (\tun \succ \tun) \succ \tun = \adtroisdeux{$l$}{$l$},\\
\tun \prec (\tun \prec \tun) = \adtroisdeux{$r$}{$r$} & \neq & (\tun \prec \tun) \prec \tun = \adtroisun{$r$}{$r$}.
\end{eqnarray*}}

\begin{prop} \label{deggerprecsucc}
For all forests $ F, G \in \mathcal{M} $, $ (G \succ F)^{\dagger} = F^{\dagger} \prec G^{\dagger} $ and $ (G \prec F)^{\dagger} = F^{\dagger} \succ G^{\dagger} $.
\end{prop}

\begin{proof}
Let $ F, G \in \mathcal{M} $ be two forests. We put $ F = F_{1} \hdots F_{n} $, $ G = G_{1} \hdots G_{m} $, $ F_{1} = B(F_{1}^{1} \otimes F_{1}^{2}) $ and $ G_{m} = B(G_{m}^{1} \otimes G_{m}^{2}) $. Then
\begin{eqnarray*}
(G \succ F)^{\dagger} & = & (B( G F_{1}^{1} \otimes F_{1}^{2}) F_{2} \hdots F_{n})^{\dagger} \\
& = & F_{n}^{\dagger} \hdots F_{2}^{\dagger} B((F_{1}^{2})^{\dagger} \otimes (G F_{1}^{1})^{\dagger}) \\
& = & F_{n}^{\dagger} \hdots F_{2}^{\dagger} B((F_{1}^{2})^{\dagger} \otimes (F_{1}^{1})^{\dagger} G^{\dagger}) \\
& = & \left( F_{n}^{\dagger} \hdots F_{2}^{\dagger} B((F_{1}^{2})^{\dagger} \otimes (F_{1}^{1})^{\dagger})\right)  \prec G^{\dagger} \\
& = & F^{\dagger} \prec G^{\dagger} .
\end{eqnarray*}
Therefore $ (G \succ F)^{\dagger} = F^{\dagger} \prec G^{\dagger} $ or all forests $ F, G \in \mathcal{M} $. Moreover, as $ \dagger $ is an involution over $ \mathcal{M} $ (proposition \ref{involution}), $ (G \prec F)^{\dagger} = ((G^{\dagger})^{\dagger} \prec (F^{\dagger})^{\dagger})^{\dagger} = ((F^{\dagger} \succ G^{\dagger})^{\dagger})^{\dagger} = F^{\dagger} \succ G^{\dagger} $.
\end{proof}

\begin{prop}
For all $ x,y,z \in \mathcal{M} $:
\begin{eqnarray}
(x y) \succ z & = & x \succ (y \succ z), \label{rel1} \\
(x \succ y) z & = & x \succ (y z), \label{rel2} \\
(x \prec y) \prec z & = & x \prec (y z), \label{rel3} \\
(x y) \prec z & = & x (y \prec z), \label{rel4} \\
(x \succ y) \prec z & = & x \succ (y \prec z). \label{rel5}
\end{eqnarray}
\end{prop}

\begin{proof}
We can restrict ourselves to $ x, y, z \in \mathbb{F} \setminus \{1 \} $. (\ref{rel2}) and (\ref{rel4}) are immediate. We put $ x = x_{1} B(x_{2} \otimes x_{3}) $ and $ z = B(z_{1} \otimes z_{2}) z_{3} $ with $ x_{i} , z_{i} \in \mathbb{F} $. Then
\begin{eqnarray*} \begin{array}{c}
x \succ (y \succ z) = x \succ (B(y z_{1} \otimes z_{2}) z_{3}) = B(x y z_{1} \otimes z_{2}) z_{3} = (xy) \succ (B(z_{1} \otimes z_{2}) z_{3}) = (xy) \succ z ,\\
(x \prec y) \prec z = (x_{1} B(x_{2} \otimes x_{3} y) ) \prec z = x_{1} B(x_{2} \otimes x_{3} y z) = (x_{1} B(x_{2} \otimes x_{3})) \prec (yz) = x \prec (yz).
\end{array}
\end{eqnarray*}
So we have proved (\ref{rel1}) and (\ref{rel3}). In order to prove (\ref{rel5}), we must study two cases :
\begin{enumerate}
\item $ y = B(y_{1} \otimes y_{2}) $ is a tree,
\begin{eqnarray*}
(x \succ y) \prec z = B(x y_{1} \otimes y_{2}) \prec z = B(x y_{1} \otimes y_{2} z) = x \succ B(y_{1} \otimes y_{2} z) = x \succ (y \prec z).
\end{eqnarray*}
\item $ y = y_{1} y_{2} $ with $ y_{1} , y_{2} \in \mathbb{F} \setminus \{1\} $,
\begin{eqnarray*}
(x \succ y) \prec z = ( (x \succ y_{1}) y_{2}) \prec z = (x \succ y_{1}) (y_{2} \prec z) = x \succ (y_{1} (y_{2} \prec z)) = x \succ (y \prec z) ,
\end{eqnarray*}
by (\ref{rel2}) and (\ref{rel4}).
\end{enumerate}
\end{proof}
\\

As an immediate consequence, we obtain

\begin{cor} \label{mestbg}
$ \mathcal{M} $ is given a graded $ \mathcal{BG} $-algebra structure by its products $ m $ (the concatenation), $ \succ $ and $ \prec $.
\end{cor}

We now prove that $ \mathcal{M} $ is the free $ \mathcal{BG} $-algebra generated by $ \tun $ and deduce the dimension of $ \mathcal{BG}(n) $ for all $ n $.

\begin{theo} \label{mestbglibre}
$ (\mathcal{M},m,\succ,\prec) $ is the free $ \mathcal{BG} $-algebra generated by $ \tun $.
\end{theo}

\begin{proof}
Let $A$ be a $ \mathcal{BG} $-algebra and let $a \in A$. Let us prove that there exists a unique morphism of $ \mathcal{BG} $-algebras $ \phi :\mathcal{M} \rightarrow A$, such that $ \phi(\tun)=a$. We define $ \phi(F)$ for any nonempty forest $ F \in \mathcal{M} $ inductively on the degree of $ F $ by:
$$\left\{\begin{array}{rcl}
\phi(\tun)&=&a,\\
\phi(F_1 \hdots F_k)&=&\phi(F_1) \hdots \phi(F_k) \mbox{ if } k\geq 2,\\
\phi(F)&=&\left( \phi(F^{1}) \succ a \right) \prec \phi(F^{2}) \mbox{ if } F = B(F^{1} \otimes F^{2}) \mbox{ with } F^{1}, F^{2} \in \mathbb{F}. 
\end{array}\right.$$
As the product of $A$ is associative, this is perfectly defined. This map is linearly extended into a map $\phi: \mathcal{M} \rightarrow A$. Let us show that it is a morphism of $\mathcal{BG}$-algebras. By the second point, $\phi(xy)=\phi(x)\phi(y)$ for any $x,y \in \mathcal{M}$. Let $F,G$ be two nonempty trees. Let us prove that $\phi(F \succ G)=\phi(F) \succ \phi(G)$ and $ \phi (F \prec G) = \phi(F) \prec \phi(G) $. Denote $ F = B(F^{1} \otimes F^{2}) $, $ G = B(G^{1} \otimes G^{2}) $ with $ F^{1}, F^{2} $ and $ G^{1} , G^{2} $ in $ \mathbb{F} $. Then:
\begin{enumerate}
\item For $\phi(F \succ G)=\phi(F) \succ \phi(G)$,
\begin{eqnarray*}
\phi(F \succ G) & = & \phi(B(FG^{1} \otimes G^{2})) \\
& = & \left( \phi(FG^{1}) \succ a \right) \prec \phi(G^{2})\\
& = & \left( \phi(F) \succ \left( \phi(G^{1}) \succ a \right) \right) \prec \phi(G^{2}) \\
& = & \phi(F) \succ \left( \left( \phi(G^{1}) \succ a \right) \prec \phi(G^{2}) \right) \\
& = & \phi(F) \succ \phi(G).
\end{eqnarray*}
\item For $ \phi (F \prec G) = \phi(F) \prec \phi(G) $,
\begin{eqnarray*}
\phi(F \prec G) & = & \phi(B(F^{1} \otimes F^{2} G)) \\
& = & \left( \phi(F^{1}) \succ a \right) \prec \phi(F^{2}G) \\
& = & \phi(F^{1}) \succ \left(  a \prec \phi(F^{2}G) \right) \\
& = & \phi(F^{1}) \succ \left( \left( a \prec \phi(F^{2}) \right) \prec \phi(G) \right) \\
& = & \left( \left( \phi(F^{1}) \succ a \right)  \prec \phi(F^{2}) \right) \prec \phi(G) \\
& = & \phi(F) \prec \phi(G).
\end{eqnarray*}
\end{enumerate}
If $ F,G $ are two nonempty forests, $ F = F_{1} \hdots F_{n} $, $ G = G_{1} \hdots G_{m} $. Then:
\begin{eqnarray*}
\phi(F \succ G) & = & \phi( (F_{1} \succ ( \hdots F_{n-1} \succ (F_{n} \succ G_{1}) \hdots )) G_{2} \hdots G_{m} ) \\
& = & \phi( F_{1} \succ ( \hdots F_{n-1} \succ (F_{n} \succ G_{1}) \hdots )) \phi(G_{2} \hdots G_{m} ) \\
& = & (\phi(F_{1}) \succ ( \hdots \phi(F_{n-1}) \succ (\phi(F_{n}) \succ \phi(G_{1})) \hdots )) \phi(G_{2} \hdots G_{m} ) \\
& = & (\phi(F_{1} \hdots F_{n}) \succ \phi(G_{1})) \phi(G_{2} \hdots G_{m} ) \\
& = & \phi(F) \succ \phi(G),
\end{eqnarray*}
and
\begin{eqnarray*}
\phi (F \prec G) & = & \phi(F_{1} \hdots F_{n-1} ( \hdots (F_{n} \prec G_{1}) \prec G_{2} \hdots ) \prec G_{m}) \\
& = & \phi(F_{1} \hdots F_{n-1}) \phi( (\hdots (F_{n} \prec G_{1}) \prec G_{2} \hdots ) \prec G_{m}) \\
& = & \phi(F_{1} \hdots F_{n-1}) ((\hdots (\phi(F_{n}) \prec \phi(G_{1})) \prec \phi(G_{2}) \hdots ) \prec \phi(G_{m})) \\
& = & \phi(F_{1} \hdots F_{n-1}) ( \phi(F_{n}) \prec \phi(G_{1} \hdots G_{m})) \\
& = & \phi(F) \prec \phi(G).
\end{eqnarray*}
So $\phi$ is a morphism of $\mathcal{BG}$-algebras. \\

Let $\phi':\mathcal{M} \rightarrow A$ be another morphism of $\mathcal{BG}$-algebras such that $\phi'(\tun) = a $. Then for any planar trees $F_1,\hdots,F_k$, $\phi'(F_1 \hdots F_k)=\phi'(F_1)\hdots \phi'(F_k)$. For any forests $F^{1} ,F^{2} \in \mathcal{M} $,
\begin{eqnarray*}
\phi'(B(F^{1} \otimes F^{2})) & = & \phi'( (F^{1} \succ \tun) \prec F^{2}) \\
& = & (\phi'(F^{1}) \succ \phi'(\tun)) \prec \phi'(F^{2}) \\
& = & (\phi'(F^{1}) \succ a) \prec \phi'(F^{2}).
\end{eqnarray*}
So $\phi=\phi'$.
\end{proof}

\begin{cor} \label{isom}
For all $ n \in \mathbb{N}^{\ast} $, $ dim(\widetilde{\mathcal{BG}}(n)) = f^{\mathcal{BT}}_{n} $ and the following map is bijective:
\begin{eqnarray} \label{upsilon}
\Psi : \left\lbrace \begin{array}{rcl}
\widetilde{\mathcal{BG}}(n) & \rightarrow & Vect(\mbox{planar forests } \in \mathcal{BT} \mbox{ of degree } n) \subseteq \mathcal{M} \\
p & \rightarrow & p.(\tun, \hdots , \tun) . 
\end{array}
\right.
\end{eqnarray}
\end{cor}

\begin{proof}
It suffices to show that $ \Psi $ is bijective. $ (\mathcal{M},m,\succ,\prec) $ is generated by $ \tun $ as $ \mathcal{BG} $-algebra (with theorem \ref{mestbglibre}) therefore $ \Psi $ is surjective. $ \Psi $ is injective by the freedom in theorem \ref{mestbglibre}.
\end{proof}
\\

The free $ \mathcal{BG} $-algebra over a vector space $ V $ is the $ \mathcal{BG} $-algebra $ BG(V) $ such that any map from $ V $ to a $ \mathcal{BG} $-algebra $ A $ has a natural extention as a $ \mathcal{BG} $-morphism $ BG(V) \rightarrow A $. In other words the functor $ BG(-) $ is the left adjoint to the forgetful functor from $ \mathcal{BG} $-algebras to vector spaces. Because the operad $ \mathcal{BG} $ is regular, we get the following result :

\begin{prop}
Let $ V $ be a $ \mathbb{K} $-vector space. Then the free $ \mathcal{BG} $-algebra on $ V $ is
$$ BG(V) = \bigoplus_{n \geq 1} \mathbb{K} [\mathbb{F}(n)] \otimes V^{\otimes n} ,$$
equipped with the following binary operations : for all $ F \in \mathbb{F}(n) $, $ G \in \mathbb{F}(m) $, $ v_{1} \otimes \hdots \otimes v_{n} \in V^{\otimes n} $ and $ w_{1} \otimes \hdots \otimes w_{m} \in V^{\otimes m} $,
\begin{eqnarray*}
(F \otimes v_{1} \otimes \hdots \otimes v_{n}) \ast (G \otimes w_{1} \otimes \hdots \otimes w_{m}) & = & (FG \otimes v_{1} \otimes \hdots \otimes v_{n} \otimes w_{1} \otimes \hdots \otimes w_{m}) ,\\
(F \otimes v_{1} \otimes \hdots \otimes v_{n}) \succ (G \otimes w_{1} \otimes \hdots \otimes w_{m}) & = & (F \succ G \otimes v_{1} \otimes \hdots \otimes v_{n} \otimes w_{1} \otimes \hdots \otimes w_{m}) ,\\
(F \otimes v_{1} \otimes \hdots \otimes v_{n}) \prec (G \otimes w_{1} \otimes \hdots \otimes w_{m}) & = & (F \prec G \otimes v_{1} \otimes \hdots \otimes v_{n} \otimes w_{1} \otimes \hdots \otimes w_{m}) .
\end{eqnarray*}
\end{prop}

\subsection{A combinatorial description of the composition}

We identify $ \mathcal{BG} $ with the vector space of nonempty planar forests via corollary \ref{isom}. In other terms, we identify $ F \in \mathbb{F}(n) $ and $ \Psi^{-1}(F) \in \mathcal{BG}(n) $.\\

{\bf Notations.} {In order to distinguish the composition in $ \mathcal{BG} $ and the action of the operad $ \mathcal{BG} $ on $ \mathcal{M} $, we now denote by
\begin{enumerate}
\item $ F \circ (F_{1}, \hdots , F_{n}) $ the composition of $ \mathcal{BG} $.
\item $ F \bullet (F_{1}, \hdots , F_{n}) $ the action of $ \mathcal{BG} $ on $ \mathcal{M} $.
\end{enumerate}}

\vspace{0.3cm}

In the following, we will describe the composition of $ \mathcal{BG} $ in term of forests.

\begin{prop} \label{pourleth1}
\begin{enumerate}
\item $ \tun $ is the unit element of $ \mathcal{BG} $.
\item $ \tun \tun = ~ m $ in $ \mathcal{BG}(2) $. Consequently, in $ \mathcal{BG} $ $ \tun \tun \circ (F,G) = FG $ for all $ F,G \in \mathbb{F} \setminus \{1\} $.
\item $ \addeux{l} = ~ \succ $ in $ \mathcal{BG}(2) $. Consequently, in $ \mathcal{BG} $ $ \addeux{l} \circ (F,G) = F \succ G $ for all $ F,G \in \mathbb{F} \setminus \{1\} $.
\item $ \addeux{r} = ~ \prec $ in $ \mathcal{BG}(2) $. Consequently, in $ \mathcal{BG} $ $ \addeux{r} \circ (F,G) = F \prec G $ for all $ F,G \in \mathbb{F} \setminus \{1\} $.
\end{enumerate}
\end{prop}

\begin{proof}
\begin{enumerate}
\item Indeed, $ \Psi (\tun) = \tun = \Psi (I) $. Hence, $ \tun = I $.
\item By definition, $ \Psi (\tun \tun) = \tun \tun = \Psi (m) $. So $ \tun \tun = m $ in $ \mathcal{BG}(2) $. Moreover, for all $ F,G \in \mathbb{F} \setminus \{1\} $,
\begin{eqnarray*}
\Psi(FG) & = & FG \\
& = & m \bullet (F,G) \\
& = & m \bullet (F \bullet (\tun, \hdots , \tun) , G \bullet (\tun, \hdots ,\tun)) \\
& = & (m \circ (F,G)) \bullet (\tun, \hdots , \tun)\\
& = & \Psi (m \circ (F,G)) .
\end{eqnarray*}
So $ FG = m \circ (F,G) = \tun \tun \circ (F,G) $.
\item Indeed, $ \Psi (\addeux{$l$}) = \tun \succ \tun = \Psi (\succ) $. So $ \addeux{$l$} = \succ $ in $ \mathcal{BG}(2) $. Moreover, for all $ F,G \in \mathbb{F} \setminus \{1\} $,
\begin{eqnarray*}
\Psi(F \succ G) & = & F \succ G \\
& = & \succ \bullet (F,G)\\
& = & \succ \bullet (F \bullet (\tun, \hdots , \tun) , G \bullet (\tun, \hdots ,\tun)) \\
& = & (\succ \circ (F,G)) \bullet (\tun, \hdots \tun)\\
& = & \Psi (\succ \circ (F,G)) .
\end{eqnarray*}
So $ F \succ G = \succ \circ (F,G) = \addeux{$l$} \circ (F,G) $.
\item $ \Psi (\addeux{$r$}) = \tun \prec \tun = \Psi (\prec) $. So $ \addeux{$r$} = \prec $ in $ \mathcal{BG}(2) $. Moreover, for all $ F,G \in \mathbb{F} \setminus \{1\} $,
\begin{eqnarray*}
\Psi(F \prec G) & = & F \prec G \\
& = & \prec \bullet (F,G)\\
& = & \prec \bullet (F \bullet (\tun, \hdots , \tun) , G \bullet (\tun, \hdots ,\tun)) \\
& = & (\prec \circ (F,G)) \bullet (\tun, \hdots \tun) \\
& = & \Psi (\prec \circ (F,G)) .
\end{eqnarray*}
So $ F \prec G = \prec \circ (F,G) = \addeux{$r$} \circ (F,G) $.
\end{enumerate}
\end{proof}

\begin{prop} \label{pourleth2}
Let $ F \in \mathbb{F}(m) $ and $ G \in \mathbb{F}(n) $ with $ m,n \geq 1 $. Let $ H_{1}, \hdots , H_{m+n+1} \in \mathbb{F} \setminus \{1\} $. Then in $ \mathcal{BG} $,
\begin{eqnarray*}
(FG) \circ (H_{1} , \hdots , H_{m + n}) & = & F \circ (H_{1} , \hdots , H_{m}) G \circ (H_{m + 1}, \hdots , H_{m + n}) \\
B(F \otimes G) \circ (H_{1}, \hdots , H_{m+n+1}) & = & \left( (F \circ (H_{1}, \hdots, H_{m})) \succ H_{m+1} \right)\\
& & \prec (G \circ (H_{m+2}, \hdots ,H_{m+n+1})).
\end{eqnarray*}
\end{prop}

\begin{proof}
Indeed, in $ \mathcal{BG} $,
\begin{eqnarray*}
(FG) \circ (H_{1} , \hdots , H_{m + n}) & = & (m \circ (F,G)) \circ (H_{1} , \hdots , H_{m + n})\\
& = & m \circ (F \circ (H_{1} , \hdots , H_{m}) , G \circ (H_{m + 1}, \hdots , H_{m + n})) \\
& = & F \circ (H_{1} , \hdots , H_{m}) G \circ (H_{m + 1}, \hdots , H_{m + n}) ,
\end{eqnarray*}
and
\begin{eqnarray*}
& & B(F \otimes G) \circ (H_{1}, \hdots , H_{m+n+1}) \\
& = & \left( (F \succ \tun) \prec G \right)  \circ (H_{1}, \hdots , H_{m+n+1}) \\
& = & (\addeux{$r$} \circ (\addeux{$l$} \circ (F,\tun),G)) \circ (H_{1}, \hdots , H_{m+n+1}) \\
& = & \addeux{$r$} \circ \left( (\addeux{$l$} \circ (F,\tun)) \circ (H_{1}, \hdots , H_{m+1}) , G \circ (H_{m+2}, \hdots, H_{m+n+1}) \right) \\
& = & \addeux{$r$} \circ \left( \addeux{$l$} \circ (F \circ (H_{1}, \hdots , H_{m}),\tun \circ H_{m+1}) , G \circ (H_{m+2}, \hdots, H_{m+n+1}) \right) \\
& = & \addeux{$r$} \circ \left( (F \circ (H_{1}, \hdots , H_{m}) ) \succ H_{m+1} , G \circ (H_{m+2}, \hdots, H_{m+n+1}) \right) \\
& = & \left( (F \circ (H_{1}, \hdots, H_{m})) \succ H_{m+1} \right) \prec (G \circ (H_{m+2}, \hdots ,H_{m+n+1})).
\end{eqnarray*}
\end{proof}
\\

Combining propositions \ref{pourleth1} and \ref{pourleth2}, we obtain the following theorem :

\begin{theo}
The composition of $ \mathcal{BG} $ in the basis of forests belong to $ \mathbb{F} \setminus \{1\} $ can be inductively defined in this way:
\begin{eqnarray*} \left\lbrace \begin{array}{rcl}
\tun \circ (H) & = & H ,\\
FG \circ (H_{1} , \hdots , H_{\left| F \right| + \left| G \right|}) & = & F \circ (H_{1}, \hdots , H_{\left| F \right|}) G \circ (H_{\left| F \right| +1}, \hdots, H_{\left| F \right| + \left| G \right|}),\\
B(F \otimes G) \circ (H_{1}, \hdots , H_{\left| F \right| + \left| G \right| +1}) & = &  \left( (F \circ (H_{1}, \hdots, H_{\left| F \right|})) \succ H_{\left| F \right| +1} \right)\\
& & \prec (G \circ (H_{\left| F \right| +2}, \hdots ,H_{\left| F \right| + \left| G \right| +1})).
\end{array} \right. 
\end{eqnarray*}
\end{theo}

{\bf Examples.} {Let $ F_{1}, F_{2}, F_{3} \in \mathbb{F} \setminus \{1\} $.
\begin{eqnarray*} \begin{array}{|rcl|rcl|}
\hline \tun \tun \circ (F_{1} , F_{2}) & = & F_{1} F_{2} & \adtroisun{$r$}{$l$} \circ (F_{1},F_{2},F_{3}) & = & (F_{1} \succ F_{2}) \prec F_{3} \\
\addeux{$l$} \circ (F_{1},F_{2}) & = & F_{1} \succ F_{2} &  \adtroisun{$r$}{$r$} \circ (F_{1},F_{2},F_{3}) & = & F_{1} \prec (F_{2} F_{3}) \\
\addeux{$r$} \circ (F_{1},F_{2}) & = & F_{1} \prec F_{2} & \adtroisun{$l$}{$l$} \circ (F_{1},F_{2},F_{3}) & = & (F_{1} F_{2}) \succ F_{3} \\
\tun \addeux{$l$} \circ (F_{1},F_{2},F_{3}) & = & F_{1} (F_{2} \succ F_{3}) & \adtroisdeux{$l$}{$r$} \circ (F_{1},F_{2},F_{3}) & = & (F_{1} \prec F_{2}) \succ F_{3} \\
\addeux{$r$} \tun \circ (F_{1},F_{2},F_{3}) & = & (F_{1} \prec F_{2}) F_{3} & \adtroisdeux{$r$}{$l$} \circ (F_{1},F_{2},F_{3}) & = & F_{1} \prec (F_{2} \succ F_{3}) \\
\tun \tun \tun \circ (F_{1},F_{2},F_{3}) & = & F_{1} F_{2} F_{3} & \adtroisdeux{$l$}{$l$} \circ (F_{1},F_{2},F_{3}) & = & (F_{1} \succ F_{2}) \succ F_{3} \\
\hline \end{array}
\end{eqnarray*}}

\subsection{Relations with the coproduct}

We recall the definition of pre-dendriform algebra (see \cite{Leroux2}):

\begin{defi}
A pre-dendriform algebra is a $ \mathbb{K} $-vector space $ A $ equipped with three binary operations denoted $ \ast $, $ \succ $ and $ \prec $ satisfying the four relations : for all $ x , y, z \in A $,
\begin{eqnarray*}
(x \ast y) \ast z & = & x \ast (y \ast z) ,\\
(x \ast y) \succ y & = & x \succ (y \succ z) , \\
(x \prec y) \prec z & = & x \prec (y \ast z) ,\\
(x \succ y) \prec z & = & x \succ (y \prec z) .
\end{eqnarray*}
In other words, $ (A,\ast, \succ) $ and $ (A, \ast , \prec) $ are dipterous algebras with the entanglement relation $ (x \succ y) \prec z = x \succ (y \prec z) $.
\end{defi}

{\bf Remark.} {We can define $ PreDend $-alg the category of pre-dendriform algebras. A $ \mathcal{BG} $-algebra is also a pre-dendriform algebra. We get the following canonical functor: $ \mathcal{BG} \mbox{-alg} \rightarrow PreDend \mbox{-alg} $.}
\\

As bigraft algebras are not objects with unit, we consider again the extended tensor product $ A \overline{\otimes} B $ for $ A , B $ two $ \mathcal{BG} $-algebras. If $ A $ is a bigraft algebra, we extend $ \succ , \prec : A \otimes A \rightarrow A $ into maps $ \succ , \prec : A \overline{\otimes} A \rightarrow A $ in the following way : for all $ a \in A $,
$$ a \succ 1 = 0 , \hspace{1cm} a \prec 1 = a , \hspace{1cm} 1 \succ a = a , \hspace{1cm} 1 \prec a = 0 . $$

Moreover, we extend the product of $ A $ into a map from $ (A \oplus \mathbb{K}) \otimes (A \oplus \mathbb{K}) $ to $ A \oplus \mathbb{K} $ by putting $ 1.a = a.1 = a $ for all $ a \in A $ and $ 1.1 = 1 $. Note that the relations (\ref{definitionalgbig}) are now satisfied on $ A \overline{\otimes} A \overline{\otimes} A $.

\begin{prop} \label{structenseur1}
Let $ A $ and $ B $ be two bigraft algebras. Then $ A \overline{\otimes} B $ is given a structure of pre-dendriform algebra in the following way : for $ a,a' \in A \cup \mathbb{K} $ and $ b , b' \in B \cup \mathbb{K} $,
\begin{eqnarray*}
(a \otimes b) \ast (a' \otimes b') & = & (a \ast a') \otimes (b \ast b') ,\\
(a \otimes b) \succ (a' \otimes b') & = & (a \ast a') \otimes (b \succ b') , \mbox{ if $ b $ or $ b' \in B $}, \\
(a \otimes 1) \succ (a' \otimes 1)  & = & (a \succ a') \otimes 1 ,\\
(a \otimes b) \prec (a' \otimes b') & = & (a \ast a') \otimes (b \prec b') , \mbox{ if $ b $ or $ b' \in B $}, \\
(a \otimes 1) \prec (a' \otimes 1) & = & (a \prec a') \otimes 1 .
\end{eqnarray*}
\end{prop}

\begin{proof}
With proposition \ref{tenseurdipterous}, $ (A \overline{\otimes} B, \ast , \prec) $ is a $ \mathcal{RG} $-algebra. In the same way, $ (A \overline{\otimes} B, \ast , \succ) $ is a $ \mathcal{LG} $-algebra. It remains to show the entanglement relation: for all $ a,a',a'' \in A $ and $ b,b',b'' \in B $,
\begin{eqnarray*}
((a \otimes b) \succ (a' \otimes b') ) \prec (a'' \otimes b'') & = & ((a \ast a') \ast a'') \otimes ((b \succ b') \prec b'') \\
& = & (a \ast (a' \ast a'')) \otimes (b \succ (b' \prec b'')) \\
& = & (a \otimes b) \succ ((a' \otimes b') \prec (a'' \otimes b'')) .
\end{eqnarray*}
This calculation is still true if $ b $, $ b' $ or $ b'' $ is equal to $ 1 $ or if $ b' = b'' = 1 $ and $ b \in B $. If $ b = b' = 1 $ and $ b'' \in B $,
\begin{eqnarray*}
((a \otimes 1) \succ (a' \otimes 1)) \prec (a'' \otimes b'') & = & ((a \succ a') \ast a'') \otimes (1 \prec b'')\\
& = & 0 ,\\
(a \otimes 1) \succ ((a' \otimes 1) \prec (a'' \otimes b'')) & = & (a \otimes 1) \succ ((a' \ast a'') \otimes (1 \prec b'')) \\
& = & 0.
\end{eqnarray*}
If $ b' = b'' = 1 $ and $ b \in B $, it is the same calculation as previously. Finally if $ b = b' = b'' = 1 $,
\begin{eqnarray*}
((a \otimes 1) \succ (a' \otimes 1)) \prec (a'' \otimes 1) & = & ((a \succ a') \prec a'') \otimes 1 \\
& = & (a \succ (a' \prec a'')) \otimes 1 \\
& = & (a \otimes 1) \succ ((a' \otimes 1) \prec (a'' \otimes 1)).
\end{eqnarray*}
\end{proof}

\begin{prop}
For all forests $ F,G \in \mathcal{M} $ and for all tree $ T \in \mathcal{M} $,
\begin{eqnarray} \label{succprecdelta}
\Delta (F \succ T \prec G) = \Delta(F) \succ \Delta(T) \prec \Delta(G) .
\end{eqnarray}
\end{prop}

\begin{proof}
With proposition \ref{relcopg}, we have $ \Delta(T \prec G) = \Delta(T) \prec \Delta(G) $. Moreover,
\begin{eqnarray*}
\Delta(F \succ T) & = & \Delta( (F^{\dagger})^{\dagger} \succ (T^{\dagger})^{\dagger}) \\
& = & (\dagger \otimes \dagger) \circ \Delta ( T^{\dagger} \prec F^{\dagger} ) \\
& = & (\dagger \otimes \dagger) \circ (\Delta ( T^{\dagger} ) \prec \Delta (F^{\dagger} ) ) \\
& = & (\dagger \otimes \dagger) \circ ( ((\dagger \otimes \dagger) \circ \Delta(T) ) \prec ((\dagger \otimes \dagger) \circ \Delta(F) )) \\
& = & (\dagger \otimes \dagger) \circ (\dagger \otimes \dagger) \circ ( \Delta(F) \succ \Delta(T) ) \\
& = & \Delta(F) \succ \Delta(T) ,
\end{eqnarray*}
using propositions \ref{daggerdelta} and \ref{deggerprecsucc}. As the entanglement relation is satisfied in $ \mathcal{M} $ and $ \mathcal{M} \overline{\otimes} \mathcal{M} $, any parenthesizing of (\ref{succprecdelta}) gives the same relation.
\end{proof}
\\

{\bf Remark.} {As $ \mathcal{M} \overline{\otimes} \mathcal{M} $ is not a $ \mathcal{BG} $-algebra, the relation (\ref{succprecdelta}) is not true if $ T $ is a forest in general.}

%% file: dualite.tex
\section{Koszul duality}

\subsection{Koszul dual of the bigraft operad}

\subsubsection{Presentation}

As $ \mathcal{BG} $ is a binary and quadratic operad, it admits a dual, in the sense of V. Ginzburg and M. Kapranov (see \cite{Kapranov}). We denote by $ \mathcal{BG}^{!} $ the Koszul dual of $ \mathcal{BG} $. We prove that $ \mathcal{BG}^{!} $ is defined as follows:

\begin{theo}
The operad $ \mathcal{BG}^{!} $, dual of $ \mathcal{BG} $, is the quadratic operad $ \mathcal{P}(E,R) $, where $ E $ is concentrated in degree 2, with $ E(2) $ the free $ S_{2} $-module $ \mathbb{K} m \oplus \mathbb{K} \succ \oplus \mathbb{K} \prec $ and $ R $ is concentrated in degree 3, with $ R(3) \subseteq \mathcal{P}_{E}(3) $ the sub-$ S_{3} $-module generated by
\begin{eqnarray} \label{relationbgdual}
\begin{minipage}{5cm}
\begin{center}
$\left\{ \begin{array}{l}
r_{1} = \succ \circ (m, I)-\succ \circ (I,\succ),\\
r_{2} = m \circ (\succ,I)-\succ \circ (I,m),\\
r_{3} = \prec \circ (\prec,I)-\prec \circ (I,m),\\
r_{4} = \prec \circ (m,I)-m \circ (I,\prec),\\
r_{5} = \prec \circ (\succ,I)-\succ \circ (I,\prec),\\
r_{6} = m \circ (m, I)-m \circ (I,m),
\end{array}\right.$
\end{center}
\end{minipage}
\hspace{1cm} and \hspace{0.5cm}
\begin{minipage}{5cm}
\begin{center}
$\left\{ \begin{array}{l}
r_{7} = \succ \circ (\succ,I),\\
r_{8} = \succ \circ (\prec,I),\\
r_{9} = m \circ (\prec,I),\\
r_{10} = \prec \circ (I,\prec),\\
r_{11} = \prec \circ (I,\succ),\\
r_{12} = m \circ (I,\succ).
\end{array}\right.$
\end{center}
\end{minipage}
\end{eqnarray}
\end{theo}

The notations $ r_{1} $ to $ r_{12} $ of the relations are introduced for future references.\\

\begin{proof}
Recall some notations. For any right $ S_{n} $-module $ V $, we denote by $ V^{!} $ the right $ S_{n} $-module $ V^{\ast} \otimes (\varepsilon) $, where $ (\varepsilon) $ is the one-dimensional signature representation. Explicitly, the action of $ S_{n} $ on $ V^{\ast} $ is defined by $ f^{\sigma}(x) = \varepsilon(\sigma) f(x^{\sigma^{-1}}) $ for all $ \sigma \in S_{n} $, $ f \in V^{!} $, $ x \in V $. The pairing between $ V^{!} $ and $ V $ is given by: if $ f \in V^{!} $, $ x \in V $,
\begin{eqnarray} \label{crochetdedualitéop}
\left\langle - , - \right\rangle : V^{!} \otimes V \longrightarrow \mathbb{K} , \left\langle f , x \right\rangle = f(x) ,
\end{eqnarray}
and if $ \sigma \in S_{n} $, $ \left\langle f^{\sigma} , x^{\sigma} \right\rangle = \varepsilon (\sigma) \left\langle f , x \right\rangle $.\\

Moreover, $ \mathcal{BG} = \mathcal{P}(E,R) $, where $ E $ is the free $ S_{2} $-module generated by $ m $, $ \succ $ and $ \prec $ and $ R $ the sub-$ S_{3} $-module of $ \mathcal{P}_{E}(3) $ generated by $ r_{1}, \hdots , r_{6} $. Note that $ dim(E)=6 $, $ dim(\mathcal{P}_{E}(3))=108 $ and $ dim(R)=36 $. $ \mathcal{BG}^{!} = \mathcal{P}(E^{!},R^{\perp}) $ where $ E^{!} = (E(n)^{!})_{n \in \mathbb{N}} $ and $ R^{\perp} $ is the annihilator of $ R $ for the pairing (\ref{crochetdedualitéop}). So $ dim(R^{\perp}) = dim(\mathcal{P}_{E}(3)) - dim(R) = 108 - 36 = 72 $. We then verify that the given relations (\ref{relationbgdual}) for $ \mathcal{BG}^{!} $ are indeed in $ R^{\perp} $, that each of them generates a free $ S_{3} $-module, and finally that there $ S_{3} $-modules are in direct sum. So these relations entirely generate $ R^{\perp} $.
\end{proof}
\\

{\bf Remarks.} {\begin{enumerate}
\item $ \mathcal{BG}^{!} $ is a quotient of $ \mathcal{BG} $.
\item $ \mathcal{BG}^{!} $ is the symmetrization of the nonsymmetric operad $ \widetilde{\mathcal{BG}}^{!} $ generated by $ m , \succ $ and $ \prec $ and satisfying the relations $ r_{1} $ to $ r_{12} $.
\item Graphically, the relations defining $ \mathcal{BG}^{!} $ can be written in the following way :
$$\bddtroisun{$\succ $}{$m$}{$1$}{$2$}{$3$}-\bddtroisdeux{$\succ $}{$\succ $}{$1$}{$2$}{$3$},\hspace{1cm}
 \bddtroisun{$m$}{$\succ$}{$1$}{$2$}{$3$}-\bddtroisdeux{$\succ$}{$m$}{$1$}{$2$}{$3$},\hspace{1cm}
\bddtroisun{$\prec $}{$\prec$}{$1$}{$2$}{$3$}-\bddtroisdeux{$\prec $}{$m $}{$1$}{$2$}{$3$},$$

$$ \bddtroisun{$\prec$}{$m$}{$1$}{$2$}{$3$}-\bddtroisdeux{$m$}{$\prec$}{$1$}{$2$}{$3$},\hspace{1cm}
\bddtroisun{$\prec $}{$\succ$}{$1$}{$2$}{$3$} - \bddtroisdeux{$\succ $}{$\prec$}{$1$}{$2$}{$3$},\hspace{1cm}
\bddtroisun{$m$}{$m $}{$1$}{$2$}{$3$}-\bddtroisdeux{$m $}{$m$}{$1$}{$2$}{$3$},$$

$$ \bddtroisun{$\succ$}{$\succ$}{$1$}{$2$}{$3$} ,\hspace{1cm} \bddtroisun{$\succ$}{$\prec$}{$1$}{$2$}{$3$},\hspace{1cm} \bddtroisun{$m$}{$\prec$}{$1$}{$2$}{$3$},\hspace{1cm} \bddtroisdeux{$\prec$}{$\prec$}{$1$}{$2$}{$3$},\hspace{1cm} \bddtroisdeux{$\prec$}{$\succ$}{$1$}{$2$}{$3$},\hspace{1cm} \bddtroisdeux{$m$}{$\succ$}{$1$}{$2$}{$3$}.$$
\end{enumerate}}

\begin{prop} \label{minorationdim}
\begin{enumerate}
\item $ \mathcal{BG}^{!}(n) $ is generated, as a $ S_{n} $-module, by the following trees:
\begin{eqnarray} \label{tree}
\begin{array}{c}
\begin{picture}(30,120)(-10,0)
\put(50,0){\line(0,0){10}}
\put(50,10){\line(1,1){10}}
\put(50,10){\line(-1,1){10}}
\put(60,23){\tiny $n$}
\put(36,22){.}
\put(34,24){.}
\put(32,26){.}
\put(30,30){\line(-1,1){20}}
\put(30,30){\line(1,1){10}}
\put(20,40){\line(1,1){10}}
\put(32,43){\tiny $k+1$} 
\put(30,53){\tiny $k$}
\put(43,7){\tiny $\prec$}
\put(23,27){\tiny $\prec$}
\put(6,52){.}
\put(4,54){.}
\put(2,56){.}
\put(0,60){\line(-1,1){30}}
\put(0,60){\line(1,1){10}}
\put(-10,70){\line(1,1){10}}
\put(-20,80){\line(1,1){10}}
\put(2,73){\tiny $l+1$}
\put(0,83){\tiny $l$}
\put(-18,93){\tiny $l-1$}
\put(13,37){\tiny $m$}
\put(-7,57){\tiny $m$}
\put(-17,67){\tiny $\succ$}
\put(-27,77){\tiny $m$}
\put(-34,92){.}
\put(-36,94){.}
\put(-38,96){.}
\put(-40,100){\line(-1,1){10}}
\put(-40,100){\line(1,1){10}}
\put(-30,113){\tiny $2$} 
\put(-47,97){\tiny $m$}
\put(-56,113){\tiny $ 1 $}
\end{picture}
\end{array}
\end{eqnarray}
where $ 1 \leq l \leq k \leq n $.
\item For all $ n \in \mathbb{N} $, $ dim(\mathcal{BG}^{!}(n)) \leq \dfrac{n (n+1)!}{2} $.
\end{enumerate}
\end{prop}

\begin{proof}
From relations $ r_{1}, \hdots , r_{6} $ and $ r_{10}, r_{11}, r_{12} $, we obtain that $ \mathcal{BG}^{!}(3) $ is generated by the trees of the following form:
$$ \bddtroisun{$ a $}{$ b $}{$1$}{$2$}{$3$} ,$$
with $ a,b \in \{\prec, \succ , m \} $. Using relations $ r_{7}, r_{8}, r_{9} $,
$$ \bddtroisun{$\succ$}{$\succ$}{$1$}{$2$}{$3$} ,\hspace{1cm} \bddtroisun{$\succ$}{$\prec$}{$1$}{$2$}{$3$},\hspace{1cm} \bddtroisun{$m$}{$\prec$}{$1$}{$2$}{$3$} $$
are eliminated. We deduce that $ \mathcal{BG}^{!}(3) $ is generated by the trees of the following form:
$$ \bddtroisun{$ a $}{$ b $}{$1$}{$2$}{$3$} ,$$
with $ (a,b) \in \{ (\succ,m) , (m,\succ),(m,m),(\prec,\succ), (\prec,m),(\prec , \prec) \} $.\\

Moreover, let us prove that for all $ n \geq 3 $, the following tree is zero :
\begin{eqnarray*}
\begin{array}{c}
\begin{picture}(30,70)(-10,0)
\put(50,0){\line(0,0){10}}
\put(50,10){\line(1,1){10}}
\put(50,10){\line(-1,1){20}}
\put(60,23){\tiny $n$}
\put(40,20){\line(1,1){10}}
\put(46,33){\tiny $n-1$}
\put(43,7){\tiny $\succ$}
\put(33,17){\tiny $m$}
\put(26,32){.}
\put(24,34){.}
\put(22,36){.}
\put(20,40){\line(-1,1){20}}
\put(20,40){\line(1,1){10}}
\put(10,50){\line(1,1){10}}
\put(28,53){\tiny $3$}
\put(18,63){\tiny $2$}
\put(-4,63){\tiny $1$}
\put(13,37){\tiny $m$}
\put(3,47){\tiny $\succ$}
\end{picture}
\end{array}
\end{eqnarray*}
By induction on $ n $. If $ n = 3 $ this tree is zero with the relation $ r_{7} $. If $ n \geq 4 $,
\begin{eqnarray*}
\begin{array}{ccc}
\begin{picture}(30,80)(50,0)
\put(50,0){\line(0,0){10}}
\put(50,10){\line(1,1){10}}
\put(50,10){\line(-1,1){20}}
\put(60,23){\tiny $n$}
\put(40,20){\line(1,1){10}}
\put(46,33){\tiny $n-1$}
\put(43,7){\tiny $\succ$}
\put(33,17){\tiny $m$}
\put(26,32){.}
\put(24,34){.}
\put(22,36){.}
\put(20,40){\line(-1,1){30}}
\put(20,40){\line(1,1){10}}
\put(10,50){\line(1,1){10}}
\put(0,60){\line(1,1){10}}
\put(28,53){\tiny $ 4 $}
\put(18,63){\tiny $ 3 $}
\put(8,73){\tiny $ 2 $}
\put(-14,73){\tiny $ 1 $}
\put(13,37){\tiny $m$}
\put(3,47){\tiny $ m $}
\put(-7,57){\tiny $\succ$}
\end{picture}
& = &
\begin{picture}(30,80)(-10,0)
\put(50,0){\line(0,0){10}}
\put(50,10){\line(1,1){10}}
\put(50,10){\line(-1,1){20}}
\put(60,23){\tiny $n$}
\put(40,20){\line(1,1){10}}
\put(46,33){\tiny $n-1$}
\put(43,7){\tiny $\succ$}
\put(33,17){\tiny $m$}
\put(26,32){.}
\put(24,34){.}
\put(22,36){.}
\put(20,40){\line(-1,1){20}}
\put(20,40){\line(1,1){10}}
\put(10,50){\line(1,1){20}}
\put(20,60){\line(-1,1){10}}
\put(28,53){\tiny $4$}
\put(30,73){\tiny $ 3 $}
\put(6,73){\tiny $2$}
\put(-4,63){\tiny $1$}
\put(10,59){\tiny $ m $}
\put(13,37){\tiny $m$}
\put(3,47){\tiny $\succ$}
\end{picture}
\end{array}
\end{eqnarray*}
with the relation $ r_{2} $ and this tree is zero by induction hypothesis.\\

So $ \mathcal{BG}^{!}(n) $ is generated, as a $ S_{n} $-module, by the $ \frac{n(n+1)}{2} $ trees given by the formula (\ref{tree}). We immediately deduce the second assertion.
\end{proof}

\subsubsection{Free $ \mathcal{BG}^{!} $-algebra and dimension of $ \mathcal{BG}^{!} $} \label{ordresurbg!}

Let $ \mathbb{F}^{!} $ be the subset of forests of $ \mathbb{F} $ containing the empty tree $ 1 $ and satisfying the following conditions : if $ F_{1} \hdots F_{n} \in \mathbb{F}^{!} $ with the $ F_{i} $'s nonempty trees, then
\begin{enumerate}
\item $ F_{1} , \hdots , F_{n} $ are corollas $ \in \mathbb{F} $,
\item if $ \exists ~ e \in E(F_{i}) $ decorated by $ l $, then $ i = 1 $,
\item if $ \exists ~ e \in E(F_{i}) $ decorated by $ r $, then $ i = n $.
\end{enumerate}
We set $ \mathbb{G} = \mathbb{F} \setminus \mathbb{F}^{!} $.\\

{\bf Remark.} {If $ F_{1} \hdots F_{n} \in \mathbb{F}^{!} $ with $ n \geq 2 $, then $ F_{2}, \hdots , F_{n-1} = \tun $, every $ e \in E(F_{1}) $ is decorated by $ l $ and every $ e \in E(F_{n}) $ is decorated by $ r $. In particular, we have $ \dagger (\mathbb{F}^{!}) = \mathbb{F}^{!} $.}
\\

{\bf Examples.} {Forests of $ \mathbb{F}^{!} $ :
\begin{itemize}
\item In degree 1 : $ \tun $.
\item In degree 2 : $ \tun \tun , \addeux{$l$} , \addeux{$r$} $.
\item In degree 3 : $ \tun \tun \tun , \addeux{$l$} \tun , \tun \addeux{$r$} , \adtroisun{$l$}{$l$} , \adtroisun{$r$}{$l$} , \adtroisun{$r$}{$r$} $.
\item In degree 4 : $ \tun \tun \tun \tun , \addeux{$l$} \tun \tun , \tun \tun \addeux{$r$} , \addeux{$l$} \addeux{$r$} , \adtroisun{$l$}{$l$} \tun , \tun \adtroisun{$r$}{$r$} , \adquatreun{$l$}{$l$}{$l$} , \adquatreun{$l$}{$l$}{$r$} , \adquatreun{$l$}{$r$}{$r$} , \adquatreun{$r$}{$r$}{$r$} $.
\end{itemize}}
\vspace{0.5cm}

Let $ \mathcal{BT}^{!} $ be the $ \mathbb{K} $-vector space generated by $ \mathbb{F}^{!} $ and $ \mathcal{M}^{!} $ the $ \mathbb{K} $-vector space generated by $ \mathbb{F}^{!} \setminus \{1 \} $. We denote by $ t_{n}^{\mathcal{BT}^{!}} $ the number of trees of degree $ n $ in $ \mathcal{BT}^{!} $ and $ f_{n}^{\mathcal{BT}^{!}} $ the number of forests of degree $ n $ in $ \mathcal{BT}^{!} $. We put $ T_{\mathcal{BT}^{!}}(x) = \sum_{n \geq 1} t_{n}^{\mathcal{BT}^{!}} x^{n} $ and $ F_{\mathcal{BT}^{!}}(x) = \sum_{n \geq 1} f_{n}^{\mathcal{BT}^{!}} x^{n} $.

\begin{prop} \label{serieformelledual}
The formal series of $ \mathcal{BT}^{!} $ are given by :
$$ T_{\mathcal{BT}^{!}}(x) = \dfrac{x}{(1-x)^{2}} \hspace{0.5cm} \mbox{ and } \hspace{0.5cm} F_{\mathcal{BT}^{!}}(x) = \dfrac{x}{(1-x)^{3}}  .$$
\end{prop}

\begin{proof}
We have $ t_{1}^{\mathcal{BT}^{!}} = 1, f_{1}^{\mathcal{BT}^{!}} = 1 $. For all $ n \geq 2 $, $ t_{n}^{\mathcal{BT}^{!}} = n $ because the corollas of degree $ n $ have $ n-1 $ edges with $ n $ possible different decorations. Moreover, for all $ n \geq 2 $,
\begin{eqnarray*}
f_{n}^{\mathcal{BT}^{!}} = f_{n-1}^{\mathcal{BT}^{!}} - t_{n-1}^{\mathcal{BT}^{!}} + t_{n}^{\mathcal{BT}^{!}} + (n-1)
\end{eqnarray*}
where the term :
\begin{itemize}
\item $ f_{n-1}^{\mathcal{BT}^{!}} - t_{n-1}^{\mathcal{BT}^{!}} $ corresponds to forests of degree $ n $ and of length $ \geq 3 $ obtained from the forests of degree $ n-1 $ and of length $ \geq 2 $ by adding $ \tun $ in the middle. For example, in degree $ 4 $, these are the forests $ \tun \tun \tun \tun , \addeux{$l$} \tun \tun , \tun \tun \addeux{$r$} $.
\item $ t_{n}^{\mathcal{BT}^{!}} $ corresponds to trees of degree $ n $. In degree $ 4 $, these are the trees $ \adquatreun{$l$}{$l$}{$l$} , \adquatreun{$l$}{$l$}{$r$} , \adquatreun{$l$}{$r$}{$r$} , \adquatreun{$r$}{$r$}{$r$} $.
\item $ n-1 $ corresponds to forests of degree $ n $ and of length $ 2 $ obtained by product of a tree of degree $ k $ and a tree of degree $ n-k $, $ 1 \leq k \leq n-1 $. For example, in degree $ 4 $, these are the forests $ \addeux{$l$} \addeux{$r$} , \adtroisun{$l$}{$l$} \tun , \tun \adtroisun{$r$}{$r$} $.
\end{itemize}
So $ f_{n}^{\mathcal{BT}^{!}} = f_{n-1}^{\mathcal{BT}^{!}} + n = \frac{n (n+1)}{2} $. We deduce that $ T_{\mathcal{BT}^{!}}(x) = \frac{x}{(1-x)^{2}} $ and $ F_{\mathcal{BT}^{!}}(x) = \frac{x}{(1-x)^{3}} $.
\end{proof}

\begin{lemma}
Let $ F , G \in \mathbb{F} $ be two nonempty forests such that $ F \in \mathbb{G} $ or $ G \in \mathbb{G} $. Then $ F G , F \succ G $ and $ F \prec G \in \mathbb{G} $.
\end{lemma}

\begin{proof}
Suppose that $ F \notin \mathbb{F}^{!} $. We have two cases :
\begin{enumerate}
\item If $ F $ is not a monomial of corollas. Then $ h(F) \geq 2 $. So $ h(F \bullet G) , h(G \bullet F) \geq 2 $ and $ F \bullet G \in \mathbb{G} , G \bullet F \in \mathbb{G} $ for all $ \bullet \in \{ \ast , \succ , \prec \} $.
\item If $ F $ is a monomial of corollas. As $ F \notin \mathbb{F}^{!} $, $ h(F) \geq 1 $ and $ F = F_{1} F_{2} $ with $ F_{1} , F_{2} $ nonempty such that $  \exists ~ e \in E(F_{2}) $ decorated by $ l $ (this is the same argument with $ r $). Then $ F G , G F \in \mathbb{G} $, $ G \succ F = (G \succ F_{1}) F_{2} \in \mathbb{G} $ and $ F \prec G = F_{1} (F_{2} \prec G) \in \mathbb{G} $. Moreover $ h(F \succ G) , h(G \prec F) \geq 2 $ and $ F \succ G , G \prec F \in \mathbb{G} $. 
\end{enumerate}
\end{proof}
\\

In other words, $ \mathbb{K}[\mathbb{G}] $ is a $ \mathcal{BG} $-ideal of the $ \mathcal{BG} $-algebra $ \mathcal{M} $. So the quotient vector space $ \mathcal{M}^{!} = \mathcal{M} / \mathbb{K}[\mathbb{G}] $ is a $ \mathcal{BG} $-algebra. In the sequel, we shall identify a forest $ F \in \mathcal{M} $ and its class in $ \mathcal{M}^{!} $.

\begin{prop} \label{M!estbgdual}
$ (\mathcal{M}^{!},\ast , \succ , \prec ) $ is a $ \mathcal{BG}^{!} $-algebra generated by $ \tun $.
\end{prop}

\begin{proof}
We already have that $ \mathcal{M}^{!} $ is a $ \mathcal{BG} $-algebra. Let us prove relations $ r7 $ to $ r12 $. Let $ F,G,H \in \mathcal{M} $ be three nonempty forests. $ h((F \succ G) \succ H ) , h((F \prec G) \succ H) , h(F \prec (G \prec H)) $ and $ h(F \prec (G \succ H)) \geq 3 $. Therefore $ (F \succ G) \succ H , (F \prec G) \succ H ,F \prec (G \prec H) , F \prec (G \succ H) \in \mathbb{G} $ and the relations $ r7 , r8 , r10 ,r11 $ are satisfied in $ \mathcal{M}^{!} $. Moreover $ (F \prec G) H , F (G \succ H) \in \mathbb{G} $ by considering the decorations of the edges and $ r9 , r12 $ are true in $ \mathcal{M}^{!} $. So $ (\mathcal{M}^{!},\ast , \succ , \prec ) $ is a $ \mathcal{BG}^{!} $-algebra. As $ \mathcal{M} $ is generated as $ \mathcal{BG} $-algebra by $ \tun $ (with theorem \ref{mestbglibre}), $ \mathcal{M}^{!} $ is also generated by $ \tun $ as $ \mathcal{BG} $-algebra and therefore as $ \mathcal{BG}^{!} $-algebra.
\end{proof}

\begin{theo}
\begin{enumerate}
\item $ \mathcal{M}^{!} $ is the free $ \mathcal{BG}^{!} $-algebra generated by $ \tun $.
\item For all $ n \in \mathbb{N}^{\ast} $, $ \mathcal{BG}^{!}(n) = \frac{n (n+1)!}{2} $.
\item For all $ n \in \mathbb{N}^{\ast} $, $ \mathcal{BG}^{!}(n) $ is freely generated, as a $ S_{n} $-module, by the trees given by the formula (\ref{tree}).
\end{enumerate}
\end{theo}

\begin{proof}
Let $ n \in \mathbb{N}^{\ast} $. Consider the following application :
\begin{eqnarray*}
\Omega : \left\lbrace \begin{array}{rcl}
\widetilde{\mathcal{BG}}^{!}(n) & \rightarrow & Vect(\mbox{planar forests } \in \mathcal{BT}^{!} \mbox{ of degree } n) \subseteq \mathcal{M}^{!} \\
p & \rightarrow & p.(\tun, \hdots , \tun) .
\end{array} \right. 
\end{eqnarray*}
By proposition \ref{M!estbgdual}, $ \Omega $ is surjective. So $ dim(\mathcal{BG}^{!}(n)) \geq f_{n}^{\mathcal{BT}^{!}} n! = \frac{n (n+1)!}{2} $. From proposition \ref{minorationdim}, we deduce that $ dim(\mathcal{BG}^{!}(n)) = \frac{n (n+1)!}{2} $. So $ \Omega $ is bijective and $ \mathcal{M}^{!} $ is the free $ \mathcal{BG}^{!} $-algebra generated by $ \tun $. Moreover $ \mathcal{BG}^{!}(n) $ is freely generated, as a $ S_{n} $-module, by the trees given by the formula (\ref{tree}), by equality of dimensions.
\end{proof}
\\

We give some numerical values:

$$\begin{array}{c|c|c|c|c|c|c|c|c|c|c}
n&1&2&3&4&5&6&7&8&9&10\\
\hline dim \left( \widetilde{\mathcal{BG}}^{!}(n) \right)  &1&3&6&10&15&21&28&36&45&55 \\
\hline dim \left( \mathcal{BG}^{!}(n) \right) &1&6&36&240&1800&15120&141120&1451520&16329600&199584000
\end{array}$$

These are the sequences A000217 and A001286 in \cite{Sloane}.\\

{\bf Remark.} {We denote $ F_{\mathcal{BG}}(x) $ and $ F_{\mathcal{BG}^{!}}(x) $ the formal series associated to operads $ \mathcal{BG} $ and $ \mathcal{BG}^{!} $. Using proposition \ref{serieformelledual}, $ F_{\mathcal{BG}^{!}}(x) = \frac{x}{(1-x)^{3}} $. Moreover, by proposition \ref{serieformelle} (see the proof), $ F_{\mathcal{BG}}(x) = \frac{T(x)}{1-T(x)} $ where $ T(x) = (x-2x^{2}+x^{3})^{-1} $. So $ F_{\mathcal{BG}}^{-1}(x) = \frac{x}{(1+x)^{3}} $ and we have :
$$ F_{\mathcal{BG}} \left( - F_{\mathcal{BG}^{!}}(-x) \right) = x .$$
This result is also a consequence of theorem \ref{bgestkoszul}.}
\\

To conclude this section, we introduce an order relation on the set of the vertices of a forest $ \in \mathbb{F}^{!} $. Let $ F = F_{1} \hdots F_{n} \in \mathbb{F}^{!} \setminus \{1\} $ and $ s , s' $ be two vertices of $ F $. Then $ s \leq s' $ if one of these assertions is satisfied :
\begin{enumerate}
\item $ s \in V(F_{i}) $, $ s' \in V(F_{j}) $ and $ i < j $.
\item $ s,s' \in V(F_{i}) $ with $ F_{i} = B(G_{1} \hdots G_{p} \otimes H_{1} \hdots H_{q}) $ (the $ G_{k} $'s and the $ H_{k} $'s are equal to $ \tun $) and :
\begin{enumerate}
\item $ s \in V(G_{k}) $, $ s' \in V(G_{l}) $ and $ k < l $.
\item $ s \in V(G_{k}) $ and $ s' $ is the root of $ F_{i} $.
\item $ s $ is the root of $ F_{i} $ and $ s' \in V(H_{k}) $.
\item $ s \in V(H_{k}) $, $ s' \in V(H_{l}) $ and $ k < l $.
\end{enumerate}
\end{enumerate}

{\bf Examples.} {We give the order relation on the vertices for the following forests:
$$ \sadquatreun{$l$}{$r$}{$r$}{$2$}{$4$}{$3$}{$1$} , \hspace{1cm} \sadquatreun{$l$}{$l$}{$r$}{$3$}{$4$}{$2$}{$1$} , \hspace{1cm} \sadtroisun{$l$}{$l$}{$3$}{$2$}{$1$} \tdun{$4$} \tdun{$5$} \saddeux{$r$}{$7$}{$6$} , \hspace{1cm} \sadquatreun{$l$}{$l$}{$l$}{$4$}{$3$}{$2$}{$1$} \tdun{$5$} \tdun{$6$} \tdun{$7$} \sadtroisun{$r$}{$r$}{$10$}{$9$}{$8$} .$$
}

\subsection{Homology of a $ \mathcal{BG} $-algebra} \label{homologie}

We denote by $ BG^{!}(V) $ the free $ \mathcal{BG}^{!} $-algebra over a vector space $ V $. Because the operad $ \mathcal{BG}^{!} $ is regular, we get the following result :

\begin{prop}
Let $ V $ be a $ \mathbb{K} $-vector space. Then the free $ \mathcal{BG}^{!} $-algebra on $ V $ is
$$ BG^{!}(V) = \bigoplus_{n \geq 1} \mathbb{K} [\mathbb{F}^{!}(n)] \otimes V^{\otimes n} ,$$
equipped with the following binary operations : for all $ F \in \mathbb{F}^{!}(n) $, $ G \in \mathbb{F}^{!}(m) $, $ v_{1} \otimes \hdots \otimes v_{n} \in V^{\otimes n} $ and $ w_{1} \otimes \hdots \otimes w_{m} \in V^{\otimes m} $,
\begin{eqnarray*}
(F \otimes v_{1} \otimes \hdots \otimes v_{n}) \ast (G \otimes w_{1} \otimes \hdots \otimes w_{m}) & = & (FG \otimes v_{1} \otimes \hdots \otimes v_{n} \otimes w_{1} \otimes \hdots \otimes w_{m}) ,\\
(F \otimes v_{1} \otimes \hdots \otimes v_{n}) \succ (G \otimes w_{1} \otimes \hdots \otimes w_{m}) & = & (F \succ G \otimes v_{1} \otimes \hdots \otimes v_{n} \otimes w_{1} \otimes \hdots \otimes w_{m}) ,\\
(F \otimes v_{1} \otimes \hdots \otimes v_{n}) \prec (G \otimes w_{1} \otimes \hdots \otimes w_{m}) & = & (F \prec G \otimes v_{1} \otimes \hdots \otimes v_{n} \otimes w_{1} \otimes \hdots \otimes w_{m}) .
\end{eqnarray*}
\end{prop}

{\bf Remark.} {We can see an element $ (F \otimes v_{1} \otimes \hdots \otimes v_{n}) \in BG^{!}(V) $ as the forest $ F $ where the vertex $ i $ (for the order relation introduced at the end of the section \ref{ordresurbg!}) is decorated by $ v_{i} $.}
\\

By taking the elements of $ V $ homogenous of degree $ 1 $, $ BG^{!}(V) $ is naturaly graduated. We define by duality on $ BG^{!}(V) $ three coproducts given in the following way: if $ F = F_{1} \hdots F_{k} \in \mathbb{F}^{!}(n) $ and $ v = v_{1} \otimes \hdots \otimes v_{n} \in V^{\otimes n} $, with
\begin{itemize}
\item if $ k = 1 $,
\begin{eqnarray*}
F = B(\underbrace{\tun \hdots \tun}_{p \times} \otimes \underbrace{\tun \hdots \tun}_{q \times}) ,
\end{eqnarray*}
\item if $ k \geq 2 $,
\begin{eqnarray*}
F_{1} = B(\underbrace{\tun \hdots \tun}_{p \times} \otimes 1) , F_{k} = B(1 \otimes \underbrace{\tun \hdots \tun}_{q \times}) \mbox{ and for all } 2 \leq i \leq k-1, F_{i} = \tun ,
\end{eqnarray*}
\end{itemize}
($ p + q + k = n $) then
\begin{eqnarray*}
\Delta(F \otimes v) & = & \sum_{i=0}^{k} (F_{1} \hdots F_{i} \otimes v_{1} \otimes \hdots \otimes v_{p+i} ) \otimes (F_{i+1} \hdots F_{k} \otimes v_{p+i+1} \otimes \hdots \otimes v_{n}) ,\\
\Delta_{\succ}(F \otimes v) & = & \sum_{i=0}^{p} (\underbrace{\tun \hdots \tun}_{i \times} \otimes v_{1} \otimes \hdots \otimes v_{i}) \otimes (B(\underbrace{\tun \hdots \tun}_{p-i \times} \otimes 1) F_{2} \hdots F_{k} \otimes v_{i+1} \otimes \hdots \otimes v_{n}) ,\\
\Delta_{\prec}(F \otimes v) & = & \sum_{i=0}^{q} (F_{1} \hdots F_{k-1} B(1 \otimes \underbrace{\tun \hdots \tun}_{q-i \times}) \otimes v_{1} \otimes \hdots \otimes v_{n-i}) \otimes ( \underbrace{\tun \hdots \tun}_{i \times} \otimes v_{n-i+1} \otimes \hdots \otimes v_{n}) .
\end{eqnarray*}

Let $ (A, \ast , \succ ,\prec) $ be a $ \mathcal{BG} $-algebra. We define a differential $ d : BG^{!}(A)(n) \rightarrow BG^{!}(A)(n-1) $ uniquely determined by the following conditions :
\begin{enumerate}
\item For all $ a \in A $, $ d( \tun \otimes a) = 0 $.
\item For all $ a,b \in A $, $ d(\tun \tun \otimes a \otimes b) = a \ast b $, $ d(\addeux{$l$} \otimes a \otimes b) = a \succ b $ and $ d(\addeux{$r$} \otimes a \otimes b) = a \prec b $.
\item Let $ \theta : BG^{!}(A) \rightarrow BG^{!}(A) $ be the following map:
\begin{eqnarray*}
\theta : \left\lbrace \begin{array}{rcl}
BG^{!}(A) & \rightarrow & BG^{!}(A) \\
x & \rightarrow & (-1)^{\mbox{degree}(x)} x \mbox{ for all homogeneous } x .
\end{array} \right. 
\end{eqnarray*}
Then $ d $ is a $ \theta $-coderivation : for all $ x \in BG^{!}(A) $,
\begin{eqnarray*}
\Delta (d(x)) & = & (d \otimes id + \theta \otimes d) \circ \Delta (x) ,\\
\Delta_{\succ} (d(x)) & = & (d \otimes id + \theta \otimes d) \circ \Delta_{\succ} (x) ,\\
\Delta_{\prec} (d(x)) & = & (d \otimes id + \theta \otimes d) \circ \Delta_{\prec} (x) .
\end{eqnarray*}
\end{enumerate}

So, $ d $ is the map which sends the element $ (B(\underbrace{\tun \hdots \tun}_{p \times} \otimes 1) \underbrace{\tun \hdots \tun}_{k-2 \times} B(1 \otimes \underbrace{\tun \hdots \tun}_{q \times}) \otimes v_{1} \otimes \hdots \otimes v_{n}) $, where $ p , q , k \in \mathbb{N} $, $ k \geq 1 $ and $ p+q+k = n $ (if $ k = 1 $, the element is $ (B(\underbrace{\tun \hdots \tun}_{p \times} \otimes \underbrace{\tun \hdots \tun}_{q \times}) \otimes v_{1} \otimes \hdots \otimes v_{n}) $), to
\begin{eqnarray*}
& & \sum_{i=1}^{p-1} (-1)^{i-1} (B(\underbrace{\tun \hdots \tun}_{p-1 \times} \otimes 1) \underbrace{\tun \hdots \tun}_{k-2 \times} B(1 \otimes \underbrace{\tun \hdots \tun}_{q \times}) \otimes v_{1} \otimes \hdots \otimes v_{i} \ast v_{i+1} \otimes \hdots \otimes v_{n}) \\
& + & (-1)^{p-1} (B(\underbrace{\tun \hdots \tun}_{p-1 \times} \otimes 1) \underbrace{\tun \hdots \tun}_{k-2 \times} B(1 \otimes \underbrace{\tun \hdots \tun}_{q \times}) \otimes v_{1} \otimes \hdots \otimes v_{p} \succ v_{p+1} \otimes \hdots \otimes v_{n}) \\
& + & \sum_{i=p+1}^{p+k-1} (-1)^{i-1} (B(\underbrace{\tun \hdots \tun}_{p \times} \otimes 1) \underbrace{\tun \hdots \tun}_{k-3 \times} B(1 \otimes \underbrace{\tun \hdots \tun}_{q \times}) \otimes v_{1} \otimes \hdots \otimes v_{i} \ast v_{i+1} \otimes \hdots \otimes v_{n}) \\
& + & (-1)^{p+k-1} (B(\underbrace{\tun \hdots \tun}_{p \times} \otimes 1) \underbrace{\tun \hdots \tun}_{k-2 \times} B(1 \otimes \underbrace{\tun \hdots \tun}_{q-1 \times}) \otimes v_{1} \otimes \hdots \otimes v_{p+k} \prec v_{p+k+1} \otimes \hdots \otimes v_{n}) \\
& + & \sum_{i = p+k+1}^{n-1} (-1)^{i-1} (B(\underbrace{\tun \hdots \tun}_{p \times} \otimes 1) \underbrace{\tun \hdots \tun}_{k-2 \times} B(1 \otimes \underbrace{\tun \hdots \tun}_{q-1 \times}) \otimes v_{1} \otimes \hdots \otimes v_{i} \ast v_{i+1} \otimes \hdots \otimes v_{n}) .
\end{eqnarray*}
The homology of this complex will be denoted by $ H_{\ast}(A) $. More clearly, for all $ n \in \mathbb{N} $ :
\begin{eqnarray*}
H_{n}(A) = \dfrac{Ker \left( d_{\left| BG^{!}(A)(n+1) \right. } \right)}{Im \left( d_{\left| BG^{!}(A)(n+2) \right.} \right)}
\end{eqnarray*}

{\bf Examples.} {Let $ a,b,c \in A $. Then $ d(\tun \otimes a) = 0 $, $ d(\tun \tun \otimes a \otimes b) = a \ast b $, $ d(\addeux{$l$} \otimes a \otimes b) = a \succ b $ and $ d(\addeux{$r$} \otimes a \otimes b) = a \prec b $. In degree 3,
\begin{eqnarray*}
d(\tun \tun \tun \otimes a \otimes b \otimes c) & = & - (\tun \tun \otimes a \otimes (b \ast c)) + (\tun \tun \otimes (a \ast b) \otimes c) \\
d(\addeux{$l$} \tun \otimes a \otimes b \otimes c) & = & (\tun \tun \otimes (a \succ b) \otimes c) - (\addeux{$l$} \otimes a \otimes (b \ast c)) \\
d(\tun \addeux{$r$} \otimes a \otimes b \otimes c) & = & (\addeux{$r$} \otimes (a \ast b) \otimes c) - (\tun \tun \otimes a \otimes (b \prec c)) \\
d(\adtroisun{$l$}{$l$} \otimes a \otimes b \otimes c) & = & - (\addeux{$l$} \otimes a \otimes (b \succ c)) + (\addeux{$l$} \otimes (a \ast b) \otimes c) \\
d(\adtroisun{$r$}{$l$} \otimes a \otimes b \otimes c) & = & - (\addeux{$l$} \otimes a \otimes (b \prec c)) + (\addeux{$r$} \otimes (a \succ b) \otimes c)\\
d(\adtroisun{$r$}{$r$} \otimes a \otimes b \otimes c) & = & - (\addeux{$r$} \otimes a \otimes (b \ast c)) + (\addeux{$r$} \otimes (a \prec b) \otimes c) \\
\end{eqnarray*}
So, we obtain
\begin{eqnarray*}
d^{2}(\tun \tun \tun \otimes a \otimes b \otimes c) & = & - a \ast (b \ast c) + (a \ast b) \ast c \\
d^{2}(\addeux{$l$} \tun \otimes a \otimes b \otimes c) & = & (a \succ b) \ast c - a \succ (b \ast c) \\
d^{2}(\tun \addeux{$r$} \otimes a \otimes b \otimes c) & = & (a \ast b) \prec c - a \ast (b \prec c) \\
d^{2}(\adtroisun{$l$}{$l$} \otimes a \otimes b \otimes c) & = & - a \succ (b \succ c) + (a \ast b) \succ c \\
d^{2}(\adtroisun{$r$}{$l$} \otimes a \otimes b \otimes c) & = & - a \succ (b \prec c) + (a \succ c) \prec c \\
d^{2}(\adtroisun{$r$}{$r$} \otimes a \otimes b \otimes c) & = & - a \prec (b \ast c) + (a \prec b) \prec c \\
\end{eqnarray*}
Hence, the nullity of $ d^{2} $ is equivalent to the six relations defining a $ \mathcal{BG} $-algebra (see formula \ref{definitionalgbig}). In particular :
\begin{eqnarray*}
H_{0}(A) = \dfrac{A}{A \ast A + A \succ A + A \prec A}
\end{eqnarray*}}

\subsection{The bigraft operad is Koszul}

In this section, we use the rewriting method described in \cite{LodayV}, see also \cite{Dotsenko,Hoffbeck}.\\

We consider $ \widetilde{\mathcal{BG}}^{!} = \mathcal{P}(E,R) $ the nonsymmetric operad associated to $ \mathcal{BG}^{!} $, where $ E $ is concentrated in degree 2 with $ E(2) = \mathbb{K} m \oplus \mathbb{K} \succ \oplus \mathbb{K} \prec $ and $ R $ is concentrated in degree 3 with $ R(3) \subseteq \mathcal{P}_{E}(3) $ the subspace generated by $ r_{i} $, for all $ i \in \{ 1,\hdots,12 \} $, defined in formula (\ref{relationbgdual}).\\

We use the lexicographical order and we set $ \succ ~ < m < ~ \prec $. So, we can give the leading term for each $ r_{i} $:
$$\begin{array}{c|c|c|c|c|c}
r_{1}&r_{2}&r_{3}&r_{4}&r_{5}&r_{6}\\
\hline \succ \circ (m, I) & m \circ (\succ,I) & \prec \circ (\prec,I) & \prec \circ (m,I) & \prec \circ (\succ,I) & m \circ (m, I)
\end{array}$$
$$\begin{array}{c|c|c|c|c|c}
r_{7}&r_{8}&r_{9}&r_{10}&r_{11}&r_{12}\\
\hline \succ \circ (\succ,I) & \succ \circ (\prec,I) &m \circ (\prec,I)&\prec \circ (I,\prec)&\prec \circ (I,\succ)&m \circ (I,\succ)
\end{array}$$

Observe that such a relation gives rise to a rewriting rule in the operad $ \widetilde{\mathcal{BG}}^{!} $:
\begin{eqnarray*} \begin{array}{|rcl|rcl|rcl|rcl|}
\hline \succ \circ (m,I) &\mapsto & \succ \circ (I,\succ) & \prec \circ (m,I) &\mapsto & m \circ (I,\prec) & \succ \circ (\succ ,I) &\mapsto & 0 & \prec \circ (I, \prec ) &\mapsto & 0 \\
m \circ (\succ , I) &\mapsto & \succ \circ (I,m) & \prec \circ (\succ , I) & \mapsto & \succ \circ (I, \prec ) & \succ \circ (\prec,I) & \mapsto & 0 & \prec \circ (I,\succ) & \mapsto & 0\\
\prec \circ (\prec , I) & \mapsto & \prec \circ (I,m) & m \circ (m,I) & \mapsto & m \circ (I, m) & m \circ (\prec,I) & \mapsto & 0 & m \circ (I,\succ) & \mapsto & 0\\
\hline \end{array} \end{eqnarray*}

Then the following proposition holds:

\begin{prop}
All critical monomials are confluent in the nonsymmetric operad $ \widetilde{\mathcal{BG}}^{!} $.
\end{prop}

\begin{proof}
We show that every critical monomial is confluent. We have $ 11 $ criticals monomials:
$$ \bdgg{$\prec$}{$\prec$}{$\prec$} , \hspace{0.5cm} \bdgg{$\prec$}{$\prec$}{$m$} , \hspace{0.5cm} \bdgg{$\prec$}{$\prec$}{$\succ$} , \hspace{0.5cm} \bdgg{$\prec$}{$m$}{$m$} , \hspace{0.5cm} \bdgg{$\prec$}{$m$}{$\succ$} , \hspace{0.5cm} \bdgg{$\prec$}{$\succ$}{$m$} , $$
$$ \bdgg{$m$}{$m$}{$m$} , \hspace{0.5cm} \bdgg{$m$}{$\succ$}{$m$} , \hspace{0.5cm} \bdgg{$\succ$}{$m$}{$m$} , \hspace{0.5cm} \bdgg{$m$}{$m$}{$\succ$} , \hspace{0.5cm} \bdgg{$\succ$}{$m$}{$\succ$} . $$
We present for example the confluent graph for the last two criticals monomials:
\begin{eqnarray*}
\begin{array}{ccc}
\xymatrix{& \bdgg{$m$}{$m$}{$\succ$} \ar[ld]_{r_{2}} \ar[rd]^{r_{6}} & \\
\bdgd{$m$}{$\succ$}{$m$} \ar[d]_{r_{2}} & & \bdm{$m$}{$\succ$}{$m$} \ar[d]^{r_{2}} \\
\bddg{$\succ$}{$m$}{$m$} \ar[rr]_{r_{6}} & & \bddd{$\succ$}{$m$}{$m$} }
& &
\xymatrix{& \bdgg{$\succ$}{$m$}{$\succ$} \ar[ld]_{r_{2}} \ar[rd]^{r_{1}} & \\
\bdgd{$\succ$}{$\succ$}{$m$} \ar[rd]_{r_{7}} & & \bdm{$\succ$}{$\succ$}{$\succ$} \ar[ld]^{r_{7}} \\
& 0 &}
\end{array}
\end{eqnarray*}
We can construct in the same way the confluent graph for each of the nine other criticals monomials.
\end{proof}
\\

We can give the quadratic part of a PBW basis of $ \mathcal{BG}^{!} $:
$$ \bddtroisdeux{$\succ$}{$\succ$}{$1$}{$2$}{$3$} , \hspace{1cm} \bddtroisdeux{$\succ$}{$m$}{$1$}{$2$}{$3$} , \hspace{1cm} \bddtroisdeux{$\succ$}{$\prec$}{$1$}{$2$}{$3$} , \hspace{1cm} \bddtroisdeux{$m$}{$m$}{$1$}{$2$}{$3$} , \hspace{1cm} \bddtroisdeux{$m$}{$\prec$}{$1$}{$2$}{$3$} , \hspace{1cm} \bddtroisdeux{$\prec$}{$m$}{$1$}{$2$}{$3$} . $$
The quadratic part of a PBW basis of its dual $ \left( \mathcal{BG}^{!} \right)^{!} = \mathcal{BG} $ is:
$$ \bddtroisun{$\succ $}{$m$}{$1$}{$2$}{$3$}, \hspace{1cm} \bddtroisun{$m$}{$\succ$}{$1$}{$2$}{$3$} , \hspace{1cm} \bddtroisun{$\prec $}{$\prec$}{$1$}{$2$}{$3$}, \hspace{1cm} \bddtroisun{$\prec$}{$m$}{$1$}{$2$}{$3$} , \hspace{1cm} \bddtroisun{$\prec $}{$\succ$}{$1$}{$2$}{$3$} , \hspace{1cm} \bddtroisun{$m$}{$m $}{$1$}{$2$}{$3$} , $$
$$ \bddtroisun{$\succ$}{$\succ$}{$1$}{$2$}{$3$} ,\hspace{1cm} \bddtroisun{$\succ$}{$\prec$}{$1$}{$2$}{$3$},\hspace{1cm} \bddtroisun{$m$}{$\prec$}{$1$}{$2$}{$3$},\hspace{1cm} \bddtroisdeux{$\prec$}{$\prec$}{$1$}{$2$}{$3$},\hspace{1cm} \bddtroisdeux{$\prec$}{$\succ$}{$1$}{$2$}{$3$},\hspace{1cm} \bddtroisdeux{$m$}{$\succ$}{$1$}{$2$}{$3$}.$$

We immediately deduce the following theorem:

\begin{theo} \label{bgestkoszul}
$ \mathcal{BG} $ is Koszul.
\end{theo}

Recall that $ H_{\ast}(.) $ is the homology of the complex defined in the section \ref{homologie}. Then we deduce of the previous theorem the following result:

\begin{cor}
Let $ N \geq 1 $ and $ (A, \ast , \succ ,\prec) $ be the free $ \mathcal{BG} $-algebra generated by $ N $ elements. Then $ H_{0}(A) $ is $ N $-dimensional and if $ n \geq 1 $, $ H_{n}(A) = (0) $.
\end{cor}

%% file: theoderigidite.tex
\section{A rigidity theorem}

\subsection{$ \mathcal{BG} $-infinitesimal bialgebra}

\begin{defi}
A $ 2 $-dipterous algebra is a $ \mathbb{K} $-vector space $ A $ equipped with three binary operations denoted $ \ast $, $ \succ $ and $ \prec $ satisfying the following relations : for all $ x , y, z \in A $,
\begin{eqnarray*}
(x \ast y) \ast z & = & x \ast (y \ast z) ,\\
(x \ast y) \succ y & = & x \succ (y \succ z) , \\
(x \prec y) \prec z & = & x \prec (y \ast z) .
\end{eqnarray*}
In other words, $ (A,\ast, \succ) $ and $ (A, \ast , \prec) $ are respectively left and right dipterous algebras.
\end{defi}

{\bf Remark.} {We can define $ 2Dipt $-alg the category of $ 2 $-dipterous algebras. A pre-dendriform algebra is also a $ 2 $-dipterous algebra. We get the following canonical functors: $ \mathcal{BG} \mbox{-alg} \rightarrow PreDend \mbox{-alg} \rightarrow 2Dipt \mbox{-alg} $.}
\\

We give in the following proposition another definition of $ \succ $ and $ \prec $, to be compared with proposition \ref{structenseur1} (and we will use this definition in the following) :

\begin{prop} Let $ A $ and $ B $ be two bigraft algebras. Then $ A \overline{\otimes} B $ is given a structure of $ 2 $-dipterous algebra in the following way : for $ a,a' \in A \cup \mathbb{K} $ and $ b , b' \in B \cup \mathbb{K} $,
\begin{eqnarray*}
(a \otimes b) \ast (a' \otimes b') & = & (a \ast a') \otimes (b \ast b') ,\\
(a \otimes b) \succ (a' \otimes b') & = & (a \succ a') \otimes (b \ast b') , \mbox{ if $ a $ or $ a' \in A $}, \\
(1 \otimes b) \succ (1 \otimes b')  & = & 1 \otimes (b \succ b') ,\\
(a \otimes b) \prec (a' \otimes b') & = & (a \ast a') \otimes (b \prec b') , \mbox{ if $ b $ or $ b' \in B $}, \\
(a \otimes 1) \prec (a' \otimes 1) & = & (a \prec a') \otimes 1 .
\end{eqnarray*}
\end{prop}

\begin{proof}
With proposition \ref{tenseurdipterous}, $ (A \overline{\otimes} B, \ast , \prec) $ is a right dipterous algebra. Applying the same reasoning to $ (A \overline{\otimes} B, \ast , \succ) $, we prove that this is a left dipterous algebra. So $ (A \overline{\otimes} B, \ast , \succ , \prec) $ is a $ 2 $-dipterous algebra.
\end{proof}
\\

{\bf Remark.} {The entanglement relation is not true in general : for $ a,a'' \in A $ and $ b,b',b'' \in B $,
\begin{eqnarray*}
((a \otimes b) \succ (1 \otimes b')) \prec (a'' \otimes b'') & = & ((a \succ 1) \otimes (b \ast b')) \prec (a'' \otimes b'') = 0 \\
(a \otimes b) \succ ((1 \otimes b') \prec (a'' \otimes b'')) & = & (a \succ a'') \otimes (b \ast (b' \prec b'')) .
\end{eqnarray*}}

We define another coproduct $ \Delta_{\mathcal{A}ss} $ on $ \mathcal{BT} $ (the deconcatenation) in the following way: for all $ F \in \mathbb{F} $,
$$ \Delta_{\mathcal{A}ss} (F) = \sum_{F_{1},F_{2} \in \mathbb{F} , F_{1} F_{2} = F} F_{1} \otimes F_{2} .$$
Remark that $ \Delta_{\mathcal{A}ss} (F^{\dagger}) = \tau \circ (\dagger \otimes \dagger) \circ \Delta_{\mathcal{A}ss} (F) $ for all $ F \in \mathbb{F} $.\\

We now have defined three products, namely $ m $, $ \succ $ and $ \prec $ and one coproduct, namely $ \tdelta_{\mathcal{A}ss} $, on $ \mathcal{M} $, obtained from $ \Delta_{\mathcal{A}ss} $ by substracting its primitive parts. The following properties sum up the different compatibilities.

\begin{prop} \label{infibialg}
For all $ x,y \in \mathcal{M} $;
\begin{eqnarray} \label{relpourMinf}
\left\lbrace \begin{array}{rcl}
\tdelta_{\mathcal{A}ss} (xy) & = & (x \otimes 1) \tdelta_{\mathcal{A}ss} (y) + \tdelta_{\mathcal{A}ss} (x) (1 \otimes y) + x \otimes y ,\\
\tdelta_{\mathcal{A}ss} (x \succ y) & = & (x \otimes 1) \succ \tdelta_{\mathcal{A}ss}(y), \\
\tdelta_{\mathcal{A}ss} (x \prec y) & = & \tdelta_{\mathcal{A}ss}(x) \prec (1 \otimes y) .
\end{array} \right. 
\end{eqnarray}
\end{prop}

\begin{proof}
We can restrict ourselves to $ F , G \in \mathbb{F} \setminus \{1\} $. We put $ F = F_{1} \hdots F_{n} $, $ G = G_{1} \hdots G_{m} $ where the $ F_{i} $'s and the $ G_{i} $'s are trees and $ G_{1} = B(G^{1}_{1} \otimes G^{2}_{1}) $. Hence :
\begin{eqnarray*}
\tdelta_{\mathcal{A}ss} (F G) & = & \sum_{H_{1},H_{2} \in \mathbb{F} \setminus \{ 1 \} , H_{1} H_{2} = F G} H_{1} \otimes H_{2} \\
& = & \sum_{H_{1},H_{2} \in \mathbb{F} \setminus \{ 1 \} , H_{1} H_{2} = G} F H_{1} \otimes H_{2} + \sum_{H_{1},H_{2} \in \mathbb{F} \setminus \{ 1 \} , H_{1} H_{2} = F} H_{1} \otimes H_{2} G + F \otimes G \\
& = & (F \otimes 1) \tdelta_{\mathcal{A}ss} (G) + \tdelta_{\mathcal{A}ss}(F) (1 \otimes G) + F \otimes G ,
\end{eqnarray*}
\begin{eqnarray*}
\tdelta_{\mathcal{A}ss} (F \succ G) & = & \tdelta_{\mathcal{A}ss} ( B(F G^{1}_{1} \otimes G^{2}_{1}) G_{2} \hdots G_{m} ) \\
& = & B(F G^{1}_{1} \otimes G^{2}_{1}) \otimes G_{2} \hdots G_{m} + \sum_{i=2}^{m-1} B(F G^{1}_{1} \otimes G^{2}_{1}) G_{2} \hdots G_{i} \otimes G_{i+1} \hdots G_{m} \\
& = & F \succ G_{1} \otimes G_{2} \hdots G_{m} + \sum_{i=2}^{m-1} F \succ G_{1} G_{2} \hdots G_{i} \otimes G_{i+1} \hdots G_{m} \\
& = & (F \otimes 1) \succ \tdelta_{\mathcal{A}ss} (G) .
\end{eqnarray*}
We deduce the third relation of (\ref{relpourMinf}) from the second one in this way :
\begin{eqnarray*}
\tdelta_{\mathcal{A}ss} (F \prec G) & = & \tdelta_{\mathcal{A}ss} ((G^{\dagger} \succ F^{\dagger})^{\dagger}) \\
& = & \tau \circ (\dagger \otimes \dagger) \circ \tdelta_{\mathcal{A}ss} (G^{\dagger} \succ F^{\dagger}) \\
& = & \tau \circ (\dagger \otimes \dagger) ((G^{\dagger} \otimes 1) \succ \tdelta_{\mathcal{A}ss} (F^{\dagger}) ) \\
& = & \tdelta_{\mathcal{A}ss} (F^{\dagger}) \prec (1 \otimes G) .
\end{eqnarray*}
\end{proof}

This justifies the following definition :

\begin{defi}
A $ \mathcal{BG} $-infinitesimal bialgebra is a family $ (A,m,\succ,\prec,\tdelta_{\mathcal{A}ss} ) $ where $ m , \succ , \prec : A \otimes A \rightarrow A $, $ \tdelta_{\mathcal{A}ss} : A \rightarrow A \otimes A $, with the following compatibilities :
\begin{enumerate}
\item $ (A,m,\succ , \prec) $ is a $ \mathcal{BG} $-algebra.
\item For all $ x,y \in A $ :
\begin{eqnarray} \label{lesrelations}
\left\lbrace \begin{array}{rcl}
\tdelta_{\mathcal{A}ss} (xy) & = & (x \otimes 1) \tdelta_{\mathcal{A}ss} (y) + \tdelta_{\mathcal{A}ss} (x) (1 \otimes y) + x \otimes y ,\\
\tdelta_{\mathcal{A}ss} (x \succ y) & = & (x \otimes 1) \succ \tdelta_{\mathcal{A}ss}(y), \\
\tdelta_{\mathcal{A}ss} (x \prec y) & = & \tdelta_{\mathcal{A}ss}(x) \prec (1 \otimes y) .
\end{array} \right. 
\end{eqnarray}
\end{enumerate}
\end{defi}

With proposition \ref{infibialg} and corollary \ref{mestbg}, we have immediately :

\begin{prop}
$ (\mathcal{M},m,\succ,\prec,\tdelta_{\mathcal{A}ss}) $ is a $ \mathcal{BG} $-infinitesimal bialgebra.
\end{prop}

{\bf Notations.} {If $ A $ is a $ \mathcal{BG} $-infinitesimal bialgebra, we denote $ Prim(A) = Ker(\tdelta_{\mathcal{A}ss}) $. In the $ \mathcal{BG} $-infinitesimal bialgebra $ \mathcal{M} $, $ Prim(\mathcal{M}) = \mathbb{K}[\mathbb{T} \setminus \{1\} ] $ and we denote by $ \mathcal{P} $ the primitive part of $ \mathcal{M} $.}
\\

Recall the definition of a $ \mathcal{L} $-algebra introduced by P. Leroux in \cite{Leroux} :

\begin{defi} \label{Lalgebra}
A $ \mathcal{L} $-algebra is a $ \mathbb{K} $-vector space $ A $ equipped with two binary operations $ \succ , \prec : A \otimes A \rightarrow A $ verifying the entanglement relation:
\begin{eqnarray*}
(x \succ y) \prec z = x \succ (y \prec z) ,
\end{eqnarray*}
for all $ x,y,z \in A $.
\end{defi}

The operad $ \mathcal{L} $ is binary, quadratic, regular and set-theoretic. We denote by $ \widetilde{\mathcal{L}} $ the nonsymmetric operad associated with the regular operad $ \mathcal{L} $. We do not suppose that $ \mathcal{L} $-algebras have unit for $ \succ $ or $ \prec $. If $ A $ and $ B $ are two $ \mathcal{L} $-algebras, a $ \mathcal{L} $-morphism from $ A $ to $ B $ is a $ \mathbb{K} $-linear map $ f : A \rightarrow B $ such that $ f(x \succ y) = f(x) \succ f(y) $ and $ f(x \prec y) = f(x) \prec f(y) $ for all $ x,y \in A $. We denote by $ \mathcal{L} $-alg the category of $ \mathcal{L} $-algebras.

\begin{prop} \label{proputile}
For any $ \mathcal{BG} $-infinitesimal bialgebra, its primitive part is a $ \mathcal{L} $-algebra.
\end{prop}

\begin{proof}
Let $ A $ be a $ \mathcal{BG} $-infinitesimal bialgebra. We put $ x,y \in Prim(A) $. Then with (\ref{lesrelations})
\begin{eqnarray*}
\tdelta_{\mathcal{A}ss} (x \succ y) & = & (x \otimes 1) \succ \tdelta_{\mathcal{A}ss}(y) = 0 ,\\
\tdelta_{\mathcal{A}ss} (x \prec y) & = & \tdelta_{\mathcal{A}ss}(x) \prec (1 \otimes y) = 0 .
\end{eqnarray*}
Therefore, $ x \succ y , x \prec y \in Prim(A) $. As we always have the relation $ (x \succ y) \prec z = x \succ (y \prec z) $ for all $ x,y,z \in Prim(A) $, $ Prim(A) $ is a $ \mathcal{L} $-algebra.
\end{proof}

\begin{cor}
$ (\mathcal{P} , \succ , \prec) $ is a $ \mathcal{L} $-algebra.
\end{cor}

We have even more than this :

\begin{theo} \label{Llibre}
$ (\mathcal{P} , \succ , \prec) $ is the free $ \mathcal{L} $-algebra generated by $ \tun $.
\end{theo}

\begin{proof}
Let $ A $ be a $ \mathcal{L} $-algebra and $ a \in A $. Let us prove that there exists a unique morphism of $ \mathcal{L} $-algebras $ \phi : \mathcal{P} \rightarrow A $ such that $ \phi (\tun) = a $. We define $ \phi (F) $ for any nonempty tree $ F \in \mathcal{P} $ inductively on the degree of $ F $ by:
$$\left\{\begin{array}{rcl}
\phi(\tun)&=&a,\\
\phi(F)&=& ( \hdots ((\phi(F^{1}_{1}) \succ ( \hdots (\phi(F^{1}_{p-1}) \succ (\phi(F^{1}_{p}) \succ a )) \hdots )) \prec \phi(F^{2}_{1}) ) \hdots ) \prec \phi(F^{2}_{q}) \\
& & \mbox{ if } F = B(F^{1}_{1} \hdots F^{1}_{p} \otimes F^{2}_{1} \hdots F^{2}_{q}) \mbox{ with the } F^{1}_{i} \mbox{'s and the } F^{2}_{j}\mbox{'s} \in \mathbb{T}. 
\end{array}\right.$$
This map is linearly extended into a map $\phi: \mathcal{P} \rightarrow A$. Let us show that it is a morphism of $\mathcal{L}$-algebras, that is to say $\phi(F \succ G)=\phi(F) \succ \phi(G)$ and $ \phi (F \prec G) = \phi(F) \prec \phi(G) $  for all $ F , G \in \mathbb{T} \setminus \{1\} $. Note $ F = B(F^{1}_{1} \hdots F^{1}_{p} \otimes F^{2}_{1} \hdots F^{2}_{q}) $, $ G = B(G^{1}_{1} \hdots G^{1}_{r} \otimes G^{2}_{1} \hdots G^{2}_{s}) $ with $ F^{1}_{1} , \hdots , F^{1}_{p}, F^{2}_{1} , \hdots , F^{2}_{q} $ and $ G^{1}_{1} , \hdots , G^{1}_{r} ,G^{2}_{1} , \hdots , G^{2}_{s} $ in $ \mathbb{T} $. Then:
\begin{enumerate}
\item For $\phi(F \succ G)=\phi(F) \succ \phi(G)$,
\begin{eqnarray*}
\phi(F \succ G) & = & \phi(B(FG^{1}_{1} \hdots G^{1}_{r} \otimes G^{2}_{1} \hdots G^{2}_{s})) \\
& = & ( \hdots ((\phi(F) \succ (\phi(G^{1}_{1}) \succ ( \hdots (\phi(G^{1}_{r}) \succ a ) \hdots ))) \prec \phi(G^{2}_{1}) ) \hdots ) \prec \phi(G^{2}_{s})\\
& = & \phi(F) \succ ( ( \hdots ((\phi(G^{1}_{1}) \succ ( \hdots (\phi(G^{1}_{r}) \succ a ) \hdots )) \prec \phi(G^{2}_{1}) ) \hdots ) \prec \phi(G^{2}_{s})) \\
& = & \phi(F) \succ \phi(G).
\end{eqnarray*}
\item For $ \phi (F \prec G) = \phi(F) \prec \phi(G) $,
\begin{eqnarray*}
\phi(F \prec G) & = & \phi(B(F^{1}_{1} \hdots F^{1}_{p} \otimes F^{2}_{1} \hdots F^{2}_{q} G)) \\
& = & (( \hdots ((\phi(F^{1}_{1}) \succ ( \hdots (\phi(F^{1}_{p}) \succ a ) \hdots )) \prec \phi(F^{2}_{1}) ) \hdots ) \prec \phi(F^{2}_{q})) \prec \phi(G) \\
& = & \phi(F) \prec \phi(G).
\end{eqnarray*}
\end{enumerate}
So $\phi$ is a morphism of $\mathcal{L}$-algebras. \\

Let $\phi':\mathcal{P} \rightarrow A$ be another morphism of $ \mathcal{L} $-algebras such that $\phi'(\tun) = a $. For any forest $ F^{1}_{1} \hdots F^{1}_{p} $ and $ F^{2}_{1} \hdots F^{2}_{q} \in \mathbb{F} $,
\begin{eqnarray*}
& & \phi'(B(F^{1}_{1} \hdots F^{1}_{p} \otimes F^{2}_{1} \hdots F^{2}_{q})) \\
& = & \phi'( ( \hdots ((F^{1}_{1} \succ ( \hdots (F^{1}_{p-1} \succ (F^{1}_{p} \succ \tun )) \hdots )) \prec F^{2}_{1} ) \hdots ) \prec F^{2}_{q} ) \\
& = & ( \hdots ((\phi'(F^{1}_{1}) \succ ( \hdots (\phi'(F^{1}_{p-1}) \succ (\phi'(F^{1}_{p}) \succ \phi'(\tun) )) \hdots )) \prec \phi'(F^{2}_{1}) ) \hdots ) \prec \phi'(F^{2}_{q}) \\
& = & ( \hdots ((\phi'(F^{1}_{1}) \succ ( \hdots (\phi'(F^{1}_{p-1}) \succ (\phi'(F^{1}_{p}) \succ a )) \hdots )) \prec \phi'(F^{2}_{1}) ) \hdots ) \prec \phi'(F^{2}_{q}).
\end{eqnarray*}
So $\phi=\phi'$.
\end{proof}
\\

{\bf Remark.} {We deduce from theorem \ref{Llibre} that $ dim \left( \widetilde{\mathcal{L}}(n)\right) = t_{n}^{\mathcal{BT}} $ for all $ n \in \mathbb{N}^{\ast} $ (with the same reasoning as in corollary \ref{isom}). We find again the result already given in \cite{Leroux}.}
\\

We denote by $ L(V) $ the free $ \mathcal{L} $-algebra over a $ \mathbb{K} $-vector space $ V $. The functor $ L(-) $ is the left adjoint to the forgetful functor from $ \mathcal{L} $-algebras to vector spaces. Because the operad $ \mathcal{L} $ is regular, we get the following result :

\begin{prop}
Let $ V $ be a $ \mathbb{K} $-vector space. Then the free $ \mathcal{L} $-algebra on $ V $ is
$$ L(V) = \bigoplus_{n \geq 1} \mathbb{K} [\mathbb{T}(n)] \otimes V^{\otimes n} ,$$
equipped with the following binary operations : for all $ F \in \mathbb{T}(n) $, $ G \in \mathbb{T}(m) $, $ v_{1} \otimes \hdots \otimes v_{n} \in V^{\otimes n} $ and $ w_{1} \otimes \hdots \otimes w_{m} \in V^{\otimes m} $,
\begin{eqnarray*}
(F \otimes v_{1} \otimes \hdots \otimes v_{n}) \succ (G \otimes w_{1} \otimes \hdots \otimes w_{m}) & = & (F \succ G \otimes v_{1} \otimes \hdots \otimes v_{n} \otimes w_{1} \otimes \hdots \otimes w_{m}) ,\\
(F \otimes v_{1} \otimes \hdots \otimes v_{n}) \prec (G \otimes w_{1} \otimes \hdots \otimes w_{m}) & = & (F \prec G \otimes v_{1} \otimes \hdots \otimes v_{n} \otimes w_{1} \otimes \hdots \otimes w_{m}) .
\end{eqnarray*}
\end{prop}

\subsection{Universal enveloping bigraft algebra}

The functor $ (-)_{\mathcal{L}} : \{ \mathcal{BG} - {\rm alg} \} \rightarrow \{ \mathcal{L} - {\rm alg} \} $ associates to a $ \mathcal{BG} $-algebra $ (A,m, \succ , \prec ) $ the $ \mathcal{L} $-algebra $ (A,\succ , \prec ) $.\\

Reciprocally, we define the universal enveloping bigraft algebra of a $ \mathcal{L} $-algebra $ (A, \succ , \prec ) $, denoted $ U_{\mathcal{BG}} (A) $, as the augmentation ideal $ \overline{T}(A) $ of tensor algebra $ T(A) $ over the $ \mathbb{K} $-vector space $ A $ equipped with two operations also denoted by $ \succ $ and $ \prec $ and defined by : for all $ p,q \in \mathbb{N}^{\ast} $ and $ a_{i} , b_{j} \in A $, $ 1 \leq i \leq p $ and $ 1 \leq j \leq q $,
\begin{eqnarray} \label{structbg}
\begin{array}{rcl}
(a_{1} \hdots a_{p}) \succ (b_{1} \hdots b_{q}) & := & (a_{1} \succ ( \hdots (a_{p-1} \succ (a_{p} \succ b_{1})) \hdots )) b_{2} \hdots b_{q} , \\
(a_{1} \hdots a_{p}) \prec (b_{1} \hdots b_{q}) & := & a_{1} \hdots a_{p-1} (( \hdots ((a_{p} \prec b_{1}) \prec b_{2}) \hdots ) \prec b_{q}) .
\end{array}
\end{eqnarray}
Then we have immediately that $ (\overline{T}(A), m ,\succ , \prec ) $ is a nonunitary bigraft algebra, where $ m $ is the concatenation.

\begin{lemma} \label{lemmeutile}
The functor $ U_{\mathcal{BG}} : \{ \mathcal{L} - {\rm alg} \} \rightarrow \{ \mathcal{BG} - {\rm alg} \} $ is the left adjoint to $ (-)_{\mathcal{L}} : \{ \mathcal{BG} - {\rm alg} \} \rightarrow \{ \mathcal{L} - {\rm alg} \} $.
\end{lemma}

\begin{proof}
We put $ A $ a $ \mathcal{L} $-algebra and $ B $ a $ \mathcal{BG} $-algebra. Let $ f : A \rightarrow (B)_{\mathcal{L}} $ be a morphism of $ \mathcal{L} $-algebras. It determines uniquely a morphism of algebras $ \tilde{f} : \overline{T}(A) \rightarrow B $ because $ (B,m) $ is an associative algebra. We endow $ \overline{T}(A) $ with a $ \mathcal{BG} $-algebra structure defined by (\ref{structbg}). Then $ \tilde{f} : U_{\mathcal{BG}}(A) \rightarrow B $ is a morphism of $ \mathcal{BG} $-algebras : for all $ p,q \in \mathbb{N}^{\ast} $ and $ a_{i} , b_{j} \in A $, $ 1 \leq i \leq p $ and $ 1 \leq j \leq q $,
\begin{eqnarray*}
\tilde{f} ((a_{1} \hdots a_{p}) \succ (b_{1} \hdots b_{q})) & = & \tilde{f} ((a_{1} \succ ( \hdots (a_{p-1} \succ (a_{p} \succ b_{1})) \hdots )) b_{2} \hdots b_{q}) \\
& = & (\tilde{f}(a_{1}) \succ ( \hdots (\tilde{f}(a_{p-1}) \succ (\tilde{f}(a_{p}) \succ \tilde{f}(b_{1}))) \hdots )) \tilde{f} (b_{2}) \hdots \tilde{f}(b_{q}) \\
& = & (\tilde{f}(a_{1}) \hdots \tilde{f}(a_{p})) \succ (\tilde{f}(b_{1}) \hdots \tilde{f}(b_{q})) \\
& = & \tilde{f}(a_{1} \hdots a_{p}) \succ \tilde{f}(b_{1} \hdots b_{q}) ,\\
\tilde{f} ((a_{1} \hdots a_{p}) \prec (b_{1} \hdots b_{q})) & = & \tilde{f} (a_{1} \hdots a_{p-1} (( \hdots ((a_{p} \prec b_{1}) \prec b_{2}) \hdots ) \prec b_{q})) \\
& = & \tilde{f} (a_{1}) \hdots \tilde{f}(a_{p-1}) (( \hdots ((\tilde{f}(a_{p}) \prec \tilde{f}(b_{1})) \prec \tilde{f}(b_{2})) \hdots ) \prec \tilde{f}(b_{q})) \\
& = & (\tilde{f}(a_{1}) \hdots \tilde{f}(a_{p})) \prec (\tilde{f}(b_{1}) \hdots \tilde{f}(b_{q})) \\
& = & \tilde{f}(a_{1} \hdots a_{p}) \prec \tilde{f}(b_{1} \hdots b_{q}) .
\end{eqnarray*}
On the other hand, let $ g : U_{\mathcal{BG}}(A) \rightarrow B $ be a morphism of $ \mathcal{BG} $-algebras. From the construction of $ U_{\mathcal{BG}}(A) $ it follows that the map $ A \rightarrow U_{\mathcal{BG}}(A) $ is a $ \mathcal{L} $-algebra morphism. Hence the composition $ \tilde{g} $ with $ g $ gives a $ \mathcal{L} $-algebra morphism $ A \rightarrow B $.\\

These two constructions are inverse of each other, and therefore $ U_{\mathcal{BG}} $ is the left adjoint to $ (-)_{\mathcal{L}} $.
\end{proof}

\begin{cor}
The universal enveloping bigraft algebra of the free $ \mathcal{L} $-algebra is canonically isomorphic to the free bigraft algebra :
$$ U_{\mathcal{BG}}(L(V)) \cong BG(V) .$$
\end{cor}

\begin{proof}
$ U_{\mathcal{BG}} $ is left adjoint to $ (-)_{\mathcal{L}} $ and $ L $ is left adjoint to the forgetful functor. The composite is the left adjoint to the forgetful functor from $ \mathcal{BG} $-algebras to vector spaces. Hence it is the functor $ BG $.
\end{proof}

\begin{theo}
For any $ \mathcal{BG} $-infinitesimal bialgebra $ A $ over a field $ \mathbb{K} $, the following are equivalent:
\begin{enumerate}
\item $ A $ is a connected $ \mathcal{BG} $-infinitesimal bialgebra,
\item $ A $ is cofree among the connected coalgebras,
\item $ A $ is isomorphic to $ U_{\mathcal{BG}} (Prim(A)) $ as a $ \mathcal{BG} $-infinitesimal bialgebra.
\end{enumerate}
\end{theo}

\begin{proof}
We prove the following implications $ 1. \Rightarrow 2. \Rightarrow 3. \Rightarrow 1. $ \\
$ 1. \Rightarrow 2. $ If $ A $ is a connected $ \mathcal{BG} $-infinitesimal bialgebra then $ A $ is isomorphic to $ (\overline{T}(Prim(A)), m , \Delta ) $ as an infinitesimal bialgebra, where $ \overline{T}(Prim(A)) $ is the augmentation ideal of the tensor algebra over $ Prim(A) $, $ m $ is the concatenation and $ \Delta $ is the deconcatenation (see \cite{Loday} for a proof). Therefore $ A $ is cofree.\\
$ 2. \Rightarrow 3. $ If $ A $ is cofree, then it is isomorphic as an infinitesimal bialgebra to $ (\overline{T}(Prim(A)), m , \Delta ) $ and $ Prim(A) $ is a $ \mathcal{L} $-algebra with proposition \ref{proputile}. $ \overline{T}(Prim(A)) $ is a $ \mathcal{BG} $-algebra with the two operations $ \succ $ and $ \prec $ defined as in (\ref{structbg}) and this is exactly $ U_{\mathcal{BG}}(Prim(A)) $. So $ A $ is isomorphic as an infinitesimal bialgebra to $ U_{\mathcal{BG}}(Prim(A)) $ and it is a $ \mathcal{BG} $-morphism by using (\ref{definitionalgbig}) and (\ref{structbg}).\\
$ 3. \Rightarrow 1. $ By contruction, $ U_{\mathcal{BG}} (Prim(A)) $ is isomorphic to $ \overline{T}(Prim(A)) $ as a bialgebra. Therefore $ A $ is isomorphic to $ \overline{T}(Prim(A)) $ as a bialgebra. As $ \overline{T}(Prim(A)) $ is connected, $ A $ is connected.
\end{proof}
\\

In other terms :

\begin{theo}
The triple of operads $ (\mathcal{A}ss,\mathcal{BG},\mathcal{L}) $ is a good triple of operads.
\end{theo}

{\bf Remark.} {Note that if $ A $ is a $ \mathcal{BG} $-infinitesimal bialgebra, then $ (A,m,\tdelta_{\mathcal{A}ss}) $ is a nonunitary infinitesimal bialgebra. Hence, if $ (\mathbb{K} \oplus A,m,\Delta_{\mathcal{A}ss}) $ has an antipode $ S $, then $ -S $ is an eulerian idempotent for $ A $ and we have :
\begin{eqnarray*}
S(a) = \left\lbrace \begin{array}{l}
-a \mbox{ if } a \in Prim(A),\\
0 \mbox{ if } a \in A^{2}.
\end{array}
\right. 
\end{eqnarray*}}